\documentclass[12pt]{article}

\usepackage[T1]{fontenc}
\usepackage[english]{babel}
\usepackage{amssymb}
\usepackage{amsmath}
\usepackage{epsfig}
\usepackage{graphicx}
\usepackage{color}
\usepackage{hyperref}
\usepackage{url}
\usepackage{color}
\usepackage[a4paper]{geometry}
\usepackage{bbm}

\allowdisplaybreaks

\def\disp{\displaystyle}

\newcommand\R{\mathbb{R}}

\newcommand{\p}{\partial}
\newcommand\1{{\bf 1}}

\def\kxi{\begin{pmatrix}1\\ \xi \\ \gamma\end{pmatrix}}
\def\F{\mathcal F}

\newtheorem{proposition}{Proposition}[section]
\newtheorem{theorem}{Theorem}[section]
\newtheorem{remark}[theorem]{Remark}
\newtheorem{lemma}[theorem]{Lemma}

\newenvironment{proof}[1][Proof]{\begin{trivlist}
\item[\hskip \labelsep {\bfseries #1}] \it }{$\blacksquare$\end{trivlist}}

\definecolor{red}{rgb}{1,0,0}

\title{Numerical approximation of the 3d hydrostatic Navier-Stokes system with free surface.}

\author{S. Allgeyer\footnote{Research
  School of Earth Sciences, Australian National University, Canberra,
  ACT, Australia}, M.-O.~Bristeau\footnote{Inria Paris, 2 rue Simone Iff, CS 42112, 75589
  Paris Cedex 12, France}$\;^{,}$\footnote{Sorbonne Universit\'e, Lab. Jacques-Louis Lions, 4 Place Jussieu, F-75252 Paris cedex 05}$\;^{,}$\footnote{CEREMA, 134 rue de Beauvais, F-60280
Margny-L\`es-Compi\`egne, France}$\;$,
D.~Froger\footnotemark[2]$\;^{,}$\footnotemark[3]$\;^{,}$\footnotemark[4]$\;$,
R.~Hamouda\footnotemark[2]$\;^{,}$\footnotemark[3]$\;^{,}$\footnotemark[4]$\;$,
A.~Mangeney\footnote{Univ. Paris Diderot, Sorbonne Paris Cit\'e, Institut de Physique du Globe de Paris, Seismology Group, 1 rue Jussieu, Paris F-75005, France}$\;^{,}$\footnotemark[2]$\;^{,}$\footnotemark[3]$\;^{,}$\footnotemark[4]$\;$,\\
  J.~Sainte-Marie\footnotemark[2]$\;^{,}$\footnotemark[3]$\;^{,}$\footnotemark[4]$\;^{,}$\footnote{Corresponding author: Jacques.Sainte-Marie@inria.fr}, F.~Souill\'e\footnotemark[2]$\;^{,}$\footnotemark[3]$\;^{,}$\footnotemark[4]$\;$ and~M. Vall\'ee\footnotemark[5]}

\date{\today}

\graphicspath{{Figures/}}

\begin{document}

\maketitle


\begin{abstract}
In this paper we propose a stable and robust strategy to approximate
the 3d incompressible hydrostatic Euler
and Navier-Stokes systems with free surface.

Compared to shallow water approximation of the Navier-Stokes system, the idea is to use a Galerkin type approximation of the velocity field
with piecewise constant basis functions in
order to obtain an accurate
description of the vertical profile of the horizontal velocity. Such a
strategy has several advantages. It allows
\begin{itemize}
\item[$\circ$] to rewrite the Navier-Stokes equations under the form
of a system of conservation laws with source terms,
\item[$\circ$] the easy handling of the free surface, which does not
  require moving meshes,
\item[$\circ$] the possibility to take advantage of robust and
  accurate numerical techniques developed in extensive amount for
  Shallow Water type systems.
\end{itemize}
Compared to previous works of some of the
authors, the three dimensional case is studied in this paper.
We show that the model admits a kinetic interpretation including the
vertical exchanges terms, and we use
this result to formulate a robust finite volume scheme for its
numerical approximation. 
All the aspects of the discrete scheme (fluxes, boundary
conditions,\ldots) are completely described and the stability
properties of the proposed numerical scheme (well-balancing,
positivity of the water depth,\ldots) are discussed.
We validate the model and the discrete scheme with some numerical
academic examples (3d non stationary analytical solutions) and illustrate the capability
of the discrete model to reproduce realistic tsunami waves.
\end{abstract}


\section{Introduction}

In this paper we present layer-averaged Euler and Navier-Stokes models
for the numerical simulation of incompressible free surface
flows over variable topographies.
We are mainly interested in applications to geophysical water flows such
as tsunamis, lakes, rivers, estuarine waters, hazardous flows in the context
either of advection dominant
flows or of wave propagation.

The simulation of these flows requires stable, accurate,
conservative schemes able to sharply resolve
stratified flows, to
handle efficiently complex topographies and free surface deformations,
and to capture robustly wet/dry fronts. In addition, the application to
realistic three-dimensional problems demands efficient methods with
respect to computational cost.
The present work is aimed at building a simulation tool endowed with
these properties.

Due to computational issues associated with the free surface Navier-Stokes or
  Euler equations, the simulations of geophysical flows are often carried
  out with shallow water type
  models of reduced complexity. Indeed, for vertically
  averaged models such as the Saint-Venant system \cite{saint-venant},
  efficient and robust numerical techniques
  (relaxation schemes \cite{bouchut}, kinetic schemes
  \cite{perthame,JSM_entro}, \ldots) are available and avoid to deal with moving meshes.
  In order to describe and simulate complex flows where the velocity
  field cannot be approximated by its vertical mean, multilayer models
  have been developed~\cite{audusse,bristeau3,bristeau2,bouchut1,pares,pares1,ovsyannikov,vreug,multi-bz}.
  Unfortunately these models are physically relevant for non miscible fluids.
  In~\cite{nieto_chacon,JSM_M2AN,JSM_JCP,JSM_M3AS,BDGSM,nieto_diaz_mangeney_reina}, some authors have proposed
  a simpler and more general formulation for multilayer model with mass
  exchanges between the layers. The obtained model has the form of a
  conservation law with source terms and presents remarkable differences with respect
 to classical models for non miscible fluids.
 In the multilayer approach with mass exchanges, the layer partition is merely a discretization artefact, and it
 is not physical. Therefore, the internal layer boundaries do not necessarily
 correspond to isopycnic surfaces.
 A critical distinguishing feature of our model is that
 it allows fluid circulation between layers. This changes dramatically
 the properties of the model and its ability to
 describe flow configurations that
 are crucial for the
 foreseen applications, such as recirculation zones.

Compared to previous works of some of the
authors~\cite{JSM_M2AN,JSM_JCP,BDGSM}, that handled only the 2d configurations, this paper deals with
the 3d case on unstructured meshes reinforcing the need of efficient
numerical schemes. The key points of this paper are the following
\begin{itemize}
\item A formulation of the 3d Navier-Stokes system under the form of a
  set of conservation laws with source terms on a fixed 2d domain.
\item A kinetic interpretation of the model allowing to derive a
  robust and accurate numerical scheme. Notice that the kinetic
  interpretation is valid for the vertical exchange terms arising in
  the multilayer description.
\item Choosing a Newtonian rheology for the fluid, we propose
  energy-consistent~-- at the continuous and discrete levels~-- models
  extending previous results~\cite{BDGSM} in the 3d context.
\item We give a
 complete description of all the ingredients of the numerical scheme (time scheme, fluxes,
 boundary conditions,\ldots). Even if some parts have been already published in 2d, the objective is to have a self-contained paper for 3d applications.
\item The numerical approximation of the 3d Navier-Stokes system is
  endowed with strong stability properties (consistency, well-balancing, positivity of the water depth, wet/dry interfaces treatment,\ldots).
\item Using academic examples, we prove the accuracy of the proposed
  numerical procedure especially convergence curves towards a 3d non-stationary analytical solution with wet-dry interfaces have been obtained (see paragraph~\ref{subsubsec:bowl}).
\end{itemize}

Most of the numerical models in the literature for environmental
stratified flows use finite difference or finite element schemes
solving the free surface Navier-Stokes equations. We refer in particular
 to \cite{casulli,hervo} and references therein for a
 partial review of these methods. Since the layer-averaged model has
 the form of a conservation law with source terms, we single out a
 finite volume scheme. Moreover, the
 kinetic interpretation of the continuous model leads to a kinetic
 solver endowed with strong stability properties (well-balancing,
 domain invariant, discrete entropy~\cite{ns_entro}). The viscous terms are discretized using a finite element approach. Considering various analytical
 solutions we emphasize the accuracy of the discrete model and we also
 show the applicability of the model to real geophysical situations. The numerical method is implemented in Freshkiss3d~\cite{freshkiss3d} and other various academic tests are documented on the web site.

The outline of the paper is as follows. In Section~\ref{sec:NS}, we
recall the incompressible and hydrostatic Navier-Stokes
equations and the associated boundary conditions.
The layer-averaged system obtained by a vertical
discretization of the hydrostatic model is described in Section~\ref{sec:ML}. The kinetic interpretation of the model is given in
Section~\ref{sec:kin} allowing to derive a numerical scheme presented in Section~\ref{sec:NumMet}.
 Numerical validations and application to a real tsunami event are shown in Section~\ref{sec:NumRes}.

\section{The hydrostatic Navier-Stokes system}
\label{sec:NS}

We consider the three-dimensional hydrostatic Navier-Stokes system \cite{lions}
  describing a free surface gravitational flow moving over a bottom
  topography $z_b(x,y)$. For free surface flows, the hydrostatic assumption consists in
  neglecting the vertical acceleration, see \cite{brenier,grenier,masmoudi,bresch_lemoine} for justifications of such hydrostatic models.

The incompressible and hydrostatic Navier-Stokes system consists in the model
\begin{eqnarray}
& & \nabla. {\bf U} = 0, \label{eq:incomp}\\
& & \frac{\p {\bf u}}{  \p t} + \nabla_{x,y}. ( {\bf u}\otimes {\bf u})+
\frac{\partial {\bf u}w}{\partial z} =  \frac{1}{\rho_0} \nabla_{x,y} . \sigma + \frac{\mu}{\rho_0} \frac{\partial^2 {\bf u}}{\partial z^2},
\label{eq:mvthm}\\
& & \frac{\partial p}{\partial z} = - \rho_0 g,\label{eq:mvthm1}
\end{eqnarray}
where
${\bf U}(t,x,y,z)=(u,v,w)^T$ is the velocity,
${\bf u}(t,x,y,z)=(u,v)^T$ is the horizontal velocity,
${\bf \sigma} = -p I_d + \mu \nabla_{x,y} {\bf u}= -p I_d + \Sigma$ is the total stress tensor, $p$ is the fluid pressure,
$g$ represents the gravity acceleration and $\rho_0$ is the fluid density. The quantity $\nabla$ denotes $\nabla = \begin{pmatrix} \frac{\partial }{\partial x},
  \frac{\partial}{\partial y},
  \frac{\partial}{\partial z}\end{pmatrix}^T$,
$\nabla_{x,y}$ corresponds to the projection of $\nabla$ on the
horizontal plane i.e. $\nabla_{x,y} = \begin{pmatrix} \frac{\partial }{\partial x},
  \frac{\partial}{\partial y} \end{pmatrix}^T$. We assume a Newtonian fluid, $\mu$ is the viscosity coefficient and we will make use of $\nu=\mu/\rho_0$.

We consider a free surface flow (see Fig. \ref{fig:free}-{\it (a)}), therefore we assume
$$z_b(x,y)\leq z\leq \eta(t,x,y) := h(t,x,y)+z_b(x,y),$$
 with $z_b(x,y)$ the bottom elevation and $ h(t,x,y)$ the water
 depth. Due to the hydrostatic assumption in Eq.~\eqref{eq:mvthm1},
the pressure gradient in Eq.~\eqref{eq:mvthm} reduces to
$\rho_0 g\nabla_{x,y} \eta$.

\subsection{Boundary conditions}

\subsubsection{Bottom and free surface}

Let ${\bf n}_b$ and ${\bf n}_s$ be the unit outward
 normals at the bottom and at the free surface respectively defined by
 (see Fig~\ref{fig:free}-{\it (a)})
$${\bf n}_b = \frac{1}{\sqrt{1 + |\nabla_{x,y} z_b|^2}}
  \left(\begin{array}{c} \nabla_{x,y} z_b\\ -1 \end{array}
  \right),\quad\mbox{and}\quad {\bf n}_s = \frac{1}{\sqrt{1 + |\nabla_{x,y} \eta|^2}}
  \left(\begin{array}{c} -\nabla_{x,y} \eta\\ 1 \end{array} \right).$$
The system~\eqref{eq:incomp}-\eqref{eq:mvthm1} is completed with 
boundary conditions.
On the bottom we prescribe an impermeability condition
\begin{equation}
{\bf U}.{\bf n}_b = 0,
\label{eq:bottom}
\end{equation}
whereas on the free surface, we impose
the kinematic boundary condition
\begin{equation}
\frac{\partial \eta}{\partial t} + {\bf u}(t,x,y,\eta).\nabla_{x,y} \eta - w(t,x,y,\eta) = 0.
\label{eq:free_surf}
\end{equation}

Concerning the dynamical boundary conditions, at the bottom we impose a friction condition given e.g. by a Navier law
\begin{equation}
\nu\sqrt{1+|\nabla_{x,y} z_b|^2}\frac{\partial {\bf u}}{\partial {\bf n}_b}  = -\kappa {\bf u},
\label{eq:uboundihmf}
\end{equation}
 with $\kappa$ a Navier coefficient.
For some applications, one can choose $\kappa=\kappa(h,{\bf u}|_b)$.

At the free surface, we impose the no stress condition
\begin{equation}
\nu \frac{\partial \tilde {\bf u}}{\partial {\bf n}_s} - p{\bf n}_s = - p^a(t,x,y) {\bf n}_s + W(t,x,y){\bf t}_s.
\label{eq:uboundihm}
\end{equation}
where $\tilde {\bf u} = ({\bf u},0)^T$, $p^a(t,x,y)$ and $W(t,x,y)$ are two given quantities, $p^a$ (resp. $W$) mimics the effects of
the atmospheric pressure (resp. the wind blowing at the free surface) and ${\bf t}_s$ is a given unit horizontal vector.
Throughout the paper $p^a=0$ except in
paragraph~\ref{subsec:sol_anal_NS} where the effects of the
atmospheric pressure is considered.

\subsubsection{Fluid boundaries and solid walls}
\label{subsubsec:fluid_boundaries}

On solid walls, we prescribe an impermeability condition
\begin{equation}
{\bf U}.{\bf n} = 0,
\label{eq:slip}
\end{equation}
coupled with an homogeneous Neumann condition
\begin{equation}
\frac{\partial {\bf u}}{\partial {\bf n}} = 0,
\label{eq:neumann}
\end{equation}
${\bf n}$ being the outward normal to the considered wall.

In this paper we consider
fluid boundaries on which we prescribe
zero, one or two of the following conditions depending on the
type of the flow (fluvial or torrential) : Water level $h+z_b(x,y)$ given,
flux $h{\bf U}$ given.

The system is completed with some initial conditions
$$h(0,x,y)=h^0(x,y), \quad {\bf U}(0,x,y,z)={\bf U}^0(x,y,z),$$
with ${\bf U}^0$ satisfying the divergence free
condition~\eqref{eq:incomp}.

\subsection{Energy balance}

The smooth solutions of the system~\eqref{eq:incomp}-\eqref{eq:uboundihm}
satisfy the energy balance
\begin{equation}
\frac{\partial}{\partial t}\int_{z_b}^{\eta} E\ dz  +
\nabla_{x,y} . \int_{z_b}^{\eta} \Bigl( {\bf u}\bigl( E + g(\eta-z)
  \bigr) -\nu \nabla_{x,y} \frac{|{\bf u}|^2}{2}\Bigr) dz
=  -\nu\int_{z_b}^\eta |\nabla_{x,y} {\bf u}|^2 dz -\kappa
\left. {\bf u}\right|_{b}^2.
\label{eq:energy_eq_mod_NS}
\end{equation}
\begin{equation}
E = E(z,{\bf u}) = \frac{|{\bf u}^2|}{2} + gz.
\label{eq:energy_exp}
\end{equation}

\subsection{The hydrostatic Euler system}

In the case of an inviscid fluid, the
system~\eqref{eq:incomp}-\eqref{eq:uboundihm} consists in
the incompressible and hydrostatic Euler equations with free surface and reads
\begin{eqnarray}
& & \nabla . {\bf U}  =0,\label{eq:div}\\
& & \frac{\p {\bf u}}{  \p t} + \nabla_{x,y}. ( {\bf u}\otimes {\bf u})+
\frac{\partial {\bf u}w}{\partial z} + g\nabla_{x,y} \eta = 0,\label{eq:u}
\end{eqnarray}
coupled with the two kinematic boundary
conditions~\eqref{eq:bottom},\eqref{eq:free_surf} and $p(t,x,y,\eta(t,x,y))= 0$.

We recall the fundamental stability property related to the fact that
the hydrostatic Euler system admits, for smooth solutions, an energy conservation that can be
written under the form
\begin{equation}
\frac{\partial}{\partial t}\int_{z_b}^{\eta} E\ dz  +
\nabla_{x,y} . \int_{z_b}^{\eta} \Bigl( {\bf u}\bigl( E + g(\eta-z)
  \bigr) \Bigr) dz\\
=  0,
\label{eq:energy_eq_mod}
\end{equation}
with $E$ defined by~\eqref{eq:energy_exp}.

\begin{figure}[hbtp]
\begin{center}
\begin{tabular}{cc}
\includegraphics[height=4.5cm]{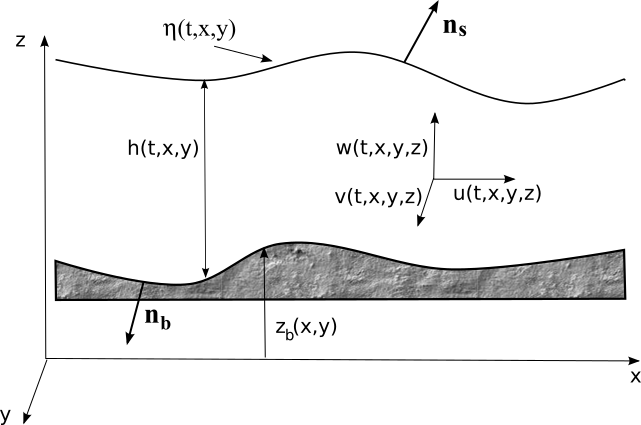}
&
\includegraphics[height=4cm]{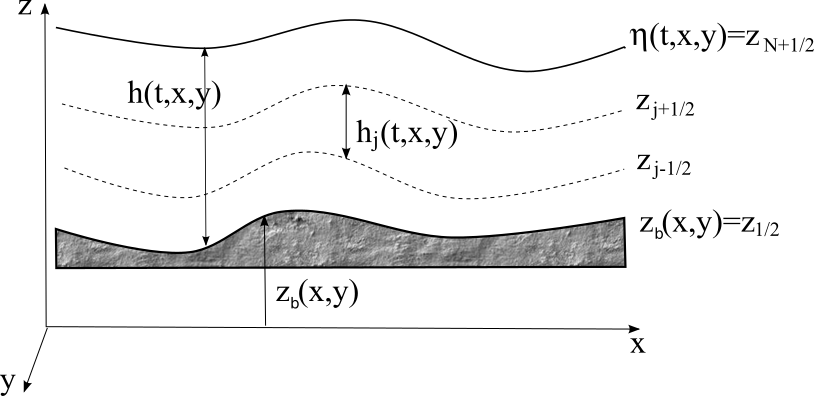}\\
{\it (a)} & {\it (b)}
\end{tabular}
\caption{Fluid domain, notations and layerwise discretization.}
\label{fig:free}
\end{center}
\end{figure}


\section{The layer-averaged model}
\label{sec:ML}

We consider a discretization of the fluid domain by layers (see Fig.~\ref{fig:free}-{\it (b)}) where the
layer $\alpha$ contains the points of coordinates $(x,y,z)$ with $z \in
L_\alpha(t,x,y)= [z_{\alpha-1/2},z_{\alpha+1/2}]$ and
$\{z_{\alpha+1/2}\}_{\alpha=1,\ldots,N}$
 is defined by
\begin{equation}
\left\{\begin{array}{l}
z_{\alpha+1/2}(t,x,y) = z_b(x,y) + \sum_{j=1}^\alpha h_j(t,x,y),\\
h_\alpha(t,x,y) = z_{\alpha+1/2}(t,x,y) - z_{\alpha-1/2}(t,x,y)=l_\alpha h(t,x,y),
\end{array}\right.
\label{eq:layer}
\end{equation}
for $\alpha\in \{0,\ldots,N\}$ and $\sum_{\alpha=1}^N l_\alpha =1$.

\subsection{The layer-averaged Euler system}
\label{subsec:layer_av_euler}

The layer-averaging process for the 2d hydrostatic Euler and
Navier-Stokes systems is precisely described in the paper~\cite{BDGSM}
with a general rheology,
the reader can refer to it. In the following, we present a Galerkin
type approximation of the 3d Euler system also leading to a layer-averaged
version of the Euler system, the obtained model reduces
to~\cite{BDGSM} in the 2d context.

Using the notations~\eqref{eq:layer}, let us consider the space
$\mathbb{P}_{0,h}^{N,t}$ of piecewise constant functions defined by
\begin{equation}
\mathbb{P}_{0,h}^{N,t} = \left\{ \1_{z \in L_\alpha(t,x,y)}(z),\quad
  \alpha\in\{1,\ldots,N\}\right\},
\label{eq:P0_space}
\end{equation}
where $\1_{z \in L_\alpha(t,x,y)}(z)$ is the characteristic
function of the layer $L_\alpha(t,x,y)$. Using this formalism, the
projection of $u$, $v$ and $w$ on $\mathbb{P}_{0,h}^{N,t}$ is a
piecewise constant function defined by
\begin{equation}
X^{N}(t,x,y,z,\{z_\alpha\})  = \sum_{\alpha=1}^N
\1_{[z_{\alpha-1/2},z_{\alpha+1/2}]}(z)X_\alpha(t,x,y),
\label{eq:ulayer}
\end{equation}
for $X \in (u,v,w)$.

The three following propositions hold.

\begin{proposition}
Using the space $\mathbb{P}_{0,h}^{N,t}$ defined by~\eqref{eq:P0_space} and the
decomposition~\eqref{eq:ulayer}, the Galerkin approximation of the incompressible and hydrostatic Euler equations~\eqref{eq:div}-\eqref{eq:u},\eqref{eq:bottom},\eqref{eq:free_surf}
leads to the system
\begin{eqnarray}
&&\sum_{\alpha=1}^N \frac{\partial  h_\alpha  }{\partial t } + \sum_{\alpha=1}^N  \nabla_{x,y} .   (h_\alpha  {\bf u} _\alpha)  =0,\label{eq:massesvml} \\
&&\frac{\p  h_{\alpha}{\bf u}_{\alpha}}{\p t} +  \nabla_{x,y} .\left(
  h_{\alpha} {\bf u}_{\alpha} \otimes {\bf u}_{\alpha} \right) +
\nabla_{x,y} \bigl(\frac{g}{2} h h_{\alpha}\bigr) = -gh_\alpha \nabla_{x,y} z_b \nonumber \\
& & \hspace*{2cm} + {\bf u}_{\alpha+1/2}
G_{\alpha+1/2} - {\bf u}_{\alpha-1/2} G_{\alpha-1/2} ,\quad {\alpha}=1,...,N. \label{eq:mvtsvml}
\end{eqnarray}
The quantity $G_{\alpha+1/2}$ (resp. $G_{\alpha-1/2}$) corresponds to mass
exchange accross the interface $z_{\alpha+1/2}$ (resp. $z_{\alpha-1/2}$)
and $G_{\alpha+1/2}$ is defined by
\begin{eqnarray}
G_{\alpha+1/2}
& = & \sum_{j=1}^\alpha \left(\frac{\partial h_j}{\partial t }
+  \nabla_{x,y} . (h_j {\bf u}_j) \right)
 =  -\sum_{j=1}^N \Bigl(\sum_{p=1}^\alpha l_p - \1_{j\leq
      \alpha}\Bigr)\nabla_{x,y} . (h_j {\bf u}_j),
\label{eq:Qalphabis}
\end{eqnarray}
for $\alpha=1,\ldots,N$. The velocities at the interfaces ${\bf u}_{\alpha+1/2}$ are defined by
\begin{equation}
{\bf u}_{\alpha+1/2} =
\left\{\begin{array}{ll}
{\bf u}_\alpha & \mbox{\rm if } \;G_{\alpha+1/2} \leq 0\\
{\bf u}_{\alpha+1} & \mbox{\rm if } \;G_{\alpha+1/2} > 0
\end{array}\right.
\label{eq:upwind_uT}
\end{equation}
\label{thm:model_ml}
\end{proposition}
The smooth
solutions of~\eqref{eq:massesvml},\eqref{eq:mvtsvml} satisfy an energy
balance and we have the following proposition.
\begin{proposition}
The system~\eqref{eq:massesvml},\eqref{eq:mvtsvml} admits, for smooth solutions, the
energy balance
\begin{eqnarray}
& & \hspace*{-0.5cm}\frac{\partial }{\partial t} E_{\alpha} +
    \nabla_{x,y} .
\left( {\bf u}_\alpha  \left(E_{\alpha} + \frac{g}{2} h_\alpha h\right)\right)
 \nonumber\\
& & \qquad=
\left( {\bf u}_{\alpha+1/2} . {\bf u}_{\alpha}
  -\frac{|{\bf u}_{\alpha}|^2-|{\bf u}_{\alpha+1}|^2}{2} \right) G_{\alpha+1/2}
-\left( {\bf u}_{\alpha-1/2} {\bf u}_\alpha - \frac{|{\bf u}_{\alpha}|^2
    - |{\bf u}_{\alpha-1}|^2}{2} \right) G_{\alpha-1/2} ,
\label{eq:energy_mcl}
\end{eqnarray}
with
\begin{equation}
E_{\alpha}=\frac{h_\alpha |{\bf u}_\alpha|^2}{2}+\frac{g}{2}h_\alpha h + gh_\alpha z_b.
\label{eq:energ_al}
\end{equation}
The sum of Eqs.~\eqref{eq:energy_mcl} for $\alpha=1,\ldots,N$ gives
the energy balance
\begin{equation}
\frac{\partial }{\partial t} \sum_{\alpha=1}^N E_{\alpha} +
\sum_{\alpha=1}^N \nabla_{x,y} . 
{\bf u}_\alpha\left(E_{\alpha} +
  \frac{g}{2} h_\alpha h \right) 
= -\sum_{\alpha=1}^{N-1} \frac{|{\bf u}_{\alpha+1/2}
- {\bf u}_\alpha|^2}{2} |G_{\alpha+1/2}|.
\label{eq:energy_glol}
\end{equation}
\label{prop:energy_bal}
\end{proposition}

\begin{remark}
Equation~\eqref{eq:energy_glol} is a layer discretization of the energy
balance~\eqref{eq:energy_exp}. The definition of ${\bf u}_{\alpha+1/2}$ given
in~\eqref{eq:upwind_uT} ensures the right hand side
in~Eq.~\eqref{eq:energy_glol} is nonpositive. Notice that a
centered definition for ${\bf u}_{\alpha+1/2}$ i.e.
\begin{equation}
{\bf u}_{\alpha+1/2} = \frac{{\bf u}_{\alpha}+{\bf
    u}_{\alpha+1}}{2},
\label{eq:centered_uT}
\end{equation}
instead of \eqref{eq:upwind_uT} leads to a vanishing right hand side
in~Eq.~\eqref{eq:energy_glol}. But the centered
choice~\eqref{eq:centered_uT} does not allow to obtain an energy balance in the variable density
case and does not give a maximum principle, at the discrete level,
see~\cite{JSM_JCP}. Simple calculations show that any other choice
than~\eqref{eq:upwind_uT} or~\eqref{eq:centered_uT} leads to a non
negative r.h.s. in~\eqref{eq:energy_glol}.
\end{remark}

It is noticeable that, thanks to the kinematic boundary condition
at each interface, the vertical velocity is no more a variable of
the system (\ref{eq:mvtsvml}). This is an advantage of this formulation
over the hydrostatic model where the vertical velocity is needed in
the momentum equation (\ref{eq:mvthm}) and is deduced
from the incompressibility condition (\ref{eq:incomp}). Even if the vertical velocity $w$ no more appears in the
model~\eqref{eq:massesvml}-\eqref{eq:mvtsvml}, it can be obtained as
follows.

\begin{proposition}
The piecewise constant approximation of the vertical velocity $w$ satifying~Eq.~\eqref{eq:ulayer} is given by
\begin{equation}
w_{\alpha} = k_\alpha - z_\alpha\nabla_{x,y} . {\bf
  u}_\alpha
\label{eq:def_w}
\end{equation}
with
\begin{eqnarray*}
k_1 & = & \nabla_{x,y} . (z_b {\bf u}_1),\qquad
k_{\alpha+1} =  k_\alpha + \nabla_{x,y} . \bigl( z_{\alpha+1/2}({\bf
  u}_{\alpha+1} - {\bf u}_\alpha) \bigr).
\end{eqnarray*}
The quantities $\{w_\alpha\}_{\alpha=1}^N$ are obtained only using a
post-processing of the variables governing the system~\eqref{eq:massesvml}-\eqref{eq:mvtsvml}.
\label{prop:def_w}
\end{proposition}

\begin{proof}[Proof of prop.~\ref{thm:model_ml}]
Considering the divergence free condition~\eqref{eq:div}, using the decomposition~\eqref{eq:ulayer} and the space of test
functions~\eqref{eq:P0_space}, we consider the quantity
$$\int_\R {\bf 1}_{z \in L_\alpha (t,x,y)} \nabla . {\bf U}^N dz =
0,$$
with ${\bf U}^N = (u^N,v^N,w^N)^T$. Simple computations give
$$
0 = \int_\R {\bf 1}_{z \in L_\alpha (t,x,y)} \nabla . {\bf U^{N}} dz  =
\frac{\partial h_\alpha}{\partial t } + \frac{\partial }{\partial x }
 \int_{z_{\alpha-1/2}}^{z_{\alpha+1/2}} u\ dz + \frac{\partial }{\partial y }
 \int_{z_{\alpha-1/2}}^{z_{\alpha+1/2}} v\ dz - G_{\alpha+1/2} +
 G_{\alpha-1/2},
$$
leading to
\begin{equation}
\frac{\partial  h_\alpha  }{\partial t } +
\nabla_{x,y} .   (h_\alpha  {\bf u} _\alpha) = G_{\alpha+1/2} -
G_{\alpha-1/2},
\label{eq:c0_mc}
\end{equation}
with $G_{\alpha\pm 1/2}$ defined by
$$
G_{\alpha+1/2}  =\frac{\partial z_{\alpha+1/2}}{\partial t} +
{\bf u}_{\alpha+1/2} . \nabla_{x,y} z_{\alpha+1/2} - w_{\alpha+1/2}.
$$
The sum for $\alpha=1,\ldots,N$ of the above relations gives Eq.~\eqref{eq:massesvml}
where the kinematic boundary
conditions~\eqref{eq:bottom},\eqref{eq:free_surf} corresponding to
\begin{equation}
G_{1/2} = G_{N+1/2} = 0,
\label{eq:G_boundary}
\end{equation}
have been used. Similarly, the sum for $j=1,\ldots,\alpha$ of the
relations~\eqref{eq:c0_mc} with~\eqref{eq:G_boundary} gives the expression~\eqref{eq:Qalphabis}
for $G_{\alpha+1/2}$.

Now we consider the Galerkin approximation of Eq.~\eqref{eq:u}
i.e. the quantity
$$\int_\R {\bf 1}_{z \in L_\alpha (t,x,y)} \left(\frac{\p {\bf u}^N}{
    \p t} + \nabla_{x,y}. ( {\bf u}^N \otimes {\bf u}^N)+
\frac{\partial {\bf u}^N w^N}{\partial z} + g\nabla_{x,y} \eta \right) dz =0,$$
leading, after simple computations, to Eq.~\eqref{eq:mvtsvml}.
\end{proof}

\begin{proof}[Proof of prop.~\ref{prop:energy_bal}]
In order to obtain~\eqref{eq:energy_mcl} we multiply~\eqref{eq:c0_mc} by $g(h+z_b)-|{\bf u}_\alpha|^2/2$ and~\eqref{eq:mvtsvml} by
${\bf u_\alpha}$ then we perform simple manipulations. More precisely,
the momentun equation along the $x$ axis multiplied by $u_\alpha$ gives
\begin{multline*}
\left( \frac{\partial }{\partial t } (h_\alpha u_\alpha) +
\frac{\partial }{\partial x }\left( h_\alpha u_\alpha^2 +
  \frac{g}{2}h h_\alpha\right) + \frac{\partial }{\partial y }\left(
  h_\alpha u_\alpha v_\alpha\right) \right) u_\alpha=\\ \biggl( - g h_\alpha \frac{\partial z_b}{\partial x}
+ u_{\alpha+1/2}G_{\alpha+1/2} -
u_{\alpha-1/2}G_{\alpha-1/2}
\biggr) u_\alpha.
\end{multline*}
Considering first the left hand side of the preceding equation
excluding the pressure terms, we denote
$$I_{u,\alpha}  =
\left( \frac{\partial }{\partial t } (h_\alpha u_\alpha) +
\frac{\partial }{\partial x }\left( h_\alpha u_\alpha^2\right) +
\frac{\partial }{\partial y }\left( h_\alpha u_\alpha v_\alpha\right)\right)
u_\alpha,$$
and using~\eqref{eq:c0_mc} we have
$$I_{u,\alpha}  = \frac{\partial }{\partial t } \left( \frac{h_\alpha u_\alpha^2}{2} \right) + \frac{\partial }{\partial x}
\left( u_\alpha \frac{h_\alpha u_\alpha^2}{2} \right) + \frac{\partial }{\partial y}
\left( v_\alpha \frac{h_\alpha u_\alpha^2}{2} \right)
+ \frac{u_\alpha^2}{2} \left( \frac{\partial  h_\alpha  }{\partial t } +
\nabla_{x,y} .   (h_\alpha  {\bf u} _\alpha)\right).$$
Now we consider the contribution of the pressure terms over the energy
balance i.e.
$$I_{p,u,\alpha} = \left( \frac{\partial}{\partial x } \left( \frac{g}{2}h
  h_\alpha \right) + g h_\alpha \frac{\partial z_b}{\partial x}\right) u_\alpha,$$
and it comes
\begin{eqnarray*}
I_{p,u,\alpha} & = & g h_\alpha \frac{\partial}{\partial x} (h+z_b)
u_\alpha
 =  g\frac{\partial}{\partial x} \left( h_\alpha (h+z_b) u_\alpha \right)
- g (h+z_b) \frac{\partial }{\partial x} ( h_\alpha u_\alpha)
\\
& = & \frac{\partial}{\partial x} \left( \left( \frac{g}{2} h_\alpha h +
  \frac{g}{2} h_\alpha (h+2z_b) \right) u_\alpha\right)
- g (h+z_b) \frac{\partial }{\partial x} ( h_\alpha u_\alpha).
\end{eqnarray*}
Performing the same manipulations over the momentum equation along $y$
and adding the terms $I_{u,\alpha}$,$I_{v,\alpha}$,$I_{p,u,\alpha}$,$I_{p,v,\alpha}$ and ~\eqref{eq:c0_mc} multiplied by $g(h+z_b)$ gives the result.
\end{proof}

\begin{proof}[Proof of prop.~\ref{prop:def_w}]
Using the boundary condition~\eqref{eq:bottom}, an integration from $z_b$ to $z$ of the divergence free
condition~\eqref{eq:incomp} easily gives
$$w = -\nabla_{x,y} . \int_{z_b}^z {\bf u}\ dz.$$
Replacing formally in the above equation ${\bf u}$ (resp. $w$) by
${\bf u}^N$ (resp. $w^N$) defined by~\eqref{eq:ulayer} and performing an integration over the
layer $L_1$ of the obtained relation yields
$$
h_1w_1 = -\int_{z_b}^{z_{3/2}} \nabla_{x,y} . \int_{z_b}^z {\bf
  u}_1\ dz dz_1
= h_1\nabla_{x,y}. (z_b{\bf u}_1) -\frac{z_{3/2}^2 - z_b^2}{2}\nabla_{x,y}. {\bf u}_1,
$$
i.e. $w_1 = \nabla_{x,y}. (z_b{\bf u}_1) - z_1 \nabla_{x,y}. {\bf
  u}_1,$
corresponding to~\eqref{eq:def_w} for $\alpha=1$. A similar
computation for the layers $L_2,\ldots,L_N$ proves the
result~\eqref{eq:def_w} for $\alpha=2,\ldots,N$. A more detailed version of this proof is given in~\cite{BDGSM}.
\end{proof}

\subsection{The layer-averaged Navier-Stokes system}
\label{subsec:NS}

In paragraph~\ref{subsec:layer_av_euler}, we have applied the
layer-averaging to the Euler system, we now use the same process for
the hydrostatic Navier-Stokes system. First, we consider the Navier-Stokes
system~\eqref{eq:incomp}-\eqref{eq:uboundihm} for a Newtonian fluid
and then with a simplified rheology.

\subsubsection{Complete model}

The layer-averaging process applied to the Navier-Stokes
system~\eqref{eq:incomp}-\eqref{eq:uboundihm} leads to the
following proposition.

\begin{proposition}
The layer-averaged hydrostatic Navier-Stokes system~\eqref{eq:incomp}-\eqref{eq:uboundihm} is given by
\begin{eqnarray}
&&\sum_{\alpha=1}^N \frac{\partial  h_\alpha  }{\partial t} + \sum_{\alpha=1}^N  \ \nabla_{x,y} .   (h_\alpha  {\bf u} _\alpha)  =0. \label{eq:nsml1} \\
&&\frac{\p  h_{\alpha}{\bf u}_{\alpha}}{\p t} + \ \nabla_{x,y} .\left(
  h_{\alpha} {\bf u}_{\alpha} \otimes {\bf u}_{\alpha} \right) +
\nabla_{x,y}  \bigl(\frac{g}{2} h h_{\alpha} \bigr) = -gh_\alpha \nabla_{x,y} z_b \nonumber\\
& & + {\bf u}_{\alpha+1/2} G_{\alpha+1/2} - {\bf u}_{\alpha-1/2}
G_{\alpha-1/2} + \nabla_{x,y} .
\bigl( h_\alpha {\bf \Sigma}_\alpha \bigr) \nonumber\\
& & - {\bf \Sigma}_{\alpha+1/2} \nabla_{x,y} z_{\alpha+1/2} + {\bf
    \Sigma}_{\alpha-1/2} \nabla_{x,y} z_{\alpha-1/2} \nonumber\\
&& +2\nu_{\alpha+1/2} \frac{{\bf u}_{{\alpha}+1}-{\bf
    u}_{\alpha}}{h_{{\alpha}+1}+h_{\alpha}} - 2\nu_{\alpha-1/2}
\frac{{\bf u}_{\alpha}-{\bf
    u}_{{\alpha}-1}}{h_{\alpha}+h_{{\alpha}-1}} - \kappa_{\alpha}{\bf
  u}_{\alpha} + W_\alpha {\bf t}_s,\quad {\alpha}=1,...,N,\label{eq:nsml2}
\end{eqnarray}
with
\begin{eqnarray}
{\bf \Sigma}_{\alpha+1/2} & = & \begin{pmatrix}
\Sigma_{xx,\alpha+1/2} & \Sigma_{xy,\alpha+1/2}\\
\Sigma_{yx,\alpha+1/2} & \Sigma_{yy,\alpha+1/2}
\end{pmatrix},\label{eq:sigma_alpha}\\
\Sigma_{xx,\alpha+1/2} & = & \frac{\nu_{\alpha+1/2}}{h_{\alpha+1}+h_\alpha}\bigl( h_{\alpha}\frac{\partial
  u_{\alpha}}{\partial x} + h_{\alpha+1}\frac{\partial u_{\alpha+1}}{\partial x}  \bigr) - 2 \nu_{\alpha+1/2}\frac{\partial z_{\alpha+1/2}}{\partial x}
\frac{u_{\alpha+1} -
  u_\alpha}{h_{\alpha+1} + h_\alpha},\label{eq:sigma_mxx}\\
\Sigma_{xy,\alpha+1/2} & = & \frac{\nu_{\alpha+1/2}}{h_{\alpha+1}+h_\alpha}\bigl( h_{\alpha}\frac{\partial
  u_{\alpha}}{\partial y} + h_{\alpha+1}\frac{\partial u_{\alpha+1}}{\partial y}  \bigr) - 2 \nu_{\alpha+1/2}\frac{\partial z_{\alpha+1/2}}{\partial y}
\frac{u_{\alpha+1} -
  u_\alpha}{h_{\alpha+1} + h_\alpha},\label{eq:sigma_mxy}\\
{\bf \Sigma}_{\alpha} & = & \begin{pmatrix}
\Sigma_{xx,\alpha} & \Sigma_{xy,\alpha}\\
\Sigma_{yx,\alpha} & \Sigma_{yy,\alpha}
\end{pmatrix} = \frac{{\bf \Sigma}_{\alpha+1/2}+ {\bf
    \Sigma}_{\alpha-1/2}}{2},\label{eq:sigma_m}
\end{eqnarray}
and
\begin{equation}
\kappa_{\alpha}=
\left\{
\begin{array}{l}
\kappa \quad {\rm if} \quad \alpha =1\\
0 \quad {\rm if} \quad\alpha \ne 1
\end{array}
\right.
\quad
\nu_{\alpha+1/2}=
\left\{
\begin{array}{l}
0 \quad {\rm if} \quad \alpha =0,N\\
\nu \quad {\rm if} \quad\alpha= 1,...,N-1\\
\end{array}\right.
\quad
W_{\alpha}=
\left\{
\begin{array}{l}
W \quad {\rm if} \quad \alpha =N\\
0 \quad {\rm if} \quad\alpha \neq N\\
\end{array}\right.
\label{eq:kappa_mu}
\end{equation}
The vertical velocities $\{w_\alpha\}_{\alpha=1}^N$ are defined by~\eqref{eq:def_w}.

For smooth solutions, the system~\eqref{eq:nsml1}-\eqref{eq:nsml2}
admits the energy balance
\begin{eqnarray}
\frac{\partial }{\partial t} \sum_{\alpha=1}^N E_{\alpha} &+&
\nabla_{x,y} . \sum_{\alpha=1}^N {\bf u}_\alpha\left(E_{\alpha} +
 \frac{g}{2} h_\alpha h - h_\alpha{\bf \Sigma}_\alpha \right) 
 =   -\sum_{\alpha=1}^{N-1} \frac{|{\bf u}_{\alpha+1/2}
 - {\bf u}_\alpha|^2}{2} |G_{\alpha+1/2}| \nonumber\\
 & &\qquad -\sum_{\alpha=1}^{N-1} \frac{h_{\alpha+1} + h_\alpha}{2\nu}{\bf
   \Sigma}_{\alpha+1/2}^2
 - \sum_{\alpha=1}^{N-1} 2\nu \frac{|{\bf u}_{{\alpha}+1}-{\bf
     u}_{\alpha}|^2}{h_{{\alpha}+1}+h_{\alpha}} -\kappa |{\bf u}_1|^2 + W {\bf u}_N . {\bf t}_s,
\label{eq:energy_glol_ns}
\end{eqnarray}
with $E_\alpha$ defined
by~\eqref{eq:energ_al} and ${\bf \Sigma}_{\alpha+1/2}^2 = \sum_{i,j} \Sigma_{i,j,\alpha+1/2}^2$. Relation~\eqref{eq:energy_glol_ns} is
consistent with a layer-averaged discretization of the equation~\eqref{eq:energy_eq_mod_NS}.
\label{prop:NS_mc}
\end{proposition}
Notice that in~\eqref{eq:energy_glol_ns}, we use the notation
$${\bf u}_\alpha {\bf \Sigma}_\alpha = \begin{pmatrix}
u_\alpha {\bf \Sigma}_{xx,\alpha} + v_\alpha {\bf \Sigma}_{yx,\alpha}\\
u_\alpha {\bf \Sigma}_{xy,\alpha}
+ v_\alpha {\bf \Sigma}_{yy,\alpha}
\end{pmatrix}.$$

\begin{proof}[Proof of prop.~\ref{prop:NS_mc}]
The proof is given in appendix A.
\end{proof}

\begin{remark}
Notice that in the definition~\eqref{eq:sigma_m}, since we consider viscous terms we use a centered
approximation.
\end{remark}


\subsubsection{Simplified rheology}

The viscous terms in the layer-averaged Navier-Stokes system are
difficult to discretize especially when a
discrete version of the energy balance has to be preserved. Hence, we propose a simplified
version of the model given in prop.~\ref{prop:NS_mc}.

First, using simple manipulations, the viscous terms in the layer-averaged
Navier-Stokes can be rewritten and Eq.~\eqref{eq:nsml2} becomes
\begin{eqnarray*}
\frac{\p  h_{\alpha}{\bf u}_{\alpha}}{\p t} && \!\!\!\!\!\!\!\!  +  \nabla_{x,y} .\left(
  h_{\alpha} {\bf u}_{\alpha} \otimes {\bf u}_{\alpha} \right) +
\nabla_{x,y}  \bigl(\frac{g}{2} h h_{\alpha} \bigr) = -gh_\alpha \nabla_{x,y} z_b \nonumber\\
& & + {\bf u}_{\alpha+1/2} G_{\alpha+1/2} - {\bf u}_{\alpha-1/2}
G_{\alpha-1/2} + \nabla_{x,y} .
\bigl( h_\alpha {\bf \Sigma}_\alpha^0 \bigr) - {\bf T}_\alpha\nonumber\\
&& +\Lambda_{\alpha+1/2} ({\bf u}_{{\alpha}+1}-{\bf
    u}_{\alpha}) - \Lambda_{\alpha-1/2} ({\bf u}_{\alpha}-{\bf u}_{{\alpha}-1}) -\kappa_{\alpha}{\bf
  u}_{\alpha} + W_\alpha {\bf t}_s,\quad {\alpha}=1,...,N,\label{eq:nsml2_1}
\end{eqnarray*}
with
\begin{eqnarray}
\Sigma_{xx,\alpha+1/2}^0 & = & \frac{\nu_{\alpha+1/2}}{h_{\alpha+1}+h_\alpha}\bigl( h_{\alpha}\frac{\partial
  u_{\alpha}}{\partial x} + h_{\alpha+1}\frac{\partial u_{\alpha+1}}{\partial x}  \bigr),\label{eq:sigma_mxx_1}\\
\Sigma_{xy,\alpha+1/2}^0 & = & \frac{\nu_{\alpha+1/2}}{h_{\alpha+1}+h_\alpha}\bigl( h_{\alpha}\frac{\partial
  u_{\alpha}}{\partial y} + h_{\alpha+1}\frac{\partial
  u_{\alpha+1}}{\partial y}  \bigr),\label{eq:sigma_mxy_1}\\
T_{x,\alpha+1/2} & = & 2\nu_{\alpha+1/2}\frac{\partial z_{\alpha+1/2}}{\partial x}\frac{\partial u_{\alpha+1}}{\partial x}
+ 2\nu_{\alpha+1/2}\frac{\partial
  z_{\alpha+1/2}}{\partial y}\frac{\partial u_{\alpha+1}}{\partial
  y},\label{eq:sigma_mxy_2}\\
\Lambda_{\alpha+1/2} & = &
2\nu_{\alpha+1/2}\frac{1 + |\nabla_{x,y}
  z_{\alpha+1/2}|^2}{h_{{\alpha}+1}+h_{\alpha}}-\nu_{\alpha+1/2}\nabla_{x,y}
                           . \left(\frac{h_\alpha \nabla_{x,y}
z_{\alpha+1/2}}{h_{{\alpha}+1}+h_{\alpha}}\right),\label{eq:lambda}\\
{\bf T}_{\alpha+1/2} & = & \begin{pmatrix}
T_{x,\alpha+1/2} \\
T_{y,\alpha+1/2} 
\end{pmatrix},\quad
{\bf T}_{\alpha} = \begin{pmatrix}
T_{x,\alpha}\\
T_{y,\alpha}
\end{pmatrix} = \frac{{\bf T}_{\alpha+1/2}+ {\bf T}_{\alpha-1/2}}{2},\label{eq:R}\\
{\bf \Sigma}^0_{\alpha+1/2} & = & \begin{pmatrix}
\Sigma_{xx,\alpha+1/2}^0 & \Sigma_{xy,\alpha+1/2}^0\\
\Sigma_{yx,\alpha+1/2}^0 & \Sigma_{yy,\alpha+1/2}^0
\end{pmatrix},\label{eq:sigma_alpha_1}\\
{\bf \Sigma}_{\alpha}^0 & = & \begin{pmatrix}
\Sigma_{xx,\alpha}^0 & \Sigma_{xy,\alpha}^0\\
\Sigma_{yx,\alpha}^0 & \Sigma_{yy,\alpha}^0
\end{pmatrix} = \frac{{\bf \Sigma}_{\alpha+1/2}^0+ {\bf
    \Sigma}_{\alpha-1/2}^0}{2}. \label{eq:sigma_m_1}
\end{eqnarray}
Considering a large number of layers i.e. $h_\alpha = l_\alpha h
\rightarrow 0$ then the quantity $\Lambda_{\alpha+1/2}$ reduces to
\begin{equation}
\Gamma_{\alpha+1/2} =
2\nu_{\alpha+1/2}\frac{1 + |\nabla_{x,y}
  z_{\alpha+1/2}|^2}{h_{{\alpha}+1}+h_{\alpha}}.
\label{eq:lambda_12}
\end{equation}
Now let us examine the quantity ${\bf T}_{\alpha}$. Once the 
approximation has been made in~\eqref{eq:lambda} replacing it
by~\eqref{eq:lambda_12}, the only way for the layer-averaged model to satisfy
an energy balance is to neglect the quantity ${\bf
  T}_{\alpha}$. The removal of the quantity ${\bf T}_{\alpha}$ is
the mandatory counterpart of the valid approximation made
in~\eqref{eq:lambda}. Moreover, the first component of ${\bf T}_{\alpha}$
writes (for the sake of simplicity, we assume $1<\alpha<N$ and
$h_j=h/N$, $\forall j$)
\begin{eqnarray*}
T_{x,\alpha} & = & -\frac{\nu}{2}
\nabla_{x,y} z_{\alpha+1/2} . \nabla_{x,y} u_{\alpha+1} +
                   \frac{\nu}{2}
                   \nabla_{x,y} z_{\alpha-1/2} . \nabla_{x,y}
                   u_{\alpha-1}\\
& = & -\frac{\nu}{2} \nabla_{x,y} z_{\alpha} . \nabla_{x,y}
      (u_{\alpha+1} - u_{\alpha-1}) -
                   \frac{\nu}{2}
                   \nabla_{x,y} (u_{\alpha+1} + u_{\alpha-1}) . \nabla_{x,y}
                   h_{\alpha},
\end{eqnarray*}
and considering smooth solutions (meaning $\frac{\partial u}{\partial z}$ is bounded) and a large value of $N$, we have
$$u_{\alpha+1} - u_{\alpha-1} \rightarrow 0,\quad h_\alpha\rightarrow
0,$$
and hence $T_{x,\alpha}$ can be neglected compared to the other
rheology terms.

\begin{remark}
The approximations concerning the viscous terms ${\bf \Lambda}_{\alpha\pm 1/2}$ and ${\bf T}_\alpha$ can be explained geometrically as follows. The second term in~\eqref{eq:lambda} and the vector ${\bf T}_\alpha$ involve the quantities $\nabla_{x,y} z_{\alpha\pm 1/2}$ i.e. the gradient of the boundaries of each layer and arise from the layer averaging formulation. Except in few particular cases (equilibrium at rest with flat topography,\ldots), when $N$ is small the two quantities $\nabla_{x,y} z_{\alpha-1/2}$ and $\nabla_{x,y} z_{\alpha+1/2}$ significantly differ (see Fig.\ref{fig:free}-{\it (b)}). Whereas for smooth solutions and $N$ large $\nabla_{x,y} z_{\alpha-1/2} \approx \nabla_{x,y} z_{\alpha+1/2}$ and the corresponding contributions in ${\bf \Lambda}_{\alpha\pm 1/2}$ and ${\bf T}_\alpha$ can be neglected.
\end{remark}

So finally, 
with a simplified expression of the rheology terms, the layer-averaged
hydrostatic Navier-Stokes system given in prop.\ref{prop:NS_mc} becomes
\begin{eqnarray}
\sum_{\alpha=1}^N \frac{\partial  h_\alpha  }{\partial t } && \!\!\!\!\!\!\!\!
+ \sum_{\alpha=1}^N  \ \nabla_{x,y} .   (h_\alpha  {\bf u} _\alpha)  =0, \label{eq:nsml1_22} \\
\frac{\p  h_{\alpha}{\bf u}_{\alpha}}{\p t} && \!\!\!\!\!\!\!\!\!  +  \nabla_{x,y} .\left(
  h_{\alpha} {\bf u}_{\alpha} \otimes {\bf u}_{\alpha} \right) +
\nabla_{x,y}  \bigl(\frac{g}{2} h h_{\alpha} \bigr) = -gh_\alpha \nabla_{x,y} z_b \nonumber\\
& & + {\bf u}_{\alpha+1/2} G_{\alpha+1/2} - {\bf u}_{\alpha-1/2}
G_{\alpha-1/2} + \nabla_{x,y} .
\bigl( h_\alpha {\bf \Sigma}_\alpha^0 \bigr) \nonumber\\
&& +\Gamma_{\alpha+1/2} ({\bf u}_{{\alpha}+1}-{\bf
    u}_{\alpha}) - \Gamma_{\alpha-1/2} ({\bf u}_{\alpha}-{\bf u}_{{\alpha}-1}) -\kappa_{\alpha}{\bf
  u}_{\alpha} + W_\alpha {\bf t}_s,\label{eq:nsml2_22}
\end{eqnarray}
with $\Sigma^0$ defined
by~\eqref{eq:sigma_mxx_1},\eqref{eq:sigma_mxy_1},\eqref{eq:sigma_alpha_1} and
\eqref{eq:sigma_m_1}.
For smooth solutions, the system~\eqref{eq:nsml1_22}-\eqref{eq:nsml2_22}
admits the energy balance
\begin{eqnarray}
\frac{\partial }{\partial t} \sum_{\alpha=1}^N E_{\alpha} &+&
\nabla_{x,y} . \sum_{\alpha=1}^N {\bf u}_\alpha\left(E_{\alpha} +
 \frac{g}{2} h_\alpha h - h_\alpha{\bf \Sigma}_\alpha^0 \right) \nonumber\\
 & = &  -\sum_{\alpha=1}^{N-1} \frac{|{\bf u}_{\alpha+1/2}
 - {\bf u}_\alpha|^2}{2} |G_{\alpha+1/2}| \nonumber\\
 & & -\sum_{\alpha=1}^{N-1} \frac{h_{\alpha+1} + h_\alpha}{2\nu}({\bf \Sigma}_{\alpha+1/2}^0)^2
 - \sum_{\alpha=1}^{N-1} \Gamma_{\alpha+1/2} |{\bf u}_{{\alpha}+1}-{\bf
     u}_{\alpha}|^2 -\kappa |{\bf u}_1|^2 + W {\bf u}_N . {\bf t}_s.
\label{eq:energy_glol_ns1}
\end{eqnarray}
The proof of the energy balance~\eqref{eq:energy_glol_ns1} is similar to the one given in prop.~\ref{prop:NS_mc}.

\begin{remark}
The layer-averaged Navier-Stokes system defined by~\eqref{eq:nsml1_22}-\eqref{eq:nsml2_22} has the form
\begin{equation}
\frac{\partial U}{\partial t} + \nabla_{x,y} . F(U) = S_b(U) +
S_e(U,\partial_t U,\partial_x U) + S_{v,f}(U),
\label{eq:glo}
\end{equation}
where $U=\left(
h,
q_{x,1},
\ldots,
q_{x,N},
q_{y,1},
\ldots,
q_{y,N}\right)^T
$, and
\begin{equation*}
S_b(U) = \left(
0,
gh_1 \frac{\partial z_b}{\partial x},
\ldots,
gh_N \frac{\partial z_b}{\partial x},
gh_1 \frac{\partial z_b}{\partial y},
\ldots,
gh_N \frac{\partial z_b}{\partial y}\right)^T,
\label{eq:SbU}
\end{equation*}
with $q_{x,\alpha}=h_\alpha u_{\alpha}$, $q_{y,\alpha}=h_\alpha
v_{\alpha}$. We denote with $F(U)$ the fluxes of the conservative part,
and with $S_e(U,\partial_t U,\partial_x U)$ and $S_{v,f}(U)$ the source terms,
 representing respectively the
momentum exchanges and the viscous, wind and
friction effects.

The numerical scheme for the system~\eqref{eq:glo} will be given in Section~\ref{sec:NumMet}.
\end{remark}

\section{Kinetic description for the Euler system}
\label{sec:kin}

In this section we give a kinetic interpretation of the
system~\eqref{eq:massesvml}-\eqref{eq:energy_mcl}. The numerical
scheme for the
system~\eqref{eq:massesvml}-\eqref{eq:mvtsvml},\eqref{eq:def_w}
will be deduced from the kinetic description.

\subsection{Preliminaries}

We begin this section by recalling the classical kinetic approach~--
used in \cite{simeoni} for example~-- for the 1d Saint-Venant system
\begin{equation}\begin{array}{l}
	 \partial_t h+\partial_x(hu)=0,\\
	\partial_t(hu)+\partial_x(hu^2+g\frac{h^2}{2})+gh\partial_x z_b=0,
	\label{eq:SV}
	\end{array}
\end{equation}
with the water depth $h(t,x)\geq 0$, the water velocity
$u(t,x)\in\R$ and a slowly varying topography $z_b(x)$.

The kinetic Maxwellian is given by
\begin{equation}
	M(U,\xi)=\frac{1}{g\pi}\Bigl(2gh-(\xi-u)^2\Bigr)_+^{1/2},
	\label{eq:kinmaxw_1d}
\end{equation}
where $U=(h,hu)^T$, $\xi\in\R$ and $x_+\equiv\max(0,x)$ for any $x\in\R$.
It satisfies the following moment relations,
\begin{equation}\begin{array}{c}
	\int_\R \begin{pmatrix} 1\\ \xi \end{pmatrix}
                  M(U,\xi)\,d\xi=U,\qquad \int_\R \xi^2 M(U,\xi)\,d\xi=hu^2+g\frac{h^2}{2}.
	\label{eq:kinmom_1d}
	\end{array}
\end{equation}
These definitions allow us to obtain a {\sl kinetic representation} of the
Saint-Venant system.
\begin{lemma}
If the topography $z_b(x)$ is Lipschitz continuous,
the pair of functions $(h, hu)$ is a weak solution to the Saint-Venant system~\eqref{eq:SV}
if and only if $M(U,\xi)$ satisfies the kinetic equation
\begin{equation}
	\partial_t M+\xi\partial_x M-g(\partial_x z_b)\partial_\xi M=Q,
	\label{eq:kinrepres}
\end{equation}
for some ``collision term'' $Q(t,x,\xi)$ that satisfies, for a.e. $(t,x)$,
\begin{equation}
	\int_\R Q d\xi = \int_\R \xi Qd \xi = 0.
	\label{eq:kinrepresintcoll}
\end{equation}
\label{lemma:sv_kin_rep}
\end{lemma}
\begin{proof}
If~\eqref{eq:kinrepres} and~\eqref{eq:kinrepresintcoll} are satisfied,
we can multiply \eqref{eq:kinrepres} by $(1,\xi)^T$, and integrate with respect to $\xi$.
Using~\eqref{eq:kinmom_1d} and \eqref{eq:kinrepresintcoll} and integrating by parts
the term in $\partial_\xi M$, we obtain \eqref{eq:SV}.
Conversely, if $(h, hu)$ is a weak solution to \eqref{eq:SV}, just define $Q$ by \eqref{eq:kinrepres};
it will satisfy \eqref{eq:kinrepresintcoll} according to the same computations.
\end{proof}
The standard way to use Lemma \ref{lemma:sv_kin_rep} is to write
a kinetic relaxation equation~\cite{JSM_entro,BGKcons,bouchut_BGK}, like
\begin{equation}
	\partial_t f+\xi\partial_x f-g(\partial_x z_b)\partial_\xi f=\frac{M-f}{\epsilon},
	\label{eq:kinrelax}
\end{equation}
where $f(t,x,\xi)\geq 0$, $M=M(U,\xi)$ with $U(t,x)=\int(1,\xi)^Tf(t,x,\xi)d\xi$,
and $\epsilon>0$ is a relaxation time. In the limit $\epsilon\to 0$ we recover formally
the formulation \eqref{eq:kinrepres}, \eqref{eq:kinrepresintcoll}.
We refer to \cite{BGKcons} for general considerations on such kinetic relaxation models
without topography, the case with topography being introduced in \cite{simeoni}.
Note that the notion of {\sl kinetic representation} as \eqref{eq:kinrepres}, \eqref{eq:kinrepresintcoll}
differs from the so called {\sl kinetic formulations} where a large set of entropies is involved,
see \cite{perthame}. For systems of conservation laws, these kinetic formulations
include non-advective terms that prevent from writing down simple approximations.
In general, kinetic relaxation approximations can be compatible with just a single entropy.
Nevertheless this is enough for proving the convergence as $\varepsilon\to 0$, see \cite{BB}.\\

\subsection{Kinetic interpretation}
\label{subsec:kin_description}

In this paragaph, we give a kinetic interpretation of the model~\eqref{eq:massesvml}-\eqref{eq:mvtsvml},\eqref{eq:energy_glol}.

To build the Gibbs equilibria, we choose the function
\begin{equation}
\chi_0 (z_1,z_2) = \frac{1}{4\pi}\1_{z_1^2 + z_2^2 \leq 4}.
\label{eq:chi0}
\end{equation}
This choice corresponds to the 2d version of the kinetic
maxwellian used in 1d (see remark~\ref{rem:kim_max})
and we have
\begin{equation}
	M_\alpha = M(U_\alpha,\xi,\gamma)= \frac{h_\alpha}{c^2} \chi_0 \left( \frac{\xi - u_\alpha}{c},\frac{\gamma - v_\alpha}{c}\right),
	\label{eq:kinmaxw}
\end{equation}
with $c = \sqrt{\frac{g}{2}h}$
\begin{equation}
U_\alpha =  (h_\alpha, h_\alpha u_\alpha, h_\alpha v_\alpha)^T,\label{eq:u_alpha}
\end{equation}
and where $(\xi,\gamma) \in \R^2$. In other words, we have $M_\alpha =
\frac{l_\alpha}{2g\pi}\1_{(\xi-u_\alpha)^2+(\gamma-v_\alpha)^2\leq
  2gh}$.
\begin{remark}
Starting from the 2d maxwellian in the single layer case i.e.
\begin{equation}
M_{sv} =
\frac{1}{2g\pi}\1_{(\xi-u)^2+(\gamma-v)^2\leq
  2gh},
\label{eq:M1d}
\end{equation}
and computing its integral w.r.t. the variable $\gamma$ yields
$$\int_{\R} M_{sv} d\gamma =
\int_{v-\sqrt{(2gh-(\xi-u)^2)_+}}^{v+\sqrt{(2gh-(\xi-u)^2)_+}} \frac{1}{2g\pi}
d\gamma = \frac{1}{g\pi}\sqrt{(2gh-(\xi-u)^2)_+},$$
that is exactly the expression~\eqref{eq:kinmaxw_1d}.
\label{rem:kim_max}
\end{remark}
The quantity $M_\alpha$ satisfies the following moment relations
\begin{equation}
\begin{array}{l}
	\disp \int_{\R^2} \kxi M(U_\alpha,\xi,\gamma)\,d\xi d\gamma
        = \begin{pmatrix} h_\alpha\\ h_\alpha u_\alpha \\ h_\alpha
          v_\alpha \end{pmatrix},\quad
	\disp\int_{\R^2}  \begin{pmatrix} \xi^2\\ \xi\gamma \\
          \gamma^2 \end{pmatrix} M(U_\alpha,\xi,\gamma)\,d\xi
        d\gamma=\begin{pmatrix} h_\alpha
        u_\alpha^2+g\frac{h_\alpha h}{2}\\
h_\alpha u_\alpha v_\alpha \\
h_\alpha v_\alpha^2+g\frac{h_\alpha h}{2}
\end{pmatrix}
.
	\label{eq:kinmom}
	\end{array}
\end{equation}
The interest of the function $\chi_0$ and hence the particular form \eqref{eq:kinmaxw} lies in its link with a kinetic entropy.
Consider the kinetic entropy
\begin{equation}
	H(f,\xi,\gamma,z_b)=
          \frac{\xi^2+\gamma^2}{2}f+gz_b f,
	\label{eq:kinH}
\end{equation}
where $f\geq 0$, $(\xi,\gamma)\in\R^2$, $z_b\in\R$. Then one can check the relations
\begin{equation}
\int_{\R^2} \kxi H(M_\alpha,\xi,\gamma) d\xi d\gamma
= \begin{pmatrix}
E_\alpha = \frac{h_\alpha}{2} (u_\alpha^2 + v_\alpha^2)+
\frac{g}{2} h_\alpha (h + 2z_b)  \\
u_\alpha
(E_\alpha + \frac{g}{2} h_\alpha h)\\ v_\alpha
(E_\alpha + \frac{g}{2} h_\alpha h).
\end{pmatrix}
\label{eq:kinmom_E}
\end{equation}

Let us introduce the Gibbs equilibria $N_{\alpha+1/2}$ defined by for
$\alpha=0,\ldots,N$ by
\begin{multline}
N_{\alpha+1/2} = N(U_{\alpha+1/2},\xi) = \frac{G_{\alpha+1/2}}{c^2} \
\chi_0 \left(\frac{\xi -  u_{\alpha+1/2}}{c},\frac{\gamma -
    v_{\alpha+1/2}}{c}\right) \\
= \frac{G_{\alpha+1/2}}{g\pi h} \1_{(\xi -
  u_{\alpha+1/2})^2+(\gamma -  v_{\alpha+1/2})^2\leq 2gh} = \frac{G_{\alpha+1/2}}{h}M_{\alpha+1/2},
\label{eq:Nbis}
\end{multline}
where $G_{\alpha+1/2}$ is defined by~\eqref{eq:Qalphabis} and
$u_{\alpha+1/2}$,$v_{\alpha+1/2}$ are given
by~\eqref{eq:upwind_uT}.
The quantity $N_{\alpha+1/2}$ satisfies the following moment relations
\begin{equation}
\int_{\R^2} \kxi N_{\alpha+1/2} d\xi d\gamma = \begin{pmatrix}
  G_{\alpha+1/2}\\ u_{\alpha+1/2}G_{\alpha+1/2}\\
  v_{\alpha+1/2}G_{\alpha+1/2}\end{pmatrix},\quad
\int_{\R^2} \begin{pmatrix} \frac{\xi^2}{2}\\
  \frac{\gamma^2}{2} \end{pmatrix}
N_{\alpha+1/2} d\xi d\gamma = \begin{pmatrix}
  \left( \frac{u_{\alpha+1/2}^2}{2} + \frac{g}{4}h
  \right) G_{\alpha+1/2}\\
  \left( \frac{v_{\alpha+1/2}^2}{2} + \frac{g}{4}h
  \right) G_{\alpha+1/2}\end{pmatrix}.
\label{eq:exchange_kin}
\end{equation}
Notice that from~\eqref{eq:Qalphabis}, we
can give a kinetic interpretation on the exchange terms under the form
\begin{equation}
G_{\alpha+1/2} = -\sum_{j=1}^N \Bigl(\sum_{p=1}^\alpha l_p - \1_{j\leq
  \alpha}\Bigr) \int_{\R^2} \begin{pmatrix} \xi\\ \gamma
    \end{pmatrix} . \nabla_{x,y} M_j
d\xi d\gamma,
\label{eq:Qalphabis_kin}
\end{equation}
for $\alpha=1,\ldots,N$.

Then we have the two following results.

\begin{proposition}
The functions ${\bf u}^{N}$ defined by~\eqref{eq:ulayer} and $h$ are strong solutions of the system
(\ref{eq:massesvml})-(\ref{eq:mvtsvml}) if and only if the sets of
equilibria $\{M_\alpha\}_{\alpha=1}^N$,~$\{N_{\alpha+1/2}\}_{\alpha=0}^N$
are solutions of the kinetic equations defined by
\begin{equation}
({\cal B}_\alpha)\qquad\frac{\partial M_\alpha}{\partial t} + \begin{pmatrix} \xi\\ \gamma
    \end{pmatrix} . \nabla_{x,y} M_\alpha
  -g \nabla_{x,y} z_b . \nabla_{\xi,\gamma} M_\alpha
 - N_{\alpha+1/2} + N_{\alpha-1/2}
 = Q_{\alpha},
\label{eq:gibbsbis}
\end{equation}
for $\alpha=1,\ldots,N$. The quantities $Q_{\alpha} = Q_{\alpha}(t,x,y,\xi,\gamma)$
 are ``collision terms''  equal to zero at the
macroscopic level, i.e.\ they satisfy  a.e.\ for values of $(t,x,y)$
\begin{equation}
\int_{\mathbb{R}^2} Q_{\alpha} d\xi d\gamma = \int_{\mathbb{R}^2} \xi
Q_{\alpha} d\xi d\gamma=\int_{\mathbb{R}^2} \gamma
Q_{\alpha} d\xi d\gamma =0.
\label{eq:collisionbis}
\end{equation}
\label{prop:kinetic_sv_mcl}
\end{proposition}

\begin{proposition}
The solutions of~\eqref{eq:gibbsbis} are entropy solutions if
\begin{equation}
\disp \frac{\partial H(M_\alpha)}{\partial t} + \begin{pmatrix} \xi\\ \gamma
    \end{pmatrix} . \nabla_{x,y} H(M_\alpha) 
-g \nabla_{x,y} z_b . \nabla_{\xi,\gamma} H(M_\alpha)
\leq  (H(N_{\alpha+1/2}) - H(N_{\alpha-1/2})),
\label{eq:kin_entro1}
\end{equation}
with the notation $H(M) = H(M,\xi,\gamma,z_b)$ and $H$ defined by~\eqref{eq:kinH}.
The integration in $\xi,\gamma$ of relation~\eqref{eq:kin_entro1} gives
\begin{equation*}
\disp \frac{\partial E_\alpha}{\partial t}  + \nabla_{x,y} . {\bf u}_\alpha (E_\alpha +
\frac{g}{2} h_\alpha h)
\leq l_\alpha \left(
  \frac{|{\bf u}_{\alpha+1/2}|^2}{2} + g z_b \right)
G_{\alpha+1/2}\\
- l_\alpha \left(
  \frac{|{\bf u}_{\alpha-1/2}|^2}{2} + g z_b \right) G_{\alpha-1/2}.
\end{equation*}
\label{prop:entropy_kin}
\end{proposition}

\begin{proof}[Proof of proposition~\ref{prop:kinetic_sv_mcl}]
The proof relies on averages w.r.t the variables $\xi,\gamma$ of
Eq.~\eqref{eq:gibbsbis} against the vector $(1,\xi,\gamma)^T$. Using
relations~\eqref{eq:kinmom},\eqref{eq:Nbis},\eqref{eq:exchange_kin}
and the properties of the collision terms~\eqref{eq:collisionbis}, the quantities
$$\int_{\R^2} ({\cal B}_\alpha)\ d\xi
d\gamma,\quad \int_{\R^2} \xi ({\cal B}_\alpha)\ d\xi d\gamma,\quad\mbox{and}\quad \int_{\R^2} \gamma ({\cal B}_\alpha)\ d\xi d\gamma,$$
respectively give Eqs.~\eqref{eq:c0_mc} and \eqref{eq:mvtsvml}. The
sum for $\alpha=1$ to $N$ of Eqs.~\eqref{eq:c0_mc} with~\eqref{eq:Qalphabis}
gives~\eqref{eq:massesvml} that completes
the proof.
\end{proof}

\begin{proof}[Proof of prop.~\ref{prop:entropy_kin}]
The proof is obtained multiplying~\eqref{eq:gibbsbis} by
$H^\prime_\alpha(\overline{M}_\alpha,\xi,\gamma,z_b)$.
Indeed, it is easy to see that
$$H^\prime_\alpha(M_\alpha,\xi,\gamma,z_b)
\frac{\partial M_\alpha}{\partial v}
=
\frac{\partial}{\partial v} H_\alpha(M_\alpha,\xi,\gamma,z_b),$$
for $v=t,x,y,\xi,\gamma$.
Likewise for the quantity
$H^\prime_\alpha(M_\alpha,\xi,\gamma,z_b) N_{\alpha+1/2}$, we have
$$H^\prime_\alpha(M_\alpha,\xi,\gamma,z_b) N_{\alpha+1/2} =
H(N_{\alpha+1/2},\xi,\gamma,z_b).$$
So finally, Eq.~\eqref{eq:gibbsbis} multiplied by
$H^\prime_\alpha(M_\alpha,\xi,\gamma,z_b)$ gives
$$
\disp \frac{\partial H_\alpha}{\partial t} + \begin{pmatrix} \xi\\ \gamma
    \end{pmatrix} . \nabla_{x,y} H_\alpha -g \nabla_{x,y} z_b . \nabla_{\xi,\gamma} H_\alpha
\leq
\left( \frac{\xi^2+\gamma^2}{2}  + g z_b \right)
( N_{\alpha+1/2} - N_{\alpha-1/2}).
$$
It remains to calculate the sum of the preceding relations from
$\alpha=1,\ldots,N$ and to integrate the obtained relation in $\xi,\gamma$ over $\R^2$
that completes the proof.
\end{proof}

\begin{remark}
If we introduce a $(2N+1)\times N$ matrix ${\cal K(\xi,\gamma)}$
defined by
$$
{\cal K}_{1,j} = 1,\quad
 {\cal K}_{i+1,j} =
\xi\delta_{i,j},\quad
{\cal K}_{i+N+1,j+N} = \gamma\delta_{i,j},
$$
for $i,j=1,\ldots,N$ with $\delta_{i,j}$ the Kronecker symbol. Then, using Prop.~\ref{prop:kinetic_sv_mcl},
we can write
\begin{eqnarray}
U= \int_{\R^2} {\cal K( \xi,\gamma)} M(\xi,\gamma) d \xi d\gamma,
\quad F(U) = \int_{\R^2}  \begin{pmatrix} \xi\\ \gamma \end{pmatrix} {\cal K( \xi,\gamma)}
M(\xi,\gamma) d \xi d\gamma,
\label{eq:fx}\\
S_e(U) = \int_{\R^2} {\cal K( \xi,\gamma)} N(\xi,\gamma) d \xi d\gamma,
\label{eq:fxx}
\end{eqnarray}
with $M(\xi,\gamma)=(M(U_1,\xi,\gamma),\ldots,M(U_N,\xi,\gamma))^T$
and
$$N(\xi,\gamma)=\begin{pmatrix}
N_{3/2}(\xi,\gamma) - N_{1/2}(\xi,\gamma)\\
\vdots\\
N_{N+1/2}(\xi,\gamma) - N_{N-1/2}(\xi,\gamma)
\end{pmatrix}.$$
Hence, using the above notations, the layer-averaged Euler
system~\eqref{eq:massesvml}-\eqref{eq:mvtsvml} can be written under
the form
$$\int_{\R^2} {\cal K}(\xi,\gamma) \left( \frac{\partial M(\xi,\gamma)}{\partial t} + \begin{pmatrix} \xi\\ \gamma
    \end{pmatrix} . \nabla_{x,y} M(\xi,\gamma) - g\nabla_{x,y} z_b
. \nabla_{\xi,\gamma} M - N(\xi,\gamma)\right) d\xi d\gamma = 0.$$
\label{rem:kxi}
\end{remark}

\section{Numerical scheme}
\label{sec:NumMet}

The numerical scheme for the model~\eqref{eq:glo} proposed in this section extends the results
presented by some of the authors
in~\cite{bristeau,JSM_M2AN,JSM_JCP,JSM_entro}. Compared to these
previous results, it has the following advantages
\begin{itemize}
\item it gives a 3d approximation of the Navier-Stokes system whereas
  2d situations $(x,y)$ and $(x,z)$ where considered
  in~\cite{bristeau,JSM_M2AN,JSM_JCP},
\item the implicit treatment of the vertical exchanges terms gives a
  bounded CFL condition even when the water depth vanishes,
\item the kinetic interpretation, on which is based the numerical
  scheme, is also valid for the vertical exchange terms~-- that was not
  the case in~\cite{JSM_M2AN,JSM_JCP}~-- and allows to derive a robust and
  accurate numerical scheme,
\item the numerical approximation of the
system given in~\eqref{eq:glo} is endowed with strong stability
properties (well-balancing, positivity of the water depth,\ldots),
\item convergence curves towards a 3d non-stationary analytical
  solution with wet-dry interfaces have been obtained (see paragraph~\ref{subsubsec:bowl}).
\end{itemize}
First, we focus on the Euler part of the
system~\eqref{eq:glo} then in paragraph~\ref{subsec:NS_dis}, a
numerical scheme for the viscous terms is proposed.

Notice that, as a consequence of the layer-averaged discretization, the system~\eqref{eq:glo} and the
Boltzmann type equation~\eqref{eq:gibbsbis} are only 2d $(x,y)$ partial
differential equations with source terms. Hence, the spacial
approximation of the considered PDEs is performed on a 2d planar mesh.

\subsection{Semi-discrete (in time) scheme}

We consider discrete times $t^n$ with $t^{n+1}=t^n+\Delta t^n$. For
the time discretisation of the layer-averaged Navier-Stokes
system~\eqref{eq:glo}
we adopt the following scheme
\begin{equation}
U^{n+1} = U - \Delta t^n \left(\nabla_{x,y} . F(U) - S_b(U)\right) \\
+ \Delta t^n S_e^{n+1} + \Delta t^n S_{v,f}^{n+p},
\label{eq:glo_dis}
\end{equation}
where the superscript $^n$ has been
omitted and the integer $p=0,1/2,1$ will be precised below.

Using the expressions~\eqref{eq:nsml1_22}-\eqref{eq:nsml2_22} for the layer
averaged model, the semi-discrete in time scheme~\eqref{eq:glo_dis} writes
\begin{eqnarray}
h^{n+1} & = & h^{n+1/2}  = h - \Delta t^n\sum_{\alpha=1}^N  \ \nabla_{x,y} .   (h_\alpha  {\bf u} _\alpha), \label{eq:nsml1_d} \\
(h_{\alpha}{\bf u}_{\alpha})^{n+1/2} & = & h_{\alpha}{\bf u}_{\alpha} -
\Delta t^n \Bigl( \nabla_{x,y} .\left(
  h_{\alpha} {\bf u}_{\alpha} \otimes {\bf u}_{\alpha} \right) +
\nabla_{x,y}  \bigl(\frac{g}{2} h h_{\alpha} \bigr)  + gh_\alpha \nabla_{x,y} z_b\Bigr),\label{eq:nsml2_d}\\
(h_{\alpha}{\bf u}_{\alpha})^{n+1} & = & (h_{\alpha}{\bf u}_{\alpha})^{n+1/2} -
\Delta t^n \Bigl( {\bf u}_{\alpha+1/2}^{n+1} G_{\alpha+1/2} - {\bf
  u}_{\alpha-1/2}^{n+1} G_{\alpha-1/2} +\nabla_{x,y} .
\bigl( h_\alpha ^{n+p} {\bf \Sigma}_\alpha^{0,n+p} \bigr)\nonumber\\
& & +\Lambda_{\alpha+1/2} \frac{{\bf u}_{{\alpha}+1}^{n+p}-{\bf
    u}_{\alpha}^{n+p}}{h_{{\alpha}+1}^{n+p}+h_{\alpha}^{n+p}} - \Lambda_{\alpha-1/2}
\frac{{\bf u}_{\alpha}^{n+p}-{\bf
    u}_{{\alpha}-1}^{n+p}}{h_{\alpha}^{n+p}+h_{{\alpha}-1}^{n+p}} -\kappa_{\alpha}{\bf
  u}_{\alpha}^{n+p} + W_\alpha^{n+p} {\bf t}_s\Bigr),\label{eq:nsml3_d}\\
G_{\alpha+1/2} & = & -\sum_{j=1}^N \Bigl(\sum_{p=1}^\alpha l_p - \1_{j\leq \alpha}\Bigr)\nabla_{x,y} . (h_j {\bf u}_j), \label{eq:nsml4_d}
\end{eqnarray}
for ${\alpha}=1,\ldots,N$. The vertical
velocities $\{w_\alpha\}_{\alpha=1}^N$ are defined
by~\eqref{eq:def_w}. The first two equations\eqref{eq:nsml1_d}-\eqref{eq:nsml2_d} consist in an explicit
time scheme where the horizontal fluxes and the topography source term
are taken into account whereas in Eq.~\eqref{eq:nsml3_d} an implicit
treatment of the exchange terms between layers is proposed. The
implicit part of the scheme requires to solve a linear problem (see
lemma.~\ref{lem:inverse}) but, on the
contrary of previous work of some of the authors~\cite{JSM_M2AN}, it implies that the
CFL condition~\eqref{eq:CFLfull} no more depends on the exchange terms.

When $\nu=\kappa=0$ in Eq.~\eqref{eq:nsml3_d},
Eqs.~\eqref{eq:nsml1_d}-\eqref{eq:nsml3_d} correspond to the
layer-averaged of the Euler system. The choice $p=1$
(resp. $p=1/2$) in
Eq.~\eqref{eq:nsml3_d} corresponds to an implicit (resp. semi-implicit) treatment of the
viscous and friction terms whereas the choice $p=0$ implies an
explicit treatment and requires a CFL condition. Notice that the
advantages and limitations of an implicit or explicit discretization
in time scheme for the viscous and friction parts of Eq.~\eqref{eq:nsml3_d} are not detailed here.

\subsection{Space discretization}
\label{subsec:triangulation}

Let $\Omega$ denote the computational domain with boundary
$\Gamma$, which we assume is polygonal. Let $T_h$ be a triangulation
of $\Omega$ for which the vertices are denoted by $P_i$ with $S_i$ the
set of interior nodes and $G_i$ the set of boundary nodes.

For the space discretization of the
system\eqref{eq:nsml1_d}-\eqref{eq:nsml3_d}, we use a finite volume
technique for the Euler part~-- that is described below~-- and a
finite element approach~--$\mathbb{P}_1$ on $T_h$~-- for the
viscous part that is described in paragraph~\ref{subsec:NS_dis}.

\subsection{Finite volume formalism for the Euler part}
\label{subsec:fv}

In this paragraph and in paragraph~\eqref{subsec:discrete_kin}, we
propose a space discretization for the
model~\eqref{eq:nsml1_d}-\eqref{eq:nsml3_d} without the viscous and
friction terms i.e. the system
\begin{eqnarray}
h^{n+1} & = & h^{n+1/2}  = h - \Delta t^n\sum_{\alpha=1}^N  \nabla_{x,y} .   (h_\alpha  {\bf u} _\alpha), \label{eq:nsml11_d} \\
(h_{\alpha}{\bf u}_{\alpha})^{n+1/2} & = & h_{\alpha}{\bf u}_{\alpha} -
\Delta t^n \Bigl( \nabla_{x,y} .\left(
  h_{\alpha} {\bf u}_{\alpha} \otimes {\bf u}_{\alpha} \right) +
\nabla_{x,y}  \bigl(\frac{g}{2} h h_{\alpha} \bigr)  + gh_\alpha \nabla_{x,y} z_b\Bigr),\label{eq:nsml22_d}\\
(h_{\alpha}{\bf u}_{\alpha})^{n+1} & = & (h_{\alpha}{\bf u}_{\alpha})^{n+1/2} -
\Delta t^n \Bigl( {\bf u}_{\alpha+1/2}^{n+1} G_{\alpha+1/2} - {\bf
  u}_{\alpha-1/2}^{n+1} G_{\alpha-1/2} \Bigr),\label{eq:nsml33_d}
\end{eqnarray}
completed with~\eqref{eq:nsml4_d}.

We recall now the general formalism of finite volumes on unstructured
meshes.

The dual
cells $C_i$ are obtained by joining the centers of mass of the
triangles surrounding each vertex $P_i$. We use the following
notations (see Fig.~\ref{fig:mesh}):
\begin{itemize}
\item $K_i$, set of subscripts of nodes $P_j$ surrounding $P_i$,
\item $|C_i|$, area of $C_i$,
\item $\Gamma_{ij}$, boundary edge between the cells $C_i$ and $C_j$,
\item $L_{ij}$, length of $\Gamma_{ij}$,
\item ${\bf n}_{ij}$, unit normal to $\Gamma_{ij}$, outward to $C_i$ (${\bf
  n}_{ji}= -{\bf n}_{ij}$).
\end{itemize}
If $P_i$ is a node belonging to the boundary $\Gamma$, we join the centers of mass of the triangles adjacent to the
boundary to the middle of the edge belonging to $\Gamma$ (see Fig.~\ref{fig:mesh}) and we denote
\begin{itemize}
\item $\Gamma_i$, the two edges of $C_i$ belonging to $\Gamma$,
\item $L_i$, length of $\Gamma_i$ (for sake of simplicity we assume in
  the following that $L_i = 0$ if $P_i$ does not belong to $\Gamma$),
\item ${\bf n}_i$, the unit outward normal defined by averaging the two adjacent normals.
\end{itemize}

\begin{figure}[hbtp]
\begin{center}
\begin{tabular}{cc}
\includegraphics[height=3.5cm]{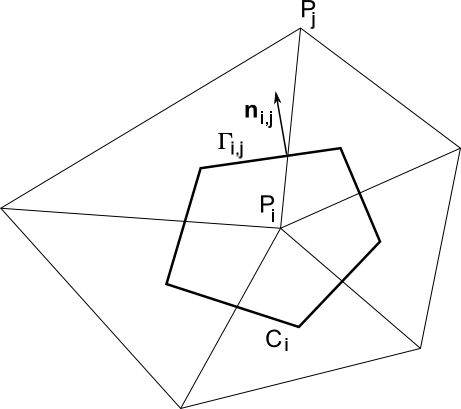}
\quad & \quad \includegraphics[height=3.5cm]{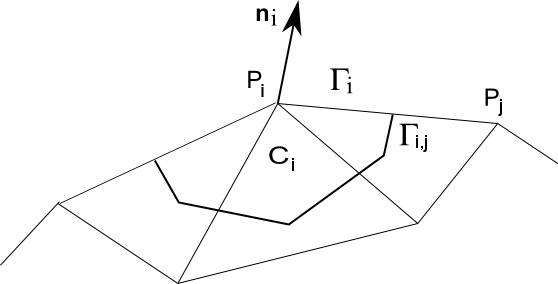} \\
{\it (a)} & {\it (b)}
\end{tabular}
\caption{{\it (a)}  Dual cell $C_i$ and {\it (b)} Boundary cell $C_i$.}
\label{fig:mesh}
\end{center}
\end{figure}

We define the piecewise constant functions $U^n(x,y)$ on cells $C_i$ corresponding to time $t^n$ and $z_b(x,y)$ as
\begin{equation}
	U^n(x,y)= U^n_i,\quad z_b(x,y)= z_{b,i},\quad\mbox{ for } (x,y)\in C_i,
	\label{eq:U^npc}
\end{equation}
with $U^n_i = (h^n_i,q^n_{x,1,i},\ldots,q^n_{x,N,i},q^n_{y,1,i},\ldots,q^n_{y,N,i})^T$ i.e.
$$U^n_i \approx \frac{1}{|C_i|}\int_{C_i} U(t^n,x,y)dxdy,\quad z_{b,i}
\approx \frac{1}{|C_i|}\int_{C_i} z_b(x,y)dxdy.$$
We will also use the notation
$$U^n_{\alpha,i} \approx \frac{1}{|C_i|}\int_{C_i}
U_\alpha(t^n,x,y)dxdy,$$
with $U_\alpha$ defined by~\eqref{eq:u_alpha}.
A finite volume scheme for solving the system~\eqref{eq:nsml11_d}-\eqref{eq:nsml22_d} is a formula of the form
\begin{equation}
	U^{n+1/2}_i=U_i- \sum_{j\in K_i} \sigma_{i,j}\F_{i,j} - \sigma_i\F_{e,i},
	\label{eq:upU0}
\end{equation}
where using the notations of~\eqref{eq:glo_dis}
\begin{equation}
\sum_{j\in K_i} L_{i,j} \F_{i,j} \approx \int_{C_i}
\nabla_{x,y} . F(U) dx dy,
\label{eq:flux_dis1}
\end{equation}
with
$$\sigma_{i,j} = \frac{\Delta t^nL_{i,j}}{|C_i|},\quad \sigma_{i} =
\frac{\Delta t^nL_{i}}{|C_i|}.$$
Here we consider first-order explicit schemes where
\begin{equation}
	\F_{i,j} = F(U_i,U_{j},z_{b,i}-z_{b,j},{\bf n}_{i,j}).
\label{eq:flux_def}
\end{equation}
and
\begin{equation}
\F_{i,j} = F(U_i,U_j,z_{b,i}-z_{b,j},{\bf n}_{i,j}) = \begin{pmatrix}
F(U_{1,i},U_{1,j},z_{b,i}-z_{b,j},{\bf n}_{i,j})\\
\vdots \\
F(U_{N,i},U_{N,j},z_{b,i}-z_{b,j},{\bf n}_{i,j})
\end{pmatrix}
\label{eq:flux}
\end{equation}
and for the boundary nodes
\begin{equation}
\F_{e,i} =
F(U_i,U_{e,i},{\bf n}_{i}) = \begin{pmatrix}
F(U_{1,i},U_{1,e,i},{\bf n}_{i})\\
\vdots \\
F(U_{N,i},U_{N,e,i},{\bf n}_{i})
\end{pmatrix}.
\label{eq:fluxbis}
\end{equation}
Relation~\eqref{eq:upU0} tells how to compute the
values $U^{n+1/2}_i$ knowing $U_i$ and discretized values $z_{b,i}$ of the
topography. Following~\eqref{eq:flux_dis1}, the term $\F_{i,j}$ in~\eqref{eq:upU0}
denotes an interpolation of the normal component of the flux
$F(U). {\bf n}_{i,j}$ along the edge $C_{i,j}$. The functions
$F(U_i,U_j,z_{b,i} - z_{b,j},{\bf n}_{i,j})\in \R^{2N+1}$ are the numerical fluxes, see~\cite{bouchut_book}.

In the next paragraph we
define $\F(U_i, U_j,z_{b,i}-z_{b,j}, {\bf n}_{i,j})$ using the kinetic
interpretation of the system. The computation of the value $U_{i,e}$,
which denotes a value outside $C_i$ (see Fig.~\ref{fig:mesh}-{\it (b)}), defined such that the boundary
conditions are satisfied, and the definition of the boundary flux
$F(U_i, U_{e,i}, {\bf n_i})$ are described paragraph~\ref{subsec:BC}. Notice that
we assume a flat topography on the boundaries i.e. $z_{b,i} = z_{b,i,e}$.

\subsection{Discrete kinetic equation}
\label{subsec:discrete_kin}

The choice of a kinetic scheme is motivated by several arguments. First, the kinetic
interpretation is a suitable starting point for building a stable numerical scheme.
 We will prove
in paragraph~\ref{subsec:discrete_kin} that the proposed kinetic scheme preserves  positivity
of the water depth and ensures
a discrete local maximum principle for a tracer concentration
(temperature, salinity...).
Second, the construction of the kinetic scheme does not need the computation of
the system eigenvalues.
 This point is very important here since these eigenvalues are not available
 in explicit analytical form, and they are hardly accessible even numerically.
 Furthermore, as previously mentioned,  hyperbolicity of the multilayer
 model may not hold, and  the kinetic scheme allows overcoming this difficulty.

\subsubsection{Without topography}

In a first step we consider a situation with flat
bottom. Following prop.~\ref{lemma:sv_kin_rep}, the
model~\eqref{eq:massesvml}-\eqref{eq:mvtsvml} reduces, for each layer, to a classical Saint-Venant system with
exchange terms and its
kinetic interpretation (see Eq.~\eqref{eq:gibbsbis}) is given by
\begin{equation}
\frac{\partial M_\alpha}{\partial t} + \begin{pmatrix} \xi\\ \gamma
    \end{pmatrix} . \nabla_{x,y} M_\alpha 
 - N_{\alpha+1/2} + N_{\alpha-1/2}
 = Q_{\alpha},\qquad \alpha\in\{1,\ldots,N\},
\label{eq:eq_kin_di}
\end{equation}
with the notations defined in paragraph~\ref{subsec:kin_description}.

Let $C_i$ be a cell, see Fig.~\ref{fig:mesh}. The integral over
$C_i$ of the convective part of the kinetic equation~\eqref{eq:eq_kin_di} gives
\begin{equation}
\int_{C_i} \left(\frac{\partial M_\alpha}{\partial t} + \begin{pmatrix} \xi\\ \gamma
    \end{pmatrix} . \nabla_{x,y} M_\alpha \right) dx dy \approx
|C_i| \frac{\partial M_{\alpha,i}}{\partial t} + \sum_{j\in K_i} \int_{\Gamma_{i,j}} M_{\alpha,i,j} dl,
\label{eq:kin_dis1}
\end{equation}
with $M_{\alpha,i} = M(U_{\alpha,i},\xi,\gamma)$,
${\bf n}_{i,j}$ being the outward normal to the cell $C_i$.
The quantity $M_{\alpha,i,j}$ is defined by the classical kinetic upwinding
\begin{equation*}
M_{\alpha,i,j}
= M_{\alpha,i}\zeta_{i,j}
\1_{\zeta_{i,j}\geq 0}  + M_{\alpha,j} \zeta_{i,j}
\1_{\zeta_{i,j}\leq 0},
\label{eq:kin_dis22}
\end{equation*}
with $\zeta_{i,j} = \begin{pmatrix} \xi & \gamma \end{pmatrix}^T
. {\bf n}_{i,j}$.

Therefore, the kinetic scheme applied for Eq.\eqref{eq:eq_kin_di} is given by
\begin{eqnarray}
f_{\alpha,i}^{n+1/2-}
& = & \Bigl( 1 - \frac{\Delta t^n}{|C_i|} \sum_{j\in K_i} L_{i,j} \zeta_{i,j}
\1_{\zeta_{i,j}\geq 0} \Bigr) M_{\alpha,i} - \frac{\Delta t^n}{|C_i|} \sum_{j\in K_i} L_{i,j} M_{\alpha,j} \zeta_{i,j}
\1_{\zeta_{i,j}\leq 0},
\label{eq:kin_dis3}\\
f_{\alpha,i}^{n+1-} & = & f_{\alpha,i}^{n+1/2-}  +\Delta t^n \Bigl( N_{\alpha+1/2,i}^{n+1-} - N_{\alpha-1/2,i}^{n+1-} \Bigr),
\label{eq:kin_dis3bis}
\end{eqnarray}
with the exchange terms $\{N_{\alpha+1/2,i}^{n+1-}\}_{\alpha=0}^{N}$
defined by
\begin{equation}
N_{\alpha+1/2,i}^{n+1-} (\xi,\gamma)
= \frac{G_{\alpha+1/2,i}}{h_i}f_{\alpha+1/2,i}^{n+1-}.
\label{eq:Ndis}
\end{equation}
Following~\eqref{eq:upwind_uT} we can write
$$f_{\alpha+1/2,i}^{n+1-} =
\left\{\begin{array}{ll}
f_{\alpha,i}^{n+1-}  & \mbox{\rm if } \;G_{\alpha+1/2} \leq 0\\
f_{\alpha+1,i}^{n+1-}  & \mbox{\rm if } \;G_{\alpha+1/2} > 0
\end{array}\right.$$
leading to
$$N_{\alpha+1/2,i}^{n+1-} (\xi,\gamma)
= \frac{|G_{\alpha+1/2,i}|_+}{h_i}f_{\alpha+1,i}^{n+1-} -
\frac{|G_{\alpha+1/2,i}|_-}{h_i}f_{\alpha,i}^{n+1-}.$$
Notice that the previous definition is consistent with~\eqref{eq:Nbis}.
From~\eqref{eq:Qalphabis_kin}, we get
\begin{equation*}
G_{\alpha+1/2,i}=-\frac{1}{|C_i|}\sum_{k=1}^N \Bigl(\sum_{p=1}^\alpha
l_p - \1_{k\leq \alpha}\Bigr)
\sum_{j\in K_i} L_{i,j} \int_{\R^2}\left( M_{k,i}\zeta_{i,j}
\1_{\zeta_{i,j}\geq 0} + M_{k,j}\zeta_{i,j}
\1_{\zeta_{i,j}\leq 0}\right) d\xi d\gamma.
\end{equation*}

By analogy with the computations in~\eqref{eq:kinmom}, we
can recover the macroscopic quantities $U_{\alpha,i}^{n+1}$ at time $t^{n+1}$ by
integration of the relation~\eqref{eq:kin_dis3bis}
\begin{equation}
U_{\alpha,i}^{n+1} = \int_{\R^2} \kxi f_{\alpha,i}^{n+1-} d\xi d\gamma.
\label{eq:update}
\end{equation}

The scheme~\eqref{eq:kin_dis3} and the definition~\eqref{eq:update} allow to complete the definition of the macroscopic scheme~\eqref{eq:upU0},\eqref{eq:flux}, \eqref{eq:fluxbis} with the numerical flux given
by the flux vector splitting formula~\cite{bouchut_BGK}
\begin{eqnarray}
\F_{i,j} & = &  F^+(U_i^n,{\bf n}_{i,j}) + F^-(U_{j}^n,{\bf n}_{i,j}) \nonumber\\
& = & \int_{\R^2} {\cal K} (\xi,\gamma) M_i \zeta_{i,j}\1_{\zeta_{i,j}\geq 0} d\xi
d\gamma +  \int_{\R^2} {\cal K} (\xi,\gamma) M_j \zeta_{i,j}\1_{\zeta_{i,j}\leq 0} d\xi
d\gamma,
\label{eq:flux_split}
\end{eqnarray}
where ${\cal K} (\xi,\gamma)$ is defined in Remark~\ref{rem:kxi} and $M_i = (M_{1,i}, \ldots, M_{N,i})^T$.

Using~\eqref{eq:Ndis}, we rewrite the step~\eqref{eq:kin_dis3bis}
under the form
$$\bigl( {\bf I}_N + \Delta t {\bf G}_{N,i} \bigr) f^{n+1-} = f^{n+1/2-},$$
where $I_N$ is the identity matrix of
size $N$ and ${\bf G}_{N,i}$ is defined by
$$G_{N,i} = \begin{pmatrix}
-\frac{|G_{3/2,i}|_-}{h_{1,i}^{n+1}} &
-\frac{|G_{3/2,i}|_+}{h_{1,i}^{n+1}} & 0 & 0 & \cdots  & 0\\
\frac{|G_{3/2,i}|_-}{h_{2,i}^{n+1}} & \ddots & \ddots & 0 & \cdots &
0\\
0 & \ddots & \ddots & \ddots & 0 & 0\\
\vdots  & 0 & \frac{|G_{\alpha-1/2,i}|_-}{h_{\alpha,i}^{n+1}} &
-\frac{|G_{\alpha+1/2,i}|_- -
    |G_{\alpha-1/2,i}|_+}{h_{\alpha,i}^{n+1}}
& -\frac{|G_{\alpha+1/2,i}|_+}{h_{\alpha,i}^{n+1}} & 0 \\
\vdots & \ddots & 0 & \ddots & \ddots & -\frac{|G_{N-1/2,i}|_+}{h_{N-1,i}^{n+1}}\\
0 & \cdots & 0 & 0 & \frac{
    |G_{N-1/2,i}|_-}{h_{N,i}^{n+1}} & \frac{
    |G_{N-1/2,i}|_+}{h_{N,i}^{n+1}}
\end{pmatrix}.$$
Hence, the resolution of the discrete kinetic
equation~\eqref{eq:kin_dis3bis} requires to inverse the matrix
$$\begin{pmatrix}
{\bf I}_{N} + \Delta t {\bf G}_{N,i} & 0\\
0 & {\bf I}_{N} + \Delta t {\bf G}_{N,i}
\end{pmatrix}$$
and we have the following lemma.
\begin{lemma}
The matrix ${\bf I}_N + \Delta t {\bf G}_{N,i}$
\begin{itemize}
\item[{\it (i)}] is invertible for any $h_i^{n+1}>0$,
\item[{\it (ii)}] $({\bf I}_N + \Delta t {\bf G}_{N,i})^{-1}$ has only
  positive coefficients,
\item[{\it (iii)}] for any vector $T$ with non negative entries
  i.e. $T_\alpha \geq 0$, for $1\leq \alpha\leq N$, one has
$$\| ({\bf I}_N + \Delta t {\bf G}_{N,i})^{-t} T \|_\infty \leq
\|T\|_\infty.$$
\end{itemize}
\label{lem:inverse}
\end{lemma}

\begin{proof}[Proof of lemma~\ref{lem:inverse}]
{\it (i)} For any $h_i^{n+1}>0$, the matrix ${\bf I}_N + \Delta t {\bf G}_{N,i}$
is a strictly dominant diagonal matrix and hence it is invertible.

{\it (ii)} Denoting ${\bf G}_{N,i}^d$ (resp. ${\bf G}_{N,i}^{nd}$) the diagonal
(resp. non diagonal) part of ${\bf G}_{N,i}$ we can write
$${\bf I}_N + \Delta t {\bf G}_{N,i} = ({\bf I}_N + \Delta t {\bf
  G}_{N,i}^d) \left( {\bf I}_N - ({\bf I}_N + \Delta t {\bf
    G}_{N,i}^d)^{-1} (-\Delta t G_{N,i}^{nd})\right),$$
where all the entries of the matrix ${\bf J}_{N,i} = ({\bf I}_N + \Delta t {\bf
    G}_{N,i}^d)^{-1} (-\Delta t {\bf G}_{N,i}^{nd})$,
are non negative and less than 1. And hence, we can write
$$({\bf I}_N + \Delta t {\bf G}_{N,i})^{-1} = \sum_{k=0}^\infty J_{N,i}^k,$$
proving all the entries of $({\bf I}_N + \Delta t {\bf G}_{N,i})^{-1}$
are non negative.

{\it (ii)} Let us consider the vector ${\bf 1}$ whose entries are all
equal to 1. Since we have
$$({\bf I}_N + \Delta t {\bf G}_{N,i})^t {\bf 1} = {\bf 1},$$
we also have ${\bf 1} = ({\bf I}_N + \Delta t {\bf G}_{N,i})^{-t} {\bf 1}$.
Now let $T$ be a vector whose entries $\{T_\alpha\}_{1 \leq
  \alpha \leq N}$ are non negative, then
$$({\bf I}_N + \Delta t {\bf G}_{N,i})^{-t} {\bf T} \leq ({\bf I}_N +
\Delta t {\bf G}_{N,i})^{-t} {\bf 1} \|{\bf T}\|_\infty = {\bf 1}
\|{\bf T}\|_\infty,$$
that completes the proof.
\end{proof}

\subsubsection{With topography}

The hydrostatic reconstruction scheme (HR scheme for short) for the Saint-Venant system
has been introduced in \cite{bristeau1} in the 1d case
and described in 2d for unstructured meshes in~\cite{bristeau}. The HR
in the context of the kinetic description for the Saint-Venant system
has been studied in~\cite{JSM_entro}.

In order to take into account the topography source and to preserve
relevant equilibria, the HR leads to a modified version
of~\eqref{eq:upU0} under the form
\begin{equation}
	U^{n+1/2}_i=U^n_i- \sum_{j\in K_i} \sigma_{i,j}\F_{i,j}^* -
        \sigma_i\F_{i,e} + \sum_{j\in K_i} \sigma_{i,j} {\cal S}_{i,j}^*,
	\label{eq:upU0_HR}
\end{equation}
where
\begin{equation}
\F^*_{i,j} = F(U_{i,j}^*,U_{j,i}^*,{\bf n}_{i,j}),\quad \mathcal{S}_{i,j}^* =
S(U_i,U_{i,j}^*,{\bf n}_{i,j}) = \begin{pmatrix}
0\\
\frac{g}{2}l_1 (h_{i,j}^{*2} - h_i^2) {\bf n}_{i,j}\\
\vdots\\
\frac{g}{2}l_N (h_{i,j}^{*2} - h_i^2) {\bf n}_{i,j}
\end{pmatrix},
\label{eq:flux_HR}
\end{equation}
with
\begin{equation}
\begin{array}{l}
z_{b,i,j}^* = \max(z_{b,i},z_{b,j}), \quad h_{i,j}^* =
\max(h_i+z_{b,i}-z_{b,i,j}^*,0),\\
U_{i,j}^* = (h_{i,j}^*,l_1 h_{i,j}^* u_{1,i},\ldots, l_N h_{i,j}^* u_{N,i}, l_1 h_{i,j}^* v_{1,i},\ldots, l_N h_{i,j}^* v_{N,i})^T.
\end{array}
\label{eq:state_HR}
\end{equation}

We would like here to propose a kinetic interpretation of the HR scheme, which means to
interpret the above numerical fluxes as averages with respect to the kinetic variables
of a scheme written on a kinetic function $f$.
More precisely, we would like to approximate the solution to \eqref{eq:gibbsbis} by a kinetic scheme
such that the associated macroscopic scheme is exactly \eqref{eq:upU0_HR}-\eqref{eq:flux_HR}
with homogeneous numerical flux $\F$ given by~\eqref{eq:flux_split}.
We denote $M_{\alpha,i,j}^*=M(U_{\alpha,i,j}^*,\xi,\gamma)$ for any $\alpha=1,\ldots,N$
and we consider the scheme
\begin{eqnarray}
f_{\alpha,i}^{n+1/2-} & = & M_{\alpha,i} - \frac{\Delta t^n}{|C_i|} \sum_{j\in K_i} L_{i,j} \zeta_{i,j}
\1_{\zeta_{i,j}\geq 0}  M_{\alpha,i,j}^* - \frac{\Delta t^n}{|C_i|} \sum_{j\in K_i} L_{i,j} M_{\alpha,j,i}^* \zeta_{i,j}
\1_{\zeta_{i,j}\leq 0},\nonumber\\
& & - \frac{\Delta t^n}{|C_i|} \sum_{j\in K_i} L_{i,j} (M_{\alpha,i}-M_{\alpha,i,j}^*) \theta_{\alpha,i,j},
\label{eq:kin_dis3_HR1}\\
f_{\alpha,i}^{n+1-} & = & f_{\alpha,i}^{n+1/2-} +\Delta t^n \Bigl( N_{\alpha+1/2,i}^{*,n+1-} - N_{\alpha-1/2,i}^{*,n+1-} \Bigr),
\label{eq:kin_dis3_HR2}
\end{eqnarray}
where
$$\theta_{\alpha,i,j} = \begin{pmatrix}
\xi - u_{\alpha,i}\\
\gamma - v_{\alpha,i}
\end{pmatrix}. {\bf n}_{i,j}.$$
For the exchange terms, by analogy with~\eqref{eq:Ndis} we define
\begin{equation}
N_{\alpha+1/2,i}^{*,n+1-} (\xi,\gamma)
= \frac{G_{\alpha+1/2,i}^*}{h_i}f_{\alpha+1/2,i}^{n+1-},
\label{eq:Ndis_HR}
\end{equation}
and using~\eqref{eq:Qalphabis_kin} we get
$$
G_{\alpha+1/2,i}^* =-\frac{1}{|C_i|}\sum_{k=1}^N \Bigl(\sum_{p=1}^\alpha
l_p - \1_{k\leq \alpha}\Bigr)
\sum_{j\in K_i} L_{i,j} \int_{\R^2}\left( M_{k,i,j}^*\zeta_{i,j}
\1_{\zeta_{i,j}\geq 0} + M_{k,i,j}^*\zeta_{i,j}
\1_{\zeta_{i,j}\leq 0}\right) d\xi d\gamma.
$$
It is easy to see that in the previous formula, we have the moment relations
\begin{eqnarray}
& & 	\int_{\R^2} (M_{\alpha,i}-M_{\alpha,i,j}^*) \theta_{\alpha,i,j} d\xi d\gamma =0,\label{eq:intdeltaM1-}\\
& & \int_{\R^2} \begin{pmatrix}
\xi \\
\gamma
\end{pmatrix} (M_{\alpha,i}-M_{\alpha,i,j}^*) \theta_{\alpha,i,j} d\xi
d\gamma = \frac{g}{2}l_\alpha (h_{i,j}^{*2} - h_i^2) {\bf n}_{i,j},
\label{eq:intdeltaM2-}
\end{eqnarray}
Using again~\eqref{eq:update},
the integration of the set of
equations~\eqref{eq:kin_dis3_HR1}-\eqref{eq:kin_dis3_HR2}, for $\alpha=1,\ldots,N$,
multiplied by ${\cal K}(\xi,\gamma)$
with respect to $\xi$,$\gamma$ then gives the HR scheme \eqref{eq:upU0_HR}-\eqref{eq:flux_HR}
with~\eqref{eq:flux_split},\eqref{eq:state_HR}.
Thus as announced, \eqref{eq:kin_dis3_HR1}-\eqref{eq:kin_dis3_HR2} is a kinetic interpretation of
the HR scheme in 3d for an unstructured mesh.

There exists a velocity $v_m\geq 0$ such that for all $\alpha,i$,
\begin{equation}
	|\xi|\geq  v_m\
        \mbox{or } |\gamma|\geq  v_m  \Rightarrow M(U_{\alpha,i},\xi,\gamma)=0.
	\label{eq:vmax}
\end{equation}
This means equivalently that $|u_{\alpha,i}|+|v_{\alpha,i}|+\sqrt{2gh_i}\leq v_m$.
We consider a CFL condition strictly less than one,
\begin{equation}
	\sigma_{i} v_m\leq
        \beta < \frac{1}{2} \quad \mbox{ for all }i,
	\label{eq:CFLfull}
\end{equation}
where $\sigma_{i}=\Delta t^n \sum_{j\in K_i} L_{i,j}/|C_i|$, and $\beta$ is a given
constant.

Then the following
proposition holds.
\begin{proposition}
Under the CFL condition~\eqref{eq:CFLfull}, the scheme~\eqref{eq:kin_dis3_HR1}-\eqref{eq:kin_dis3_HR2} verifies the following properties.

\noindent (i) The macroscopic scheme derived
from~\eqref{eq:kin_dis3_HR1}-\eqref{eq:kin_dis3_HR2} using~\eqref{eq:update} is a consistent
discretization of the layer-averaged Euler system~\eqref{eq:massesvml}-\eqref{eq:mvtsvml}.

\noindent (ii) The kinetic function remains nonnegative i.e.
$$f^{n+1-}_{\alpha,i}\geq
0, \qquad \forall (\xi,\gamma)\in\R^2,\;\forall i,\ \forall \alpha.$$

\noindent (iii) The scheme~\eqref{eq:kin_dis3_HR1}-\eqref{eq:kin_dis3_HR2} is kinetic well balanced i.e. at rest
\begin{equation}
f_{\alpha,i}^{n+1-}  =  M_{\alpha,i},\quad \forall
(\xi,\gamma)\in\R^2,\;\forall i,\ \forall \alpha=1,\ldots,N.
\label{eq:rest}
\end{equation}
\label{prop:dis_kin_HR}
\end{proposition}

\begin{proof}[Proof of prop.~\ref{prop:dis_kin_HR}]

\noindent (i) Since the Boltzmann type equations~\eqref{eq:gibbsbis}
are almost linear transport equations with source terms, the discrete kinetic scheme~\eqref{eq:kin_dis3_HR1}-\eqref{eq:kin_dis3_HR2} is
clearly a consistent discretization of~\eqref{eq:gibbsbis}. And therefore
using the kinetic interpretation given in prop.~\ref{prop:kinetic_sv_mcl}, the macroscopic scheme obtained
from~\eqref{eq:kin_dis3_HR1}-\eqref{eq:kin_dis3_HR2} using~\eqref{eq:update} is a consistent
discretization of the layer-averaged Euler system~\eqref{eq:massesvml}-\eqref{eq:mvtsvml}.

\noindent (ii) In~\eqref{eq:kin_dis3_HR1}-\eqref{eq:kin_dis3_HR2} we have
$$\frac{\Delta t^n}{|C_i|} \sum_{j\in K_i} L_{i,j} M_{\alpha,j,i}^* \zeta_{i,j}
\1_{\zeta_{i,j}\leq 0} \leq 0,$$
and the HR~\eqref{eq:state_HR} ensures $M_{\alpha,i,j}^* \leq M_{\alpha,i}$, $\forall
(\xi,\gamma)\in\R^2,\forall \alpha$
leading to
$$
f_{\alpha,i}^{n+1/2-} \geq \left( 1 - \frac{\Delta t^n}{|C_i|}
  \sum_{j\in K_i} L_{i,j} \bigl( \zeta_{i,j}
\1_{\zeta_{i,j}\geq 0} + \theta_{\alpha,i,j}
\1_{\theta_{\alpha,i,j}\geq 0}\bigr)\right) M_{\alpha,i},\nonumber
$$
But $\zeta_{i,j}\1_{\zeta_{i,j}\geq 0} \leq \max\{|\xi|,|\gamma|\}$,
$\theta_{\alpha,i,j}\1_{\theta_{\alpha,i,j}\geq 0} \leq
\max\{|\xi - u_{\alpha,i}|,|\gamma - v_{\alpha,i}|\}$
and therefore
$$
\frac{\Delta t^n}{|C_i|}\sum_{j\in K_i} L_{i,j} \bigl( \zeta_{i,j}
\1_{\zeta_{i,j}\geq 0} + \theta_{\alpha,i,j}
\1_{\theta_{\alpha,i,j}\geq 0}\bigr)
\leq \sigma_i
(\max\{|\xi|,|\gamma|\} + \max\{|\xi -
u_{\alpha,i}|,|\gamma - v_{\alpha,i}|\} )
\leq 1,$$
where~\eqref{eq:vmax},\eqref{eq:CFLfull} have been used, proving
$f_{\alpha,i}^{n+1/2-} \geq 0$ for any $\xi\in\R$ and any
$\alpha\in\{1,\ldots,N\}$. Now using the results of lemma~\ref{lem:inverse}, it
ensures that $f_{\alpha,i}^{n+1-}$ defined by~\eqref{eq:kin_dis3_HR2} satisfies
$f_{\alpha,i}^{n+1-} \geq 0$ for any $(\xi,\gamma)\in\R^2$ and any
$\alpha\in\{1,\ldots,N\}$, proving {\it (ii)}.

\noindent (iii) Considering the situation at rest
i.e. $u_{\alpha,i}=v_{\alpha,i}=0$, $\forall \alpha,i$ and
$h_i+z_{b,i}=h_j+z_{b,j}$, $\forall i,j$ we have
$$M_{\alpha,i} = M_{\alpha,i,j}^*,\qquad \forall \alpha,i,j.$$
From~\eqref{eq:kin_dis3_HR1}-\eqref{eq:kin_dis3_HR2}, this gives~\eqref{eq:rest}.
\end{proof}

\subsection{Macroscopic scheme}
\label{subsec:macro_scheme}

The numerical scheme for the
system~\eqref{eq:nsml11_d}-\eqref{eq:nsml33_d} is given
by~\eqref{eq:update},\eqref{eq:kin_dis3_HR1}, \eqref{eq:kin_dis3_HR2} and requires to calculate fluxes having the form
\begin{eqnarray*}
F_{sv}(U) & = & \left(\begin{array}{c}
F_h\\ F_{hu}\\ F_{hv}
\end{array}\right)
= \int_{n_x\xi+n_y\gamma\geq 0}
\left(\begin{array}{c}
1\\
\xi\\
\gamma
\end{array}\right)
(n_x\xi+n_y\gamma) M(U,\xi,\gamma) d\xi d\gamma.
\end{eqnarray*}
with $M$ given by~\eqref{eq:M1d}, $n_x$ and $n_y$ being the components of a normal unit vector
${\bf n}$.
Defining the change of variables
$$\xi = u + c z_1,\qquad \gamma = v + c z_2,$$
we can write
$$
F_{sv}(U) = h \int_{n_x(cz_1+u)+n_y(cz_2+v)\geq 0 } (n_x(cz_1+u)+n_y(cz_2+v))
\left(\begin{array}{c}
1\\
u+c z_1\\
v + c z_2
\end{array}\right) \chi_0(z_1,z_2)
dz_1 dz_2,$$
where $\chi_0$ is defined by~\eqref{eq:chi0}. A second change of variables $y_1 = n_x z_1 + n_y z_2$, $y_2 = n_x z_2 - n_y z_1$, $\tilde{u} = n_x u + n_y v$ gives
\begin{equation}
F_{sv}(U) = h \int_{\{y_1 \geq - \frac{\tilde{u}}{c}\} \times \R} (\tilde{u}+cy_1)
\left(\begin{array}{c}
1\\
u+c n_x y_1\\
v + c n_y y_1
\end{array}\right) \chi_0(y_1,y_2)
dy_1 dy_2,\label{eq:flu_hc}
\end{equation}
since $\chi_0$ is odd. The details of the computations of
formula~\eqref{eq:flu_hc} is given in appendix B.

Using the properties obtained at the kinetic level for the resolution
of the
system~\eqref{eq:massesvml}-\eqref{eq:mvtsvml}, the following
proposition holds.

\begin{proposition}
Under the CFL condition~\eqref{eq:CFLfull}, the
scheme~\eqref{eq:update},\eqref{eq:kin_dis3_HR1},
\eqref{eq:kin_dis3_HR2} satisfies the following properties.

\noindent (i) The macroscopic scheme derived
from~\eqref{eq:kin_dis3_HR1}-\eqref{eq:kin_dis3_HR2} using~\eqref{eq:update} is a consistent
discretization of the layer-averaged Euler system~\eqref{eq:massesvml}-\eqref{eq:mvtsvml}.

\noindent (ii) The water depth remains nonnegative i.e.
$$h^{n+1}_{i}\geq
0, \quad \forall i,\quad \mbox{when}\; h^{n}_{i}\geq
0\quad \forall i.$$

\noindent (iii) The
scheme~\eqref{eq:update},\eqref{eq:kin_dis3_HR1},
\eqref{eq:kin_dis3_HR2} is well-balanced.
\label{prop:prop_scheme_macro}
\end{proposition}

\begin{proof}[Proof of prop.~\ref{prop:prop_scheme_macro}]
The proof is similar to the one given in prop.~\ref{prop:dis_kin_HR}.
\end{proof}

\subsection{The discrete layer-averaged Navier-Stokes system}
\label{subsec:NS_dis}

In this paragraph, we detail the space discretization of the viscous
terms. Several expressions have been obtained for the viscous terms,
see paragraph~\ref{subsec:NS}. In this paragraph, we give a numerical scheme for the
model~\eqref{eq:nsml3_d}, rewriting it under the form
\begin{equation}
(h_{\alpha}{\bf u}_{\alpha})^{n+1} = \widetilde{h_{\alpha}{\bf
  u}_{\alpha}} + \Delta t^n S_{v,f}(U),
\label{eq:viscous_terms}
\end{equation}
with $S_{v,f}(U) = (S_{v,f,1},\ldots,S_{v,f,N})^T$ and
\begin{eqnarray*}
\widetilde{h_{\alpha}{\bf u}_{\alpha}} & = & h_{\alpha}{\bf u}_{\alpha} -
\Delta t^n \Bigl( \nabla_{x,y} .\left(
  h_{\alpha} {\bf u}_{\alpha} \otimes {\bf u}_{\alpha} \right) -
\nabla_{x,y}  \bigl(\frac{g}{2} h h_{\alpha} \bigr)  - gh_\alpha \nabla_{x,y} z_b \nonumber\\
& & + {\bf u}_{\alpha+1/2}^{n+1} G_{\alpha+1/2} - {\bf u}_{\alpha-1/2}^{n+1}
G_{\alpha-1/2} \Bigr),\label{eq:nsml2_d1}\\
S_{v,f,\alpha} & = & \nabla_{x,y} .
\bigl( h_\alpha {\bf \Sigma}_\alpha^0 \bigr) + \Gamma_{\alpha+1/2} ({\bf u}_{{\alpha}+1}-{\bf
    u}_{\alpha})- \Gamma_{\alpha-1/2} ({\bf u}_{\alpha}-{\bf u}_{{\alpha}-1}) - \kappa_{\alpha} {\bf u}_{\alpha} + W_\alpha {\bf t}_s,
\end{eqnarray*}
with the definitions~\eqref{eq:sigma_mxx_1},\eqref{eq:sigma_mxy_1},\eqref{eq:sigma_alpha_1},
\eqref{eq:sigma_m_1},\eqref{eq:lambda_12} for
${\bf T}^0_{\alpha}$, $\Gamma_{\alpha\pm 1/2}$. It remains to give a fully discrete scheme for the viscous and friction
terms $\{S_{v,f,\alpha}\}$.

The discretization of~\eqref{eq:viscous_terms} is done using a finite
element / finite difference approximation obtained as follows. We depart form the
triangulation defined in paragraph~\eqref{subsec:triangulation} and we
use the cells values of the variables~-- inherited from the finite
volume framework~-- to define a $\mathbb{P}_1$ approximation of the variables.

Notice that, compared to the advection and pressure terms, the discretization of
the viscous terms raises less difficulties and we propose a stable
scheme that will be extended to more general rheology
terms~\cite{BDGSM} and more completely analyzed in a forthcoming paper.

Using a classical $\mathbb{P}_1$ finite element type approximation
with mass lumping of
Eq.~\eqref{eq:viscous_terms}, we get
\begin{eqnarray}
{\bf U}_\alpha^{n+1} & = &  \widetilde{{\bf U}}_\alpha -\Delta t^n \Bigl( {\cal K}_{\alpha+1} {\bf U}_{\alpha+1}
+ {\cal
  K}_{\alpha} {\bf U}_\alpha + {\cal K}_{\alpha-1} {\bf U}_{\alpha-1} \Bigr)\nonumber\\
& & + \Delta t^n
{\cal G}_{\alpha+1/2} ({\bf U}_{\alpha+1}-{\bf U}_{\alpha}) - \Delta
    t^n {\cal G}_{\alpha-1/2} ({\bf U}_{\alpha}-{\bf U}_{\alpha-1}) - \Delta t^n\kappa_\alpha
{\bf U}_\alpha + \Delta t^n W_\alpha {\bf t}_s,
\label{eq:viscous_fe}
\end{eqnarray}
with the matrices
\begin{eqnarray*}
{\cal K}_{\alpha,ji} & = & \frac{\nu}{2}\int_\Omega \left(\frac{h_{\alpha,j}}{h_{\alpha+1,j}
  + h_{\alpha,j}} + \frac{h_{\alpha,j}}{h_{\alpha,j}  + h_{\alpha-1,j}}\right)\nabla_{x,y}
\varphi_i . \nabla_{x,y} \varphi_j\ dx dy,\\
{\cal K}_{\alpha\pm 1,ji} & = & \frac{\nu_{\alpha\pm 1/2}}{2}\int_\Omega \frac{h_{\alpha\pm 1,j}}{h_{\alpha+1,j}  + h_{\alpha,j}} \nabla_{x,y}
\varphi_i . \nabla_{x,y} \varphi_j\ dx dy,\\
{\cal G}_{\alpha+1/2,ji} & = & \nu_{\alpha+1/2} \int_\Omega \frac{1 + |\nabla_{x,y}
  z_{\alpha+1/2}|^2}{h_{{\alpha}+1}+h_{\alpha}}
\varphi_i .
\varphi_j dx dy,
\end{eqnarray*}
where $\varphi_i$, $\varphi_j$ are the basis functions. We have presented an explicit
in time version of~\eqref{eq:viscous_fe} that is stable under a
classical CFL condition. An implicit or semi-implicit version
of~\eqref{eq:viscous_fe} can also
be used.

The main purpose of this paper is to propose a stable and robust
numerical approximation of the incompressible Euler system with free
surface. Voluntarily, we give few details
concerning the numerical approximation of the dissipative terms:
\begin{itemize}
\item the viscous and friction terms are dissipative and hence a
  reasonable approximation leads to a stable numerical scheme.
\item In this paper, we consider a simplified Newtonian rheology for
  the fluid, the numerical approximation of the
  general (layer-averaged) rheology~\cite{BDGSM} will be
  studied in a forthcoming paper.
\end{itemize}

\subsection{Boundary conditions}
\label{subsec:BC}

The contents of this paragraph
slightly differ from previous works of one of the
authors~\cite{coussin} and valid for the classical Saint-Venant system. First,  we focus on the boundary conditions for the layer-averaged Euler system
i.e. the system~\eqref{eq:nsml1_d}-\eqref{eq:nsml2_d} for $\nu=0$, $\kappa=0$
and then for the viscous part.

\subsubsection{Layer-averaged Euler system}

In this paragraph we detail the computation of the boundary flux ${\cal F}({\bf U}_i,{\bf
U}_{e,i},{\bf n}_i)$
appearing in~\eqref{eq:upU0},\eqref{eq:flux},\eqref{eq:fluxbis}.
The variable ${\bf U}_{i,e}^n$ can be interpreted as an approximation
of the solution in a ghost cell adjacent to the boundary. As before we
introduce the vector
$$U_{i,e} = \left(h_{i,e}^n, (hu)_{1,i,e},\ldots,(hu)_{N,i,e},(hv)_{1,i,e},\ldots,(hv)_{N,i,e}\right)^T,$$
and we will use the flux vector splitting form associated to the
kinetic formulation~\eqref{eq:flux_split}
\begin{equation}
  {\cal F}(U_i,U_{i,e},{\bf n}_{i}) = F^+(U_i,{\bf n}_{i}) + F^-(U_{i,e},{\bf n}_{i})
\label{frie}
\end{equation}
with $U_{i,e}^n$ defined according to the boundary type.

\paragraph{Solid wall}
If we consider a node $i_0$ belonging to a solid wall, we prescribe a slip condition written
\begin{equation}
  {\bf u}_{\alpha} \cdot {\bf n}_{i_0} = 0,
\label{slcond}
\end{equation}
for $\alpha=1,\ldots,N$. We assume the continuity of the water depth
$h_{i_0,e}=h_{i_0}$ and of the tangential
component of velocity.

From~\eqref{eq:flu_hc} with~\eqref{slcond} we obtain
$$F_h^+(U_{\alpha,i_0})+F_h^-(U_{\alpha,i_0,e}) = 0,$$
and
\begin{equation*}
\begin{pmatrix}
F_{hu}^+(U_{\alpha,i_0})+F_{hu}^-(U_{\alpha,i_0,e}) \\
F_{hv}^+(U_{\alpha,i_0})+F_{hv}^-(U_{\alpha,i_0,e}) \end{pmatrix} . {\bf n}_{i_0}=
\frac{gh_{\alpha,i_0} h_{i_0}}{2},\quad
\begin{pmatrix}
F_{hu}^+(U_{\alpha,i_0})+F_{hu}^-(U_{\alpha,i_0,e}) \\
F_{hv}^+(U_{\alpha,i_0})+F_{hv}^-(U_{e,\alpha,i_0}) \end{pmatrix} . {\bf t}_{i_0} = 0,
\label{Fslcond}
\end{equation*}
for $\alpha=1,\ldots,N$ for a vector ${\bf t}_{i_0}$ orthogonal to
${\bf n}_{i_0}$. The condition (\ref{slcond}) is therefore prescribed
weakly but a posteriori, in order to be sure
that~\eqref{slcond} is satisfied, we can apply
$${\bf u}_{\alpha,i_0,e} = {\bf u}_{\alpha,i_0}  - ({\bf
  u}_{\alpha,i_0} . {\bf n}_{i_0}) {\bf n}_{i_0}.$$

\paragraph{Fluid boundary}
Even if the considered model is more complex than the Shallow water
system, we can consider that the type of the flow depends, for each layer, on the value of
the Froud number $Fr_\alpha=|{\bf u}_\alpha|/\sqrt{gh}$,
a flow is said torrential, for $|{\bf u}_\alpha| >\sqrt{gh}$ and fluvial,  for $|{\bf u}_\alpha| <\sqrt{gh}$.

Generally, for the fluid boundaries, the conditions
prescribed by the
user depend on the type of the flow defined by this criterion.



 We have also to notice that with ${\bf n}$ the outward unit normal to the
 boundary edge, an inflow boundary corresponds to $ {\bf u}_\alpha \cdot {\bf
n} <0$ and an
outflow one to $ {\bf u}_\alpha \cdot {\bf n} >0$.

We will treat the following cases:
for a fluvial flow boundary, we distinguish the cases where the flux
or the water depth are given, while for a torrential flow we
distinguish the inflow or outflow boundaries.

\paragraph{Fluvial boundary. Flux given}
 We consider first a fluvial boundary, so we assume that
\begin{equation}
|{\bf u}_\alpha| < \sqrt{gh}.
\label{fluv}
\end{equation}
If for each layer, the flux ${\bf q}_{g,\alpha}$ is given, then we wish to impose
\begin{equation}
\left( F_h({\bf U}_{\alpha,i}) +  F_h({\bf U}_{e,\alpha,i}) \right)
 . {\bf n}_i =
 {\bf q}_{g,\alpha} . {\bf n}_i,\quad F_h({\bf U}_{e,\alpha,i}) . {\bf t}_i = q_{g,\alpha} . {\bf t}_i,\label{flqn}\\
\end{equation}
with ${\bf n}_i . {\bf t}_i = 0$. The value ${\bf q}_{g,\alpha}$
depends on the value of the prescribed flux along the vertical axis.

If one directly imposes~\eqref{flqn}, it leads to instabilities (especially because the
numerical values are not necessarily in the regime of validity of this
condition). We propose to discretize it in a weak form.
We denote
\begin{equation}
a_1={\bf q}_{g,\alpha} . {\bf n}_i - F_h({\bf U}_{\alpha,i})  . {\bf n}_i.
\label{a1}
\end{equation}
If $a_1\geq 0$, we prescribe
$$F_h({\bf U}_{e,\alpha,i})=0,\quad F_{hu}({\bf
  U}_{e,\alpha,i})=0,\quad\mbox{and}\quad F_{hv}({\bf U}_{e,\alpha,i})=0.$$
If $a_1< 0$, we have to write a third equation to be able to compute
the three components of ${\bf U}_{e,\alpha}$ and by analogy with what
is done for the Saint-Venant system~-- where the Riemann invariant
related to the outgoing characteristic is preserved~-- we assume the quantity
${\bf u}_\alpha . {\bf n}$ is constant though the interface, i.e.
\begin{equation}
{\bf u}_{e,\alpha,i} . {\bf n}_i - 2 \sqrt{gh_{e,i}} = {\bf
  u}_{\alpha,i} . {\bf n}_i - 2 \sqrt{gh_i}.
\label{Rout}
\end{equation}
As~\eqref{fluv} is satisfied, the eigenvalue $u_{e,\alpha,i} . {\bf
  n}_i - 2 \sqrt{gh_{e,i}}$ is positive.

We use the equations~\eqref{flqn} and~\eqref{Rout} to compute
$h_{e,i}$ and ${\bf u}_{e,\alpha,i} . {\bf n}_i$. We denote
$a_2 = {\bf u}_{\alpha,i} . {\bf n}_i - 2 \sqrt{g h_i}$
and
\begin{equation}
m = \frac{{\bf u}_{e,\alpha,i} . {\bf
  n}_i}{\sqrt{gh_{e,i}}}.
\label{defm}
\end{equation}
 Then the equation~(\ref{Rout}) gives
\begin{equation}
\sqrt{gh_{e,i}}\ (m-2)=a_2
\label{a2m}
\end{equation}
and using the definition of $F_h$
(see~\eqref{eq:flux_h_expr}) with~\eqref{flqn},\eqref{a1}  we have
\begin {equation}
\frac{h_{e,i}}{\pi}\int_{z\leq \frac{-{\bf u}_{e,\alpha,i}.{\bf n}_i}{\sqrt{\frac{gh_{e,i}}{2}}}}
\left( {\bf u}_{e,\alpha,i} . {\bf n}_i + \sqrt{\frac{gh_{e,i}}{2}} z \right)
\sqrt{1 - \frac{z^2}{4}} dz = a_1,
\label{a11}
\end {equation}
or using~\eqref{Rout}
\begin {equation}
\psi (h_{e,i}) = a_1,
\label{a12}
\end {equation}
with
$$\psi (h_{e,i}) = \frac{h_{e,i}}{\pi} \int_{z\leq\,\frac{-(2\sqrt{gh_{e,i}} + a_2)}{\sqrt{\frac{gh_{e,i}}{2}}}}
\left( 2\sqrt{gh_{e,i}} + a_2 + \sqrt{\frac{gh_{e,i}}{2}} z \right)
\sqrt{1 - \frac{z^2}{4}} dz
.$$
It is easy to see that $h \mapsto \psi (h)$ is a growing function
of $h$ with $\psi (0)=0$ and $\psi (+\infty)=+\infty$ and
therefore, Eq.~\eqref{a11} admits a unique solution for any $a_1 >
0$. Using~\eqref{a2m}, Eq.~\eqref{a12} is equivalent
to solve for $m$
\begin {equation}
\Psi(m)=a_2,
\label{psia2}
\end{equation}
with
$$\Psi (m) = K\frac{m-2}{\phi(m)^{1/3}},$$
and $K = \bigl( {\sqrt{2} g a_1 } \bigr)^{1/3}$
$$\phi (m) = \frac{1}{\pi}\int_{z\leq -\sqrt{2}m}
(\sqrt{2}m + z)
\sqrt{1-\frac{z^2}{4}} dz.$$
In practice, we use a Newton-Raphson
algorithm to solve an equivalent form of Eq.~\eqref{psia2}, namely
$$m-2 - \frac{a_2}{K}\phi(m)^{1/3} = 0.$$
Once the above equation has been solved, from~\eqref{defm},\eqref{a2m} we deduce
\begin{equation*}
h_{e,i} = \frac{1}{g}\bigl(\frac{a_2}{m-2}\bigr)^2,
\quad
{\bf u}_{e,\alpha,i} . {\bf n}_i = \frac{a_2 m}{m-2} = m\sqrt{g h_{e,i}}.
\end{equation*}
\begin{remark}
Notice that in the procedure proposed to calculate ${\bf
  u}_{e,\alpha,i}$, $h_{e,i}$, even if $h_{e,i}$ represents a total
water depth, a different value of $h_{e,i}$ is
calculated for each layer $\alpha$. $h_{e,i}$ is only used to ensure~\eqref{flqn}.
\end{remark}

\paragraph{ Fluvial boundary. Water depth given}
 We verify that the flow is actually fluvial, i.e.
\begin {equation}
({\bf u}_{\alpha,i}. {\bf n}_i-\sqrt{gh}_i)({\bf u}_{\alpha,i}. {\bf n}_i+\sqrt{gh}_i)\leq 0.
\label {vpfluv}
\end {equation}
Since the water depth is given, we write
\begin {equation}
h_{e,i}=h_{g,i}.
\label{hg}
\end {equation}
We assume the continuity of the tangential component
\begin {equation}
(h{\bf u})_{e,\alpha,i} . {\bf t}_i = (h{\bf u})_{\alpha,i} . {\bf t}_i,
\end {equation}
with ${\bf t}_i . {\bf n}_i =0$. To define completely ${\bf u}_{e,\alpha,i}$, we assume, as in the previous case, that
 the Riemann invariant is constant along
the outgoing characteristic (\ref{Rout}), so we obtain
\begin {equation}
{\bf u}_{e,\alpha,i}. {\bf n}_i = {\bf u}_{\alpha,i}. {\bf n}_i+2 \sqrt{g}(\sqrt{h_i}-\sqrt{h_{g,i}}).
\end {equation}
 Sometimes it appears that the numerical values do not satisfy
 the condition (\ref{vpfluv}), then the flow
 is in fact torrential and
\begin{itemize}
\item[$\bullet$] if ${\bf u}_{\alpha,i}. {\bf n}_i > 0$ , the condition (\ref{hg}) cannot be
satisfied (see Sec. 6.2.4),
\item[$\bullet$] if ${\bf u}_{\alpha,i}. {\bf n}_i < 0$ , one
  condition is missing and we
 prescribe ${\bf u}_{e,\alpha,i}. {\bf n}_i={\bf u}_{\alpha,i}. {\bf
   n}_i$.
\end{itemize}
\paragraph{Torrential inflow boundary}
 For a torrential inflow boundary we assume that the water depth
and the flux are given, then we prescribe
$$h_{e,i}=h_{g,i},\quad
(h{\bf u})_{e,\alpha,i} . {\bf t}_i = (h{\bf u})_{g,\alpha,i} . {\bf t}_i,
$$
and
$$
\left( F_h({\bf U}_{\alpha,i}) +  F_h({\bf U}_{e,\alpha,i}) \right)
 . {\bf n}_i =
 {\bf q}_{g,\alpha} . {\bf n}_i = (h{\bf u})_{g,\alpha,i} . {\bf n}_i.
$$
In this case we have to compute $(h{\bf u})_{e,\alpha,i} . {\bf n}_i$
or ${\bf u}_{e,\alpha,i} .  {\bf n}_i$.
We consider an inflow boundary, so $(h{\bf u})_{g,\alpha,i} . {\bf n}_i < 0$ therefore
using the notation (\ref{a1}) we have $a_1< 0$. By analogy with the previous section we denote
$$m = \frac{ {\bf u}_{e,\alpha,i} .  {\bf n}_i}{\sqrt{gh_{g,i}}},$$
then the equation for $m$ is (see (\ref{defm})-(\ref{a11}))
$$
\phi(m)= \sqrt{\frac{2}{g}} \frac{a_1}{h_{g,i}^{3/2}}.
$$
As in the paragraph entitled~{\it Flux given}, the above equation has
 a unique solution $m <2$ for $a_1 < 0$.

\paragraph{Torrential outflow boundary}
In the case of a torrential outflow boundary,
we do not prescribe any condition.
We
assume that the two Riemann invariants are constant along the outgoing
characteristics leading to
\begin{eqnarray*}
& & {\bf u}_{e,\alpha,i} .  {\bf n}_i - 2 \sqrt{gh_{e,i}}= {\bf
  u}_{\alpha,i} .  {\bf n}_i - 2 \sqrt{gh_i},\\
& & {\bf u}_{e,\alpha,i} .  {\bf n}_i + 2 \sqrt{gh_{e,i}} = {\bf
  u}_{\alpha,i} .  {\bf n}_i + 2 \sqrt{gh_i},
\end{eqnarray*}
and we deduce $h_{e,i}=h_i$, ${\bf u}_{e,\alpha,i} .  {\bf n}_i= {\bf u}_{\alpha,i} .  {\bf n}_i$.
We assume that we also have $(h{\bf u})_{e,\alpha,i} . {\bf t}_i = (h{\bf u})_{\alpha,i} . {\bf t}_i$.


\subsubsection{Layer-averaged Navier-Stokes system}

Because of the fractional step we use, the boundary conditions for the
layer-averaged Euler system are, to some extent, independent from the
one used for the rheology terms.

For the resolution of Eq.~\eqref{eq:viscous_terms}, boundary conditions associated with the operator
$$\nabla_{x,y} . (h_\alpha {\bf T}_\alpha^0),$$
have to be specified and usually we prescribe homogeneous Neumann boundary
conditions (corresponding to an imposed stress). Of course, in particular cases, Dirichlet or Robin type
boundary conditions can also be considered.

\subsection{Toward second order schemes}

In order to improve the accuracy of the results the first-order scheme defined in paragraphs~\ref{subsec:fv}-\ref{subsec:macro_scheme} can be extended to a formally second-order one using a MUSCL like extension (see
\cite{VL}).

\subsubsection{Second order reconstruction for the layer-averaged
  Euler system}

In the definition of the flux~\eqref{eq:flux_HR},
we replace the piecewise constant values
${\bf U}_{i,j},{\bf U}_{j,i}$ by more accurate reconstructions deduced from
piecewise linear approximations, namely the values $\tilde{\bf U}_{i,j},\tilde{\bf
U}_{j,i}$ reconstructed on both sides of the interface. The
reconstruction procedure
is similar to the one used and described
in~\cite[paragraph.~5.1]{bristeau}.

The second order reconstruction is only applied for the horizontal
fluxes. For the exchange terms along the vertical axis involving the
quantities $G_{\alpha\pm 1/2}$, we keep the first order
approximation. Despite this, we recover over the simulations (see paragraphs~\ref{subsec:sta_sol_anal},~\ref{subsec:thacker}) a second order type convergence curve. For this reason, we call this reconstruction "second order".

\subsubsection{Modified Heun scheme}

The explicit time scheme~\eqref{eq:nsml1_d}-\eqref{eq:nsml2_d} used in the previous paragraphs corresponds to a first order explicit Euler
scheme. The second-order accuracy in time is usually recovered by the Heun
method~\cite{heun} that is a slight modification of the second order
Runge-Kutta method. More precisely, for a dynamical system written
under the form
\begin{equation}
\frac{\partial y}{\partial t} = f(y),
\label{eq:sysd}
\end{equation}
the Heun scheme consists in defining $y^{n+1}$ by
\begin{equation}
y^{n+1} = y(t^n+\Delta t^n) = \frac{y^n+\tilde{y}^{n+2}}{2},
\label{eq:heun}
\end{equation}
with
\begin{equation}
\tilde{y}^{n+1} = y^n + \Delta t^n f(y^n,t^n),\quad
\tilde{y}^{n+2} = \tilde{y}^{n+1} + \Delta t^n f(\tilde{y}^{n+1},t^{n+1}).\label{eq:2nd_order}
\end{equation}
But the scheme defined by \eqref{eq:heun} does not preserve the
invariant domains. Indeed, the time step being given by a CFL condition, $\Delta t^n$ in the relation
(\ref{eq:2nd_order}) should be replaced by $\tilde{\Delta
  t}^{n+1}$ i.e. the time step satisfying the CFL condition and calculated using $\tilde{y}^{n+1}$. Thus in situations where
the time step strongly varies from one iteration to another, the Heun
scheme does not preserve the positivity of the scheme.

To overcome this difficulty, we propose an improvement of the Heun
scheme
\begin{proposition}
The scheme defined by $y^{n+1} =  (1 - \gamma) y^n
+ \gamma \tilde{y}^{n+2}$ with
$$\tilde{y}^{n+1} =  y^n + \Delta t_1^n f(y^n),\qquad
\tilde{y}^{n+2} = \tilde{y}^{n+1} + \Delta t_2^{n}
f(\tilde{y}^{n+1}),$$
and
$$\Delta t^n = \frac{2\Delta t_1^n\Delta t _2^n}{\Delta t_1^n + \Delta t_2^n},\quad
\gamma = \frac{(\Delta t^n)^2}{2 \Delta t_1^n \Delta t_2^n},$$
is second order and compatible with a CFL constraint. Since
$\gamma \geq 0$, $y^{n+1}$ is a convex combination of  $y^n$
and $\tilde{y}^{n+2}$ so the scheme preserves the positivity. For the previous relations $\Delta t_1^{n}$ and $\Delta t_2^{n}$ respectively satisfy the
CFL conditions associated with $y^n$ and $\tilde{y}^{n+1}$.
\label{prop:2nd_order}
\end{proposition}
When $\Delta t_1^n = \Delta t_2^n = \Delta t^n$, the scheme reduces to
the classical Heun scheme with $\alpha=\gamma=1/2$.

\begin{proof}[Proof of proposition~\ref{prop:2nd_order}]
Using \eqref{eq:sysd}, a Taylor expansion of $ y(t^n + \Delta
t^n)$ gives
$$ y(t^n+\Delta t^n) = y^n + \Delta t^n f(y^n) + \frac{(\Delta
  t^n)^2}{2}f(y^n)f'(y^n) + {\cal O}((\Delta t^n)^3).$$
Using the definitions given in the proposition,  we have
$$\tilde{y}^{n+2} = y^n + \left( \Delta t_1^n + \Delta t^n_2
\right) f(y^n) + \Delta t_1^n \Delta t_2^n f(y^n)f'(y^n) + {\cal
  O}(\Delta t_2^n (\Delta t_1^n)^2),$$
and a simple calculus gives $\alpha y^n+ \beta \tilde{y}^{n+1}+ \gamma \tilde{y}^{n+2} -
y(t^n+\Delta t^n) = {\cal O}((\Delta t^n)^3)$,
that completes the  proof.
\end{proof}

\section{Numerical applications}
\label{sec:NumRes}

In this section, we use the numerical scheme to simulate several test cases:
analytical solutions or in situ measurements, stationary or
non-stationary solutions, for the Euler and Navier-Stokes systems. The
obtained results emphasize the accuracy of the numerical procedure in
a wide range of typical applications and its applicability to a real
tsunami case.

The numerical simulations presented in this section have been obtained
with the code Freshkiss3d~\cite{freshkiss3d} where the numerical scheme
presented in this paper is implemented.

\subsection{Stationary analytical solution}
\label{subsec:sta_sol_anal}

First, we compare our numerical model with stationary analytical
solutions for the free surface Euler system proposed by some of the authors
in~\cite{JSM_ANALYTIC}.

We consider as geometrical domain a channel $(x,y)\in
[0,x_{max}]\times [0,2]$. The
analytical solution given in~\cite[Prop.~3.1]{JSM_ANALYTIC} and
defined by
\begin{eqnarray}
z_b & = & \overline{z}_b-h_0-\frac{\alpha^2\beta^2}{2g\sin^2(\beta h_0)},
\label{eq:zb_anal}\\
u_{\alpha,\beta} & = & \frac{\alpha\beta}{\sin(\beta h_0)}\cos(\beta(z-z_b)),
\label{eq:udef}\\
v_{\alpha,\beta} & = & 0,\nonumber\\
w_{\alpha,\beta} & = & \alpha\beta \left(\frac{\partial z_b}{\partial
    x}\frac{\cos(\beta(z-z_b))}{\sin(\beta h_0)} + \frac{\partial h_0}{\partial
    x}\frac{\sin(\beta(z-z_b))\cos(\beta h_0)}{\sin^2(\beta
    h)}\right),\nonumber
\end{eqnarray}
with $\alpha=1$ m$^2$.s$^{-1}$, $\beta=1$ m$^{-1}$, $\overline{z}_b=cst$, $x_{max} = 20$ m and
\begin{equation}
h_0(x,y) =
\frac{1}{2}+\frac{3}{2}\frac{1}{1+\left(x-\frac{1}{2}x_{max}\right)^2}-\frac{1}{2}\
\frac{1}{2+\left(x-\frac{2}{3}x_{max}\right)^2},
\label{eq:H0}
\end{equation}
is a stationary regular analytical solution of the incompressible and
hydrostatic Euler system with free
surface~\eqref{eq:div}-\eqref{eq:u},\eqref{eq:bottom},\eqref{eq:free_surf}
with $p^a=0$.

In order to obtain the simulated solution, we consider the
topography defined by~\eqref{eq:zb_anal},\eqref{eq:H0} and we impose the following boundary
conditions
\begin{itemize}
\item[$\circ$] solid wall for the two boundaries $y=0\!$ m and $y=2\!$ m,
\item[$\circ$] given water depth $h_0(x_{max},y)$ at $x=x_{max} =
  20\!$ m,
\item[$\circ$] given flux defined by~\eqref{eq:udef} at $x = 0\!$ m.
\end{itemize}
We have performed the simulations for several unstructured meshes having
290 nodes and 2 layers, 597 nodes and 4 layers, 1010 nodes and 8
layers, 2112 nodes and 17 layers, see Remark~\ref{rem:polyhedron}.

\begin{remark}
In each case where a convergence curve towards an analytical solution
is presented, we have proceeded as follows. First, we choose a sequence
of unstructured meshes for the considered horizontal geometrical
domain. Then the number of layers is adapted so that each
3d element of the mesh can be approximatively considered as a regular
polyhedron.
\label{rem:polyhedron}
\end{remark}

On Fig.~\ref{fig:euler_sta}{\it (a)}, we have depicted the features of the
analytical solution we use for the convergence test, it clearly
appears on Fig.~\ref{fig:euler_sta}{\it (a)} that the velocity profile of
chosen analytical solution varies along the $z$ axis. Figure~\ref{fig:euler_sta}-{\it (b)} gives the convergence curve
towards the analytical solution i.e. the $\mbox{log}(L^2-error)$ of the
water depth~-- at time $T=300$ seconds when the stationary regime is reached~-- versus
$\mbox{log}(h_{a_0}/h_a)$ for the first and second-order schemes and they are compared to the theoretical order (we denote by $h_a$ the average edge length
and $h_{a_0}$ the average edge length of the coarser mesh).


\begin{figure}[htbp]
\begin{center}
\begin{tabular}{cc}
\includegraphics[height=4cm]{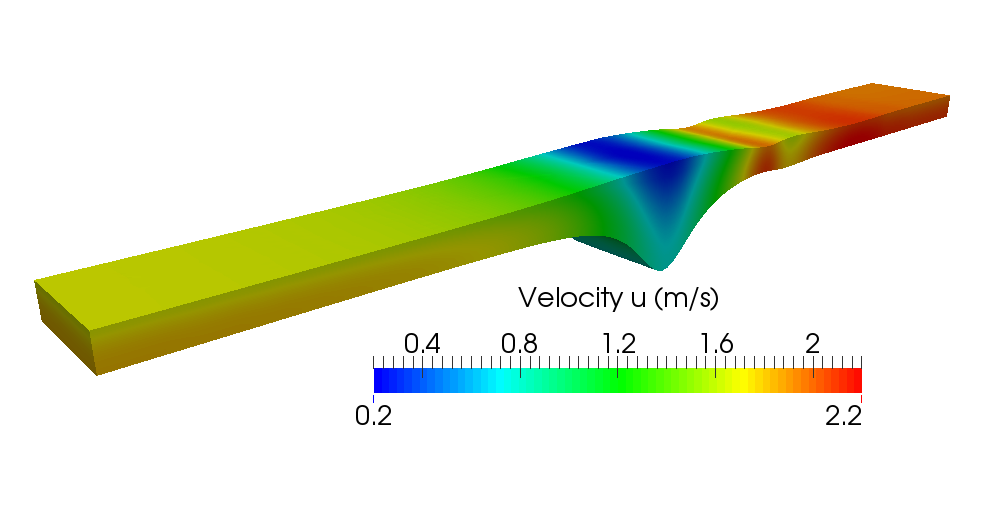}
&
\includegraphics[height=5cm]{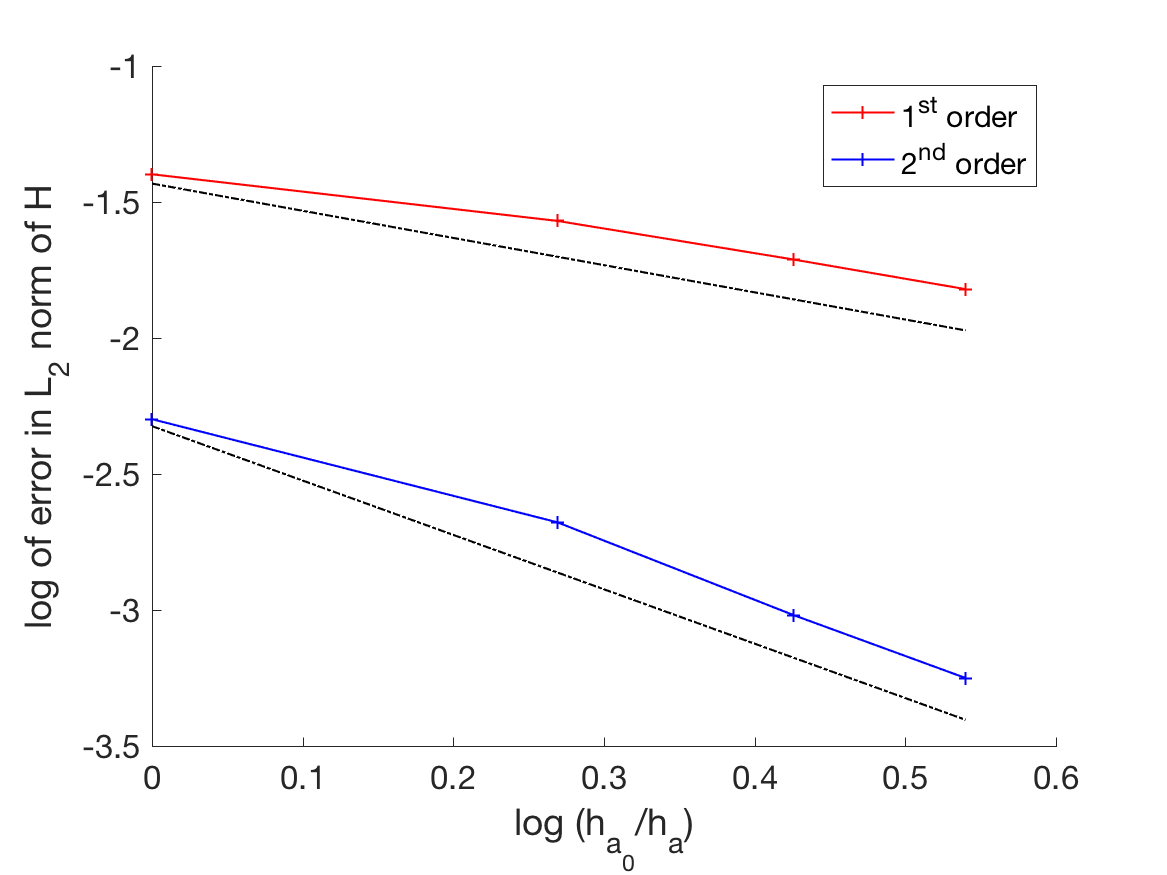}\\
{\it (a)} & {\it (b)}
\end{tabular}
\caption{{\it (a)}  Surface level of the analytical
  solution~\eqref{eq:zb_anal}-\eqref{eq:H0} and horizontal velocity $u_{\alpha,\beta}$, {\it (b)}  error between
 the analytical solution and the simulated one with the six
 meshes, first order  (space and time) and second order extension
 (space and time) schemes. The first and second order theoretical
 curves correspond to the dashed lines.}
\label{fig:euler_sta}
\end{center}
\end{figure}

\begin{remark}
Following the results given
in~\cite[paragraph~3.4]{JSM_ANALYTIC}, it is possible to obtain
stationary analytical solutions with
discontinuities for the Euler system. In this case, the unknowns are
not given by algebraic expressions but are
obtained through the resolution of an ODE involving only the water
depth $h$.

The numerical scheme has been used in the context of such a discontinuous
analytical solution. As planned, for the first and second order
schemes, we recover a first order convergence of the simulated
solution towards the analytical one because of the discontinuity of
the reference solution.
\end{remark}

\subsection{Non-stationary analytical solutions}
\label{subsec:thacker}

In a recent paper~\cite{sol_anal_NS}, some of the authors have
proposed time-dependent 3d analytical
solutions for the Euler and Navier-Stokes equations, some of
them concern hydrostatic models. We confront our numerical scheme
to these situations where analytical solutions are available.

\subsubsection{Radially-symmetrical parabolic bowl}
\label{subsubsec:bowl}

The Thacker' analytical solution~\cite{thacker}, corresponds to a
periodic oscillation in a parabolic bowl. In~\cite{sol_anal_NS} an extension of the Thacker'
radially-symmetrical solution to the situation where the velocity
field depends on the vertical coordinate is proposed.
This means the proposed solution, described hereafter in prop.~\ref{prop:sol_NS2d_draining}, is analytical
for the 3d incompressible hydrostatic Euler system but does not
correspond to a shallow water regime.

\begin{proposition}
For some $t_0\in\mathbb{R}$, $(\alpha,\beta,\gamma)\in\mathbb{R}_{+*}^3$ such
that 
$\gamma< 1$ let us consider
the functions $h,u,v,w,p$ defined for $t\geq t_0$ by
\begin{eqnarray}
& & h(t,x, y) = \max \left\lbrace 0,\frac{1}{r^2}f\left(\frac{r^2}{\gamma \cos(\omega t)-1}\right) \right\rbrace, \label{eq:sol11_hyd_thacker}\\
& & u(t,x,y,z) =x \left(  \beta  \left(
    z-z_b-\frac{h}{2}\right)+\frac{ \omega \gamma \sin(\omega t)}{2(1-\gamma \cos(\omega t))} \right), \label{eq:sol22_hyd_thacker}\\
& & v(t,x,y,z) =  y\left( \beta  \left(
    z-z_b-\frac{h}{2}\right)+\frac{\omega \gamma \sin(\omega t)}{2(1- \gamma \cos(\omega t))} \right), \label{eq:sol22bis_hyd_thacker}\\
& & w(t,x,y,z) = -\frac{\partial}{\partial x}\int_{z_b}^z udz -\frac{\partial}{\partial y}\int_{z_b}^z vdz,  \label{eq:sol33_hyd_thacker1}\\
& & p(t,x,y,z) =  
g(h+z_b-z),\label{eq:sol55_hyd_thacker1}
\end{eqnarray}
with $\omega=\sqrt{4\alpha g}$, $r=\sqrt{x^2 + y^2} $  and with a bottom topography defined by
\begin{equation}
z_b(x,y)=\alpha\frac{r^2}{2},
\label{eq:topo}
\end{equation}
and the function $f$ given by
$$f(z)=-\frac{4g}{\beta^2}+\frac{2}{\beta^2}\sqrt{4g^2+c z+\beta^2 \alpha g (\gamma^2-1)z^2},$$
$c$ being a negative constant such that $ c\le 4 g^2/(\gamma-1)$. 

Then $h,u,v,w,p$ as defined previously
satisfy the 3d hydrostatic Euler
system~\eqref{eq:div}-\eqref{eq:u} completed with~\eqref{eq:bottom},\eqref{eq:free_surf}.
The appropriate
boundary conditions in lateral boundary are also determined by the expressions
of $h,u,v,w$ given above.
\label{prop:sol_NS2d_draining}
\end{proposition}



The geometrical domain is defined by~$(x,y)\in
[-L/2,L/2]^2$ and the chosen
parameters are $\alpha=2$, $\beta=1$, $\gamma=0.3$, $c=-1$,
$L=1$, the considered analytical solution is depicted on Fig;~\ref{fig:thacker2d_xz_slice}. The initial
conditions correspond to~\eqref{eq:sol11_hyd_thacker}-\eqref{eq:sol55_hyd_thacker1} at time $t=t_0=0$ s.

\begin{figure}[htbp]
\begin{center}
\begin{tabular}{cc}
\includegraphics[width=8cm]{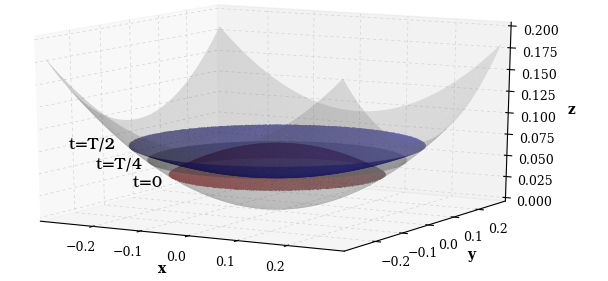}
&
\includegraphics[width=8.5cm]{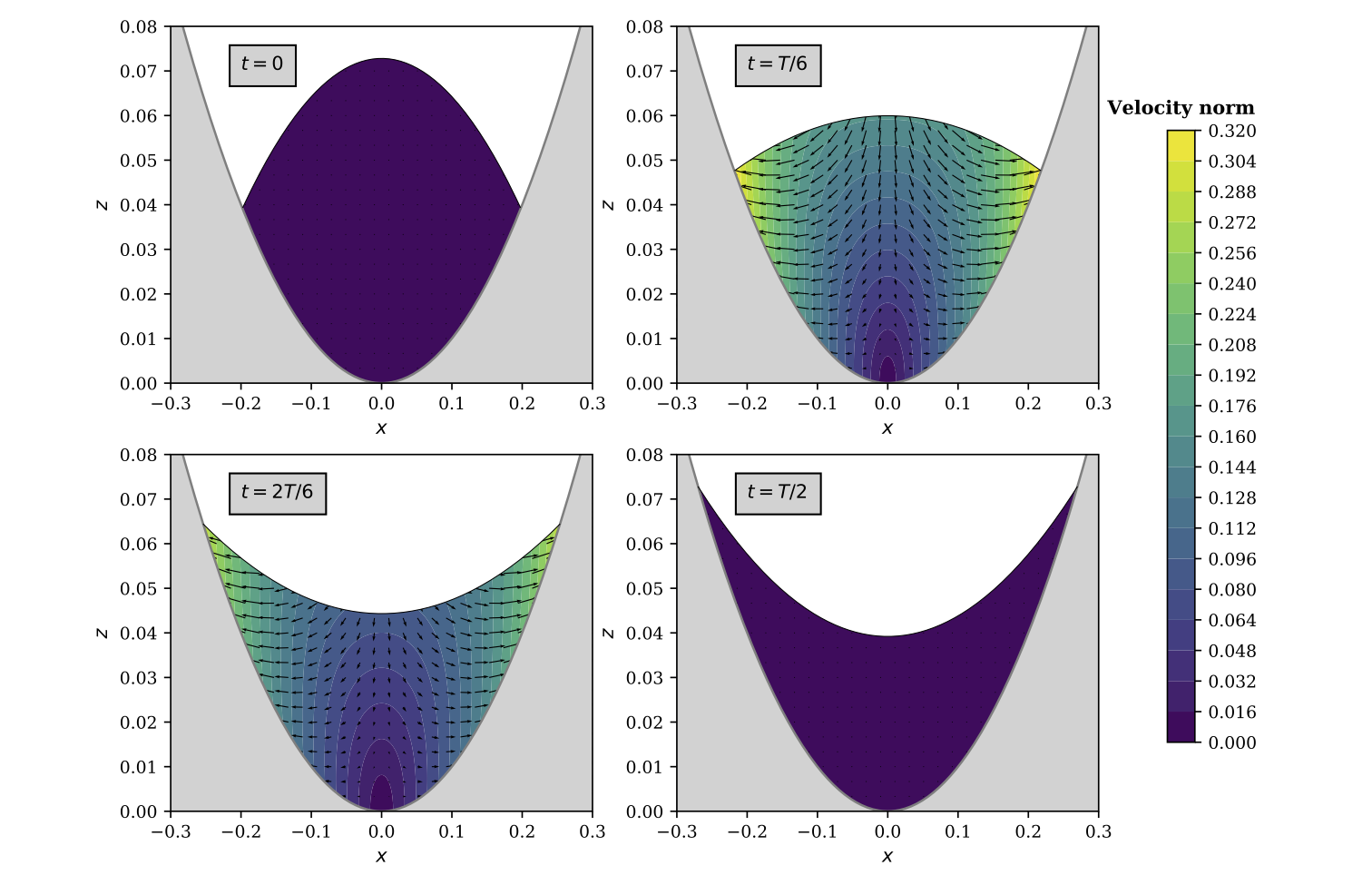}\\
{\it (a)} & {\it (b)}
\end{tabular}
\caption{3D Axisymmetrical parabolic bowl: {\it (a)} free surface at $t=0$
  (red), $t=T/4$ (dark grey), $t=T/2$ (blue), with the period $T$
  defined by $T=2 \pi /\omega$, {\it (b)} velocity norm and vectors at
  $t=0, T/6, 2T/6, T/2$, in $(x,y=0,z)$ slice plane.}
\label{fig:thacker2d_xz_slice}
\end{center}
\end{figure}

In order to evaluate the convergence rate of the simulated solution $h_{sim}$
towards the analytical one $h_{anal}$, we plot the error rate versus the space
discretization for five unstructured meshes with 1273 nodes and a single
layer, 11104 nodes and 6 layers, 30441 nodes and 15 layers, 59473 nodes and
30 layers
and 98137 nodes and 50 layers, see Remark~\ref{rem:polyhedron}. We have plotted (see Fig.~\ref{fig:conv_thacker_ns}) the $\mbox{log}(L^2-error)$ over the
water depth at time $T=2\pi/\omega$ seconds versus
$\mbox{log}(h_{a_0}/h_a)$ for the first and second-order schemes and they
are compared to the theoretical order.

\begin{figure}[hbtp]
\begin{center}
\begin{tabular}{cc}
\includegraphics[width=7cm]{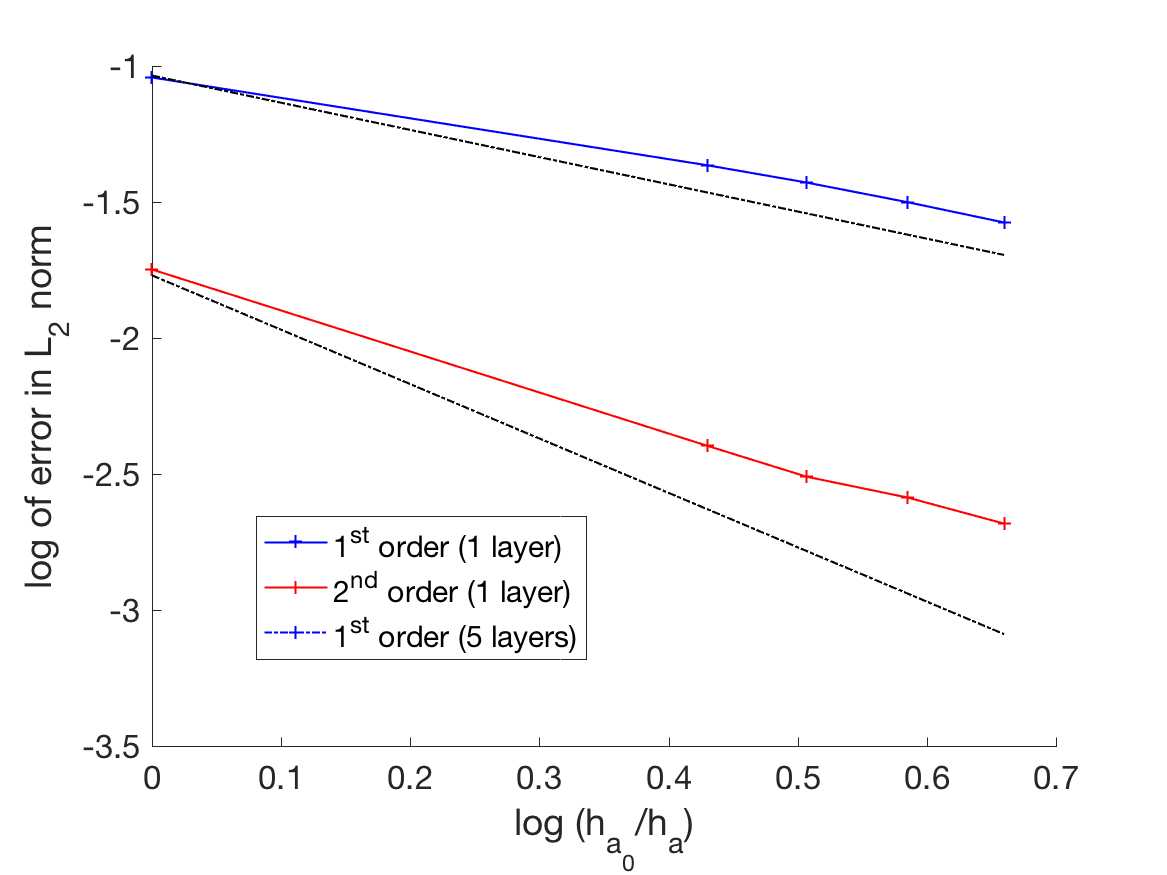}\\
\end{tabular}
\end{center}
\caption{Parabolic bowl:
error between
 the analytical water depth and the simulated one with the five
 unstructured meshes. The curves for the first order scheme (space and
 time) and its second order extension (space and time) are compared to
 the first and second order theoretical curves (dashed lines).}
\label{fig:conv_thacker_ns}
\end{figure}
Notice that on Fig.~\ref{fig:conv_thacker_ns}-{\it (b)}, we have plotted the
simulated solution obtained with 1 layer~-- that corresponds to the
classical Saint-Venant system~-- and the simulated solution with 5
layers. The obtained convergence curves coincide proving the
stability of the numerical scheme for the layers averaged system. The
analytical solution is non stationary and hence, the errors due to the
time scheme are combined with the one induced by the space
discretization and it is the reason why we do not recover the
theoretical order of convergence.

\subsubsection{Draining of a tank}
\label{subsec:sol_anal_NS}

Considering the Navier-Stokes
system~\eqref{eq:incomp}-\eqref{eq:mvthm1} completed with the boundary
conditions~\eqref{eq:free_surf}-\eqref{eq:slip}, the following
proposition holds, see~\cite{sol_anal_NS} for more details about the
proposed analytical solution.

\begin{proposition}
For some $t_0\in\mathbb{R}$, $t_1\in\mathbb{R}_+^*$, $(\alpha,\beta)\in\mathbb{R}_+^2$ such
that $\alpha\beta >L$, let us consider
the functions $h,u,v,w,p,\phi$ defined for $t\geq t_0$ by
\begin{eqnarray*}
& & h(t,x, y) = \alpha f(t), \label{eq:sol11_hyd}\\
& & u(t,x,y,z) = \beta\left(
    (z-z_b)-\frac{\alpha}{2} f(t) \right)+ f(t) (x\cos^2(\theta) +y \sin^2(\theta) ), \label{eq:sol22_hyd}\\
& & v(t,x,y,z) = \beta\left( (z-z_b)-\frac{\alpha}{2}f(t)\right)+ f(t)(x \cos^2(\theta) + y\sin^2(\theta) ), \label{eq:sol22bis_hyd}\\
& & w(t,x,y,z) = f(t)(z_b-z), \label{eq:sol33_hyd}\\
& & p(t,x,y,z) =  p^a(t,x,y) - 2 \nu f(t) +
g(h-(z-z_b)),\label{eq:sol55_hyd}
\end{eqnarray*}
where $f(t)=1/(t-t_0+t_1)$ and with a flat bottom
$z_b(x,y)=z_{b,0}=cst$ and $p^a(t,x,y)=p^{a,1}(t)$,
with $p^{a,1}(t)$ a given function. 

Then $h,u,v,w,p$ as defined previously
satisfy the 3d hydrostatic Navier-Stokes
system~\eqref{eq:incomp}-\eqref{eq:mvthm1}
completed with the boundary 
conditions~\eqref{eq:uboundihmf},\eqref{eq:bottom},\eqref{eq:uboundihm},\eqref{eq:free_surf} and
$\kappa = \frac{2 \nu \alpha \beta}{h(t,x,y) [\alpha \beta -2 (x\cos^2(\theta) + y\sin^2(\theta))] }$
in~\eqref{eq:uboundihmf}, $W=\nu\beta/2$ with ${\bf t}_s =\frac{1}{\sqrt{2}} (1,1,0)^t$ in~\eqref{eq:uboundihm}. The appropriate
boundary conditions for  $x\in \{-L/2, L/2\}$ or $y\in \{-L/2, L/2\}$ are also determined by the expressions
of $h,v,u,w$ given above.

Choosing the viscosity $\nu=0$, the variables $h,u,v,w,p$ become
analytical solutions of the 3d hydrostatic Euler
system~\eqref{eq:div}-\eqref{eq:u} completed
with the boundary conditions~\eqref{eq:free_surf},\eqref{eq:bottom}
and $p(t,x,y,\eta(t,x,y))=0$.
\label{prop:sol_NS3d}
\end{proposition}

\begin{proof}[Proof of prop.~\ref{prop:sol_NS3d}]
The proof of prop.~\ref{prop:sol_NS3d}
relies on very simple computations and is not detailed here.
\end{proof}

We have performed the simulations for several unstructured meshes of
the geometrical domain $(x,y)\in
[0,5]\times [0,1]$ and an adapted number of layers so that each
3d element of the mesh can be approximatively considered as a regular
polyhedron, the considered meshes have 483 nodes and 3 layers, 700
nodes and 6 layers, 1306 nodes and 10 layers, 2781 nodes and 20 layers.

For $L=2$ m, $\alpha=1$ m.s, $t_0=0$ s, $t_1=0.5$ s, $\beta=2.5$
s$^{-1}$, $\theta=0$, $\nu=0$ m$^2$.s$^{-1}$, $p^{a,1}=0$
m$^2$.s$^{-2}$ on Fig.~\ref{fig:NS_hyd}-{\it (a)}, we have depicted the features of the
analytical solution~-- at time $T=0.5$ second~-- we use for the convergence test. Figure~\ref{fig:NS_hyd}-{\it (b)} gives the convergence curve
towards the analytical solution i.e. the $\mbox{log}(L^2-error)$ of the
water depth~-- at time $T=1$ second~-- versus
$\mbox{log}(h_{a_0}/h_a)$ for the first and second-order schemes and they
are compared to the theoretical order (we denote by $h_a$ the average edge length
and $h_{a_0}$ the average edge length of the coarser mesh). Notice that in this test case,
the errors due to the space and time discretization are combined, this
explains why the theoretical orders of convergence are not exactly
obtained. Moreover, the boundary conditions (inflow prescribed) play
an important role and since their numerical treatment is only at the
first order in space, this also explains the difference between the
theoretical and observed orders of convergence. With the mesh having 2781 nodes, we have tested the influence of the number of layers, see Fig.~\ref{fig:NS_hyd}-{\it (c)}. When the numbers of layers increase, we recover the analytical velocity profile.

\begin{figure}[hbtp]
\begin{center}
\begin{tabular}{ccc}
\includegraphics[width=0.4\textwidth]{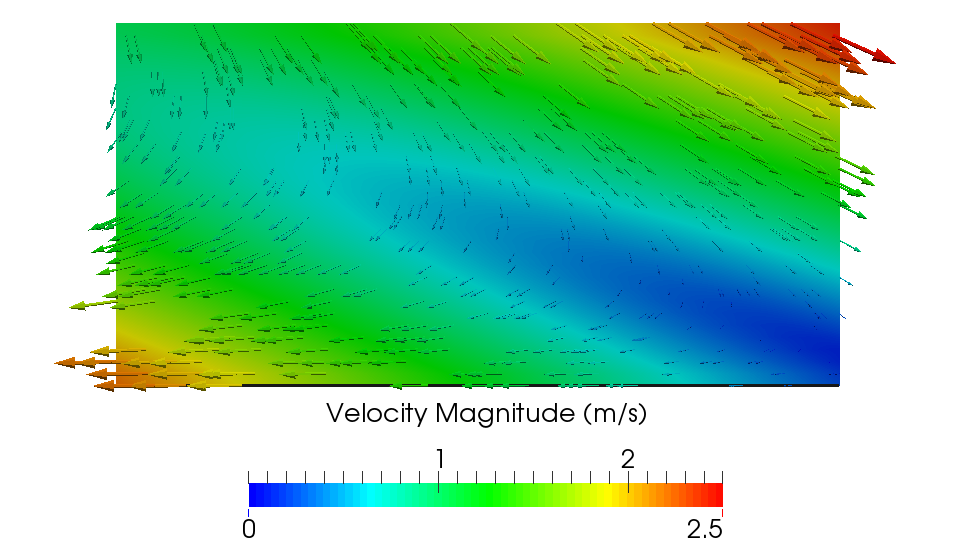}
&
\includegraphics[width=0.28\textwidth]{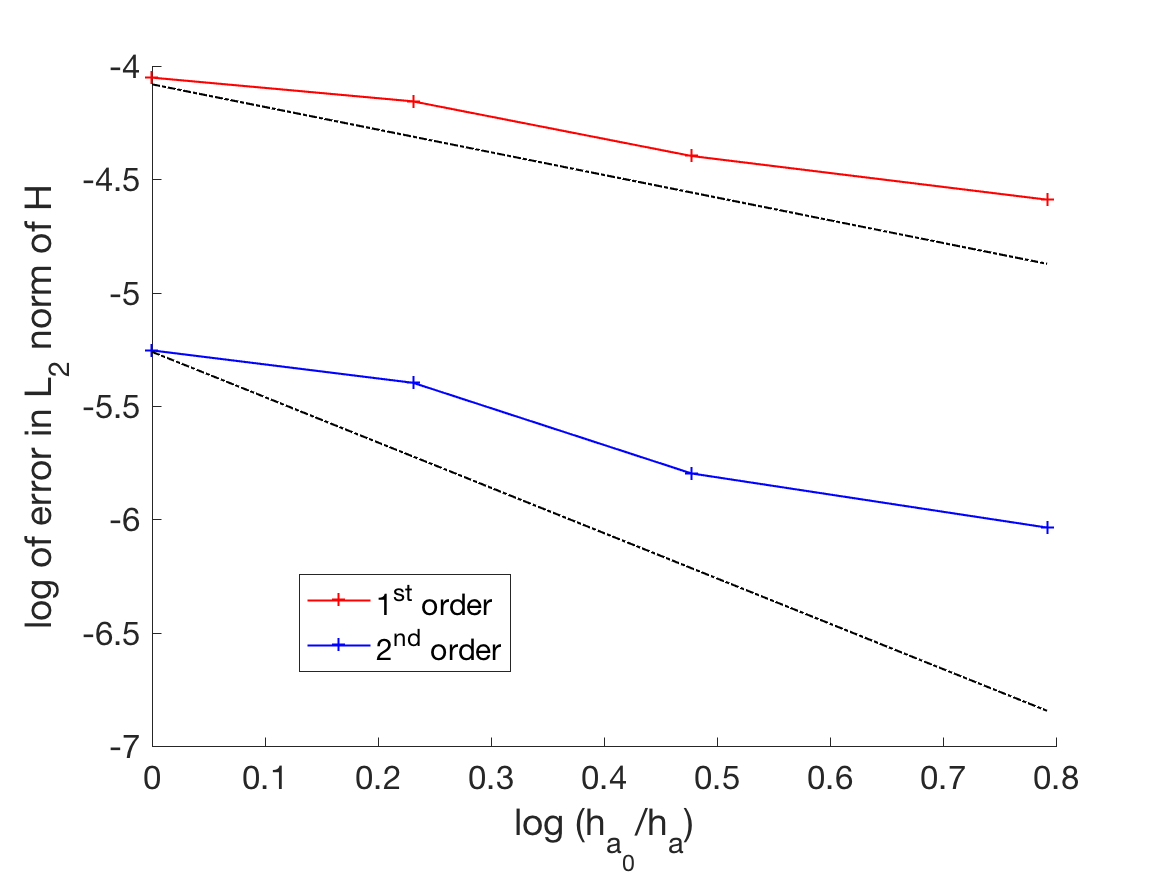}
&
\includegraphics[width=0.28\textwidth]{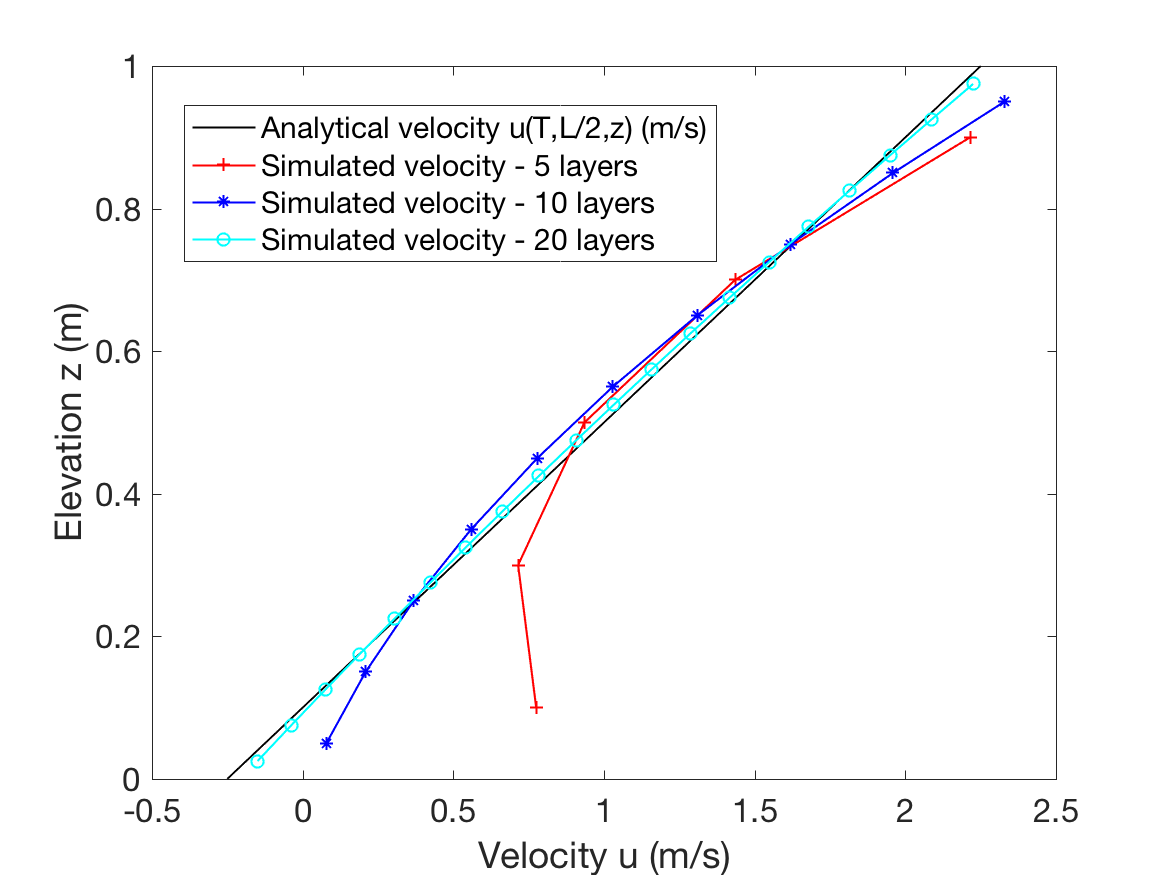}\\
{\it (a)} & {\it (b)} & {\it (c)}\\
\end{tabular}
\caption{Analytical solution given in
prop.~\ref{prop:sol_NS3d}~: {\it (a)} slice of the fluid domain for
$y=0.5$ m and velocity
field at time $T=0.5\!$ s, {\it (b)} convergence curve towards the
reference solution, first order (space and time) and second order extension
 (space and time) schemes. The first and second order theoretical
 curves correspond to the dashed lines. {\it (c)} Analytical horizontal velocity $u$ along $z$ at abscissa $x=L/2$ m and time $T$ and its simulated values for different numbers of layers.}
\label{fig:NS_hyd}
\end{center}
\end{figure}

\subsection{Simulation of a tsunami}
\label{subsec:tsunami}

In this section, we test our discrete model in the case of a real
tsunami propagation for which field measurements are available (free
surface variations recorded by buoys). Even if in such cases, involving long wave propagation, 2d shallow water models can be used instead of 3d description, and we test here the capacity of our model and of the numerical procedure to handle this complex situation.

Simulation of tsunami waves generated by earthquakes is very important in Earth science for hazard assessment and for recovering earthquakes characteristics. Indeed, tsunami waves can be analyzed to recover the earthquake source that generated the tsunami and are now classically used in joint inversion methods. It has been shown that tsunami waves provide strong constraints on the spatial distribution of the source, especially in the case of shallow slip~\cite{gusman}. In some cases, far-field tsunami gauges may help constrain the earthquake source process even though they are affected by the compressibility of the water column and of the Earth~\cite{gusman,watada}.

The 2014/04/01 Iquique earthquake struck off the coast of Chile at 20:46 local time (23:46 UTC), with a
moment magnitude of 8.1. The
epicenter of the earthquake was approximately 95 kilometers (59 mi)
northwest of Iquique, as shown in Fig.~\ref{fig:map_chile}.

We have carried out simulations of the tsunami induced by the
earthquake using
\begin{itemize}
\item[$\bullet$] a topography obtained from the National Oceanic and
  Atmospheric Administration (NOAA, \cite{noaa}) using the ETOPO1
  data (1-arc minute global relief model),
\item[$\bullet$] an unstructured mesh whose dimensions~-- a square of
  2224.2 km$^2$~-- correspond to
  the domain covered by Fig.~\ref{fig:map_chile},
\item[$\bullet$] a source corresponding to the seafloor displacement induced by the earthquake (Fig.~\ref{fig:map_chile}). This source is obtained by computing the 3D final displacements of the seafloor generated by the earthquake coseismic slip. This coseismic slip has been itself retrieved by inversion of numerous geodetic and seismic data, according to the model determined by \cite{iquiquemodel}. The source is activated at time $t_0$, just after the earthquake occurrence ($t_0$ is here 2014/04/01,23h47mn25s)
\end{itemize}
We did not consider here the Coriolis force, the tides and the ocean currents. The results shown in Fig.~\ref{fig:dart} have been
obtained with a mesh containing 545821 nodes and 5 layers (computation time was 35 minutes with a Mac book air 1.7 GHz Intel core i7).  We compare
the numerical solutions ~-- provided by the first order scheme (space and time)
and the second order scheme (space and time)~-- with the DART
measurements (obtained from the NOAA website http://www.ndbc.noaa.gov/dart.shtml). A series of simulations have been performed using several meshes and we present
``converged'' results in the sense that a finer mesh would give the
same results. This is illustrated in Fig.~\ref{fig:dart}-{\it (d)}, where we plot the simulation results obtained with three meshes having respectively 311687 nodes (coarse mesh), 545821 nodes (fine mesh), and 985327 nodes (very fine mesh): the curves corresponding to the fine (cyan) and very fine (blue) curves are very similar.

Fig.~\ref{fig:dart}-{\it (a)},{\it (b)},{\it (c)} shows that the second order scheme significantly improves the results both for the amplitude and phase of the water waves.
The second order scheme is able to very accurately reproduce the shape of the first wave at the closest DART buoy 32401, located at 287 km from the epicenter. The two following peaks in the waveform are quite well reproduced up to about $1.755 \times 10^5$ s. This is also the case at the DART buoy 32402, located 853 km from the epicenter. The arrival time of the first wave is very well reproduced at the three DART buoys, slightly better with the second order scheme. At the most distant buoy 32412 (1650 km from the source), the first order scheme is not able to reproduce the recorded wave. The second order scheme reproduces the first wave quite well but not the rest of the waveform, possibly due to Earth curvature effects that are not taken into account here. Globally, the low frequency content of the signals is better explained by the model than the high frequency fluctuations. These high frequency fluctuations may be related to effects not accounted for here, such as spatio-temporal heterogeneity of the real source, small wavelength fluctuation of the topography, and possibly non-hydrostatic effects \cite{JSM_nhyd,nora1}.

\begin{figure}[hbtp]
\begin{center}
\includegraphics[width=2.5cm]{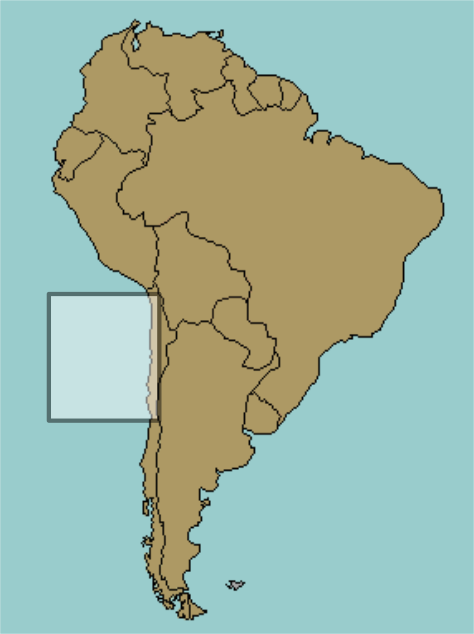}\hspace*{0.5cm}
\includegraphics[width=6cm]{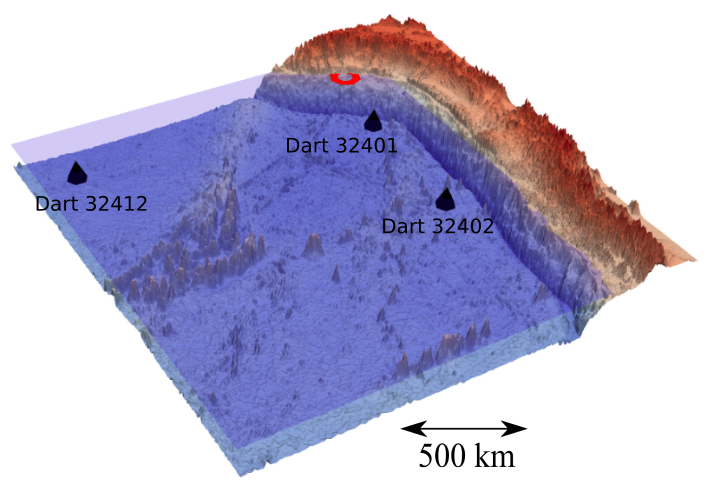}
\hspace*{0.1cm}\includegraphics[width=6cm]{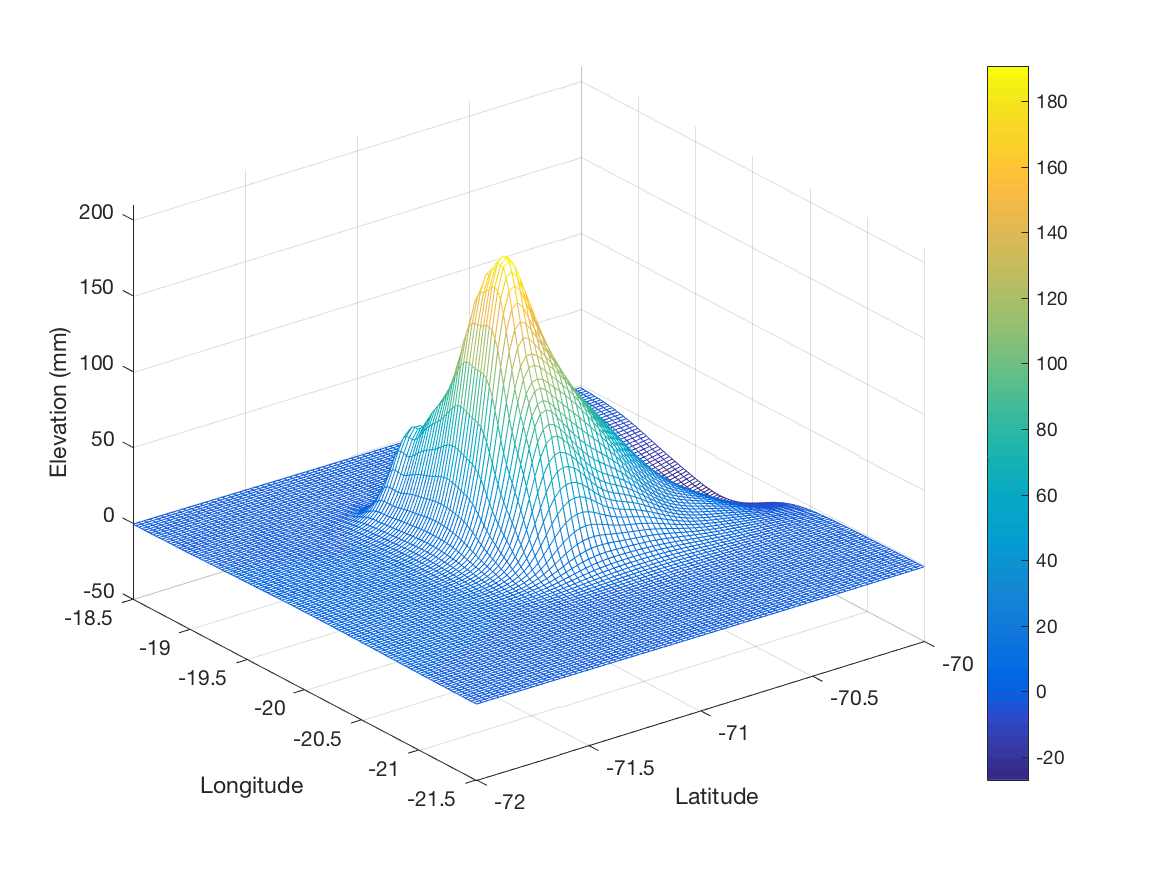}
\end{center}
\caption{(Left) Location of the zone of interest, located offshore Northern Chile ; (Center) Bathymetric map showing the earthquake epicenter (red cone) and the location of the three DART buoys (black boxes); (Right) Vertical displacement (in centimeter) of the topography due to the earthquake. Horizontal displacements (not shown here) are also taken into account in the simulation.}
\label{fig:map_chile}
\end{figure}


\begin{figure}
\begin{center}
\begin{tabular}{cc}
\includegraphics[height=4.8cm,width=7cm]{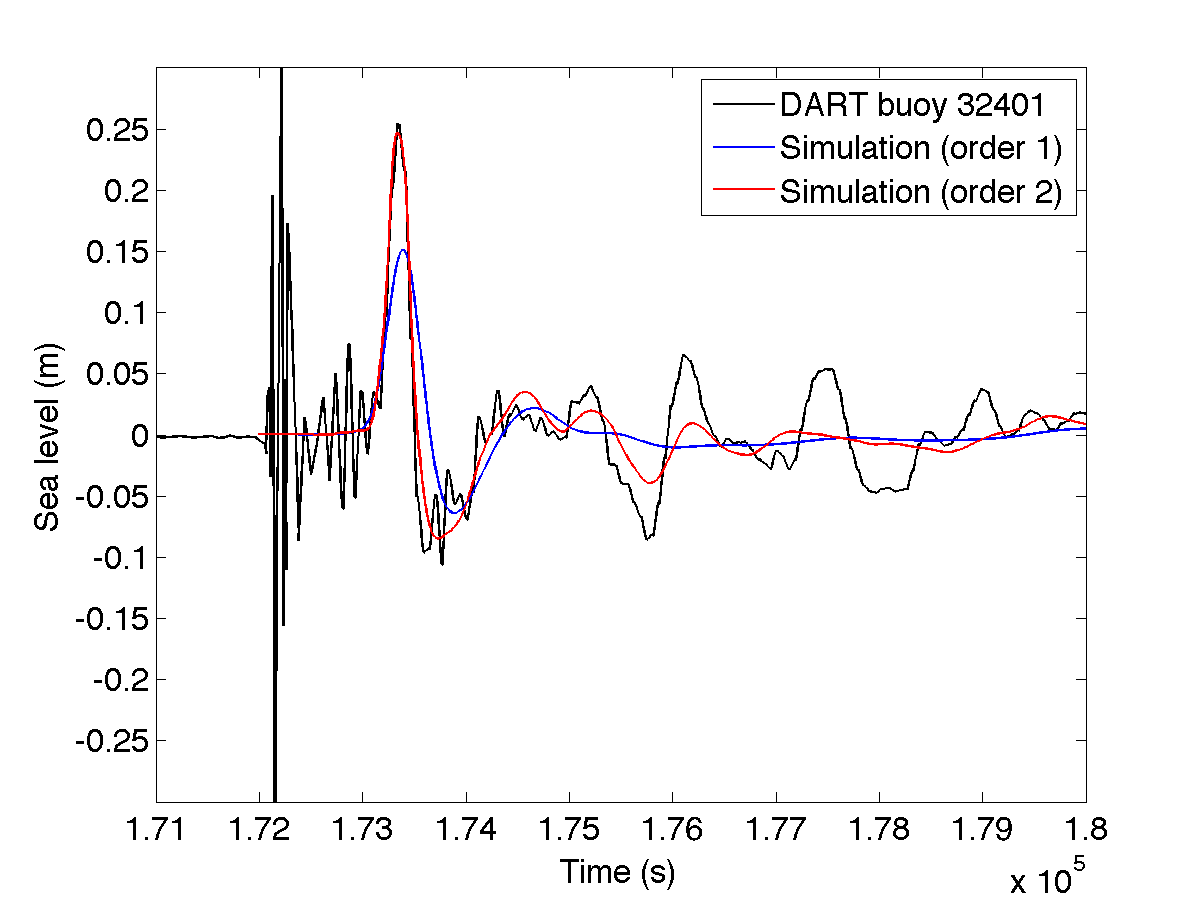}
&
\includegraphics[height=4.8cm,width=7cm]{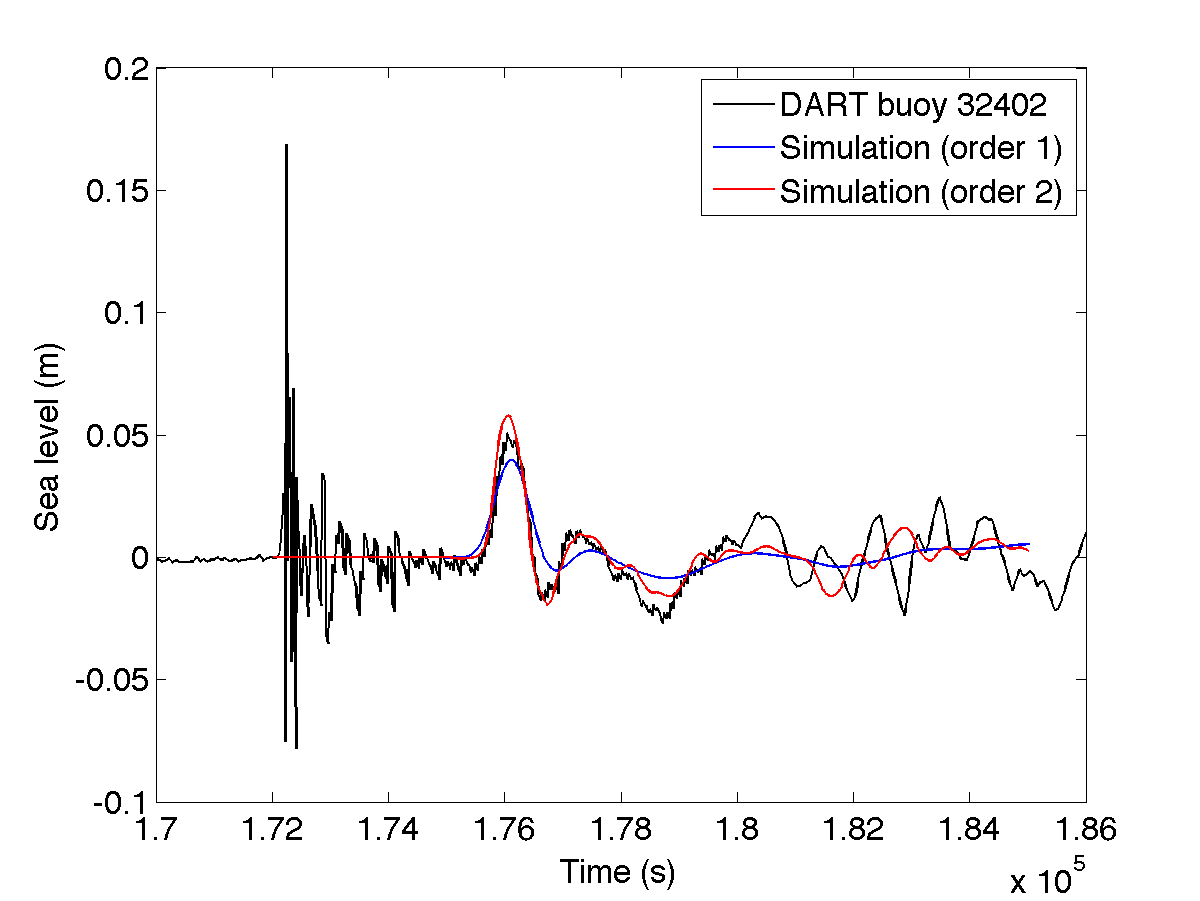}\\
{\it (a)} & {\it (b)}\\
\includegraphics[height=4.8cm,width=7cm]{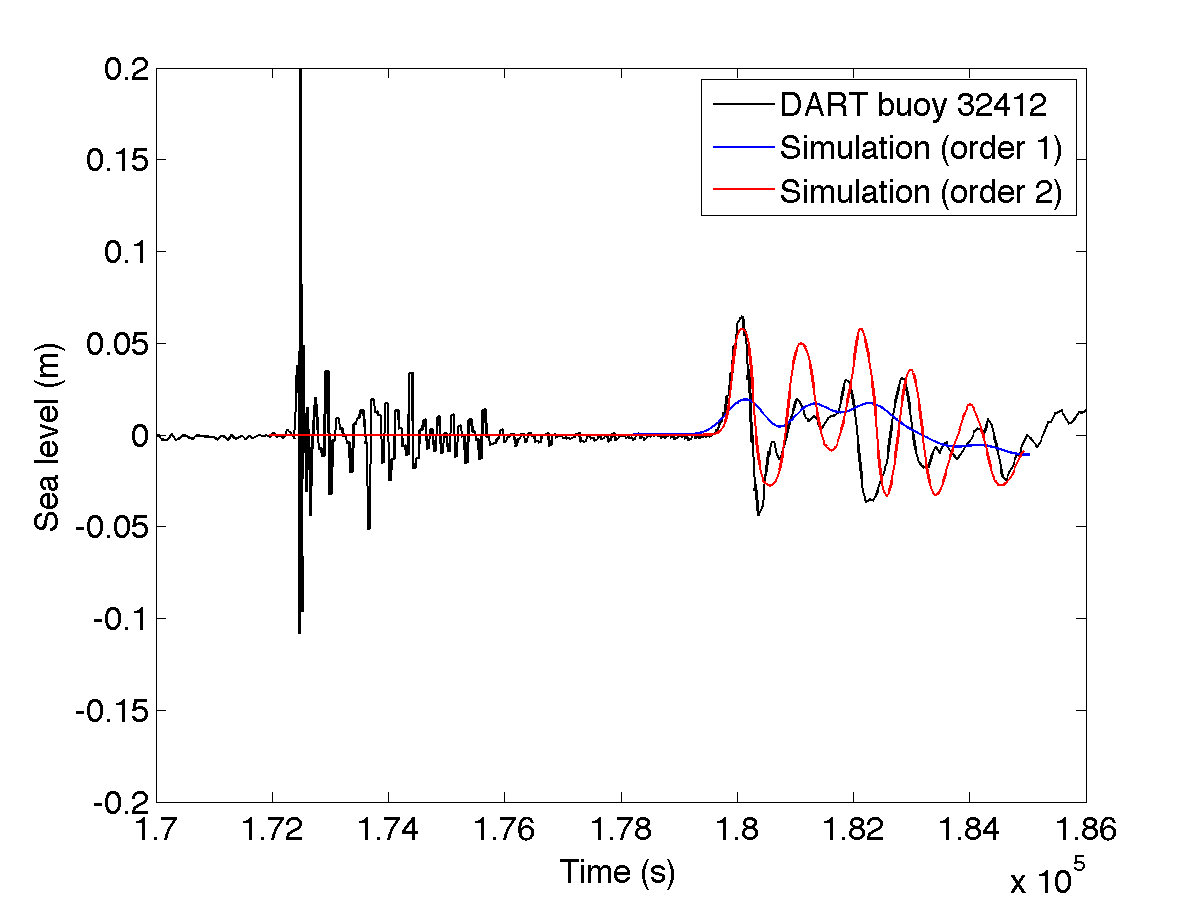}
&
\includegraphics[height=4.8cm,width=7cm]{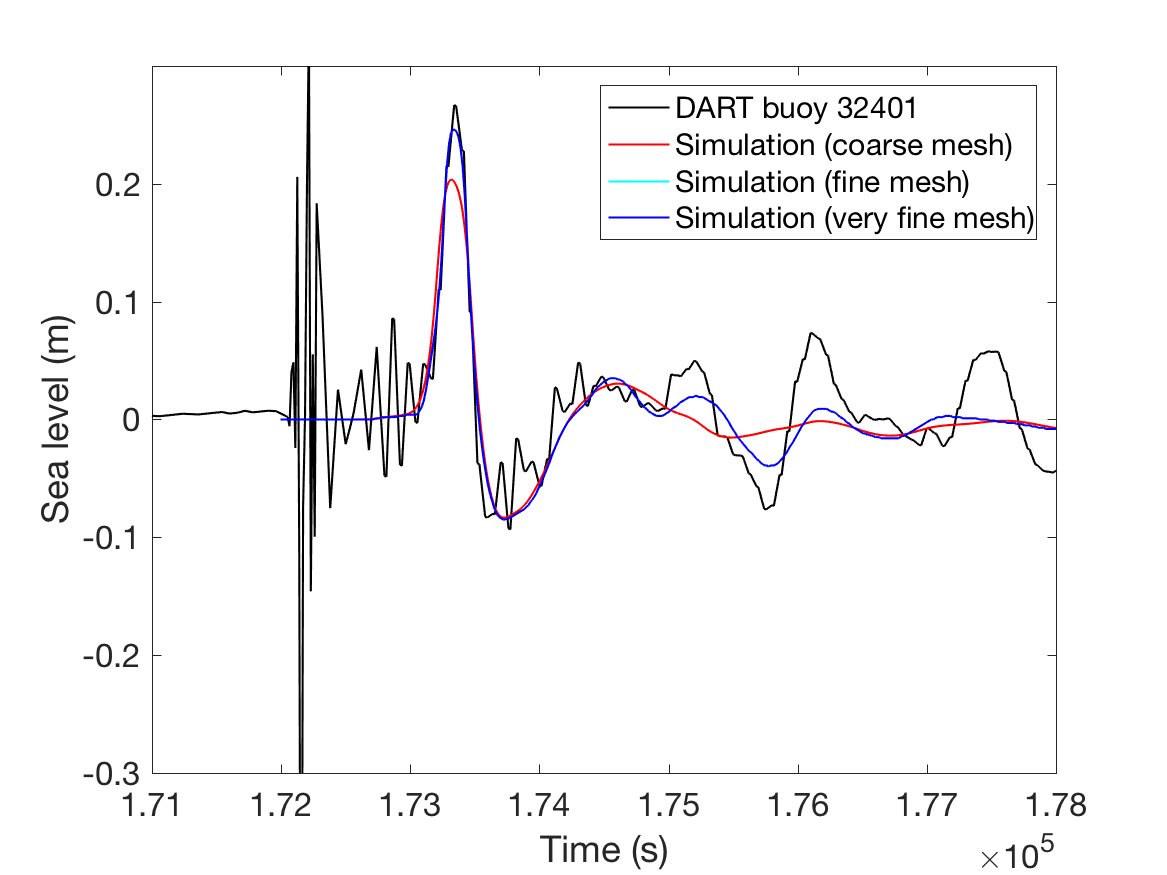}\\
{\it (c)} & {\it (d)}\\
\end{tabular}
\end{center}
\caption{a), b) c) : Comparison between the sea level variations recorded by the
  3 DART buoys and the corresponding simulations (1$^{st}$ and 2$^{nd}$ order 
  schemes). d) Effect of the mesh size on the simulation accuracy for DART buoy 32401. 
  In a), b), c), d), observations have been detided using a low pass 
  Butterworth filter (order 4 and cutoff frequency of 4 hours), 
  and the simulated waveforms have been filtered with the same Butterworth filter.}
\label{fig:dart}
\end{figure}

\section{Conclusion}

In this paper, we have presented a layer-averaged version of the
3d incompressible, hydrostatic Euler and Navier-Stokes systems with free
surface. Based on a kinetic interpretation of the system for the Euler part, we have
derived a stable, robust and efficient numerical scheme in a finite
volume/finite element framework on unstructured meshes. The numerical scheme is endowed with
strong stability properties (domain invariant, well-balancing, wet/dry interfaces treatment,\ldots).

The numerical scheme is successfully validated with analytical solutions and is shown to be applicable to simulate
complex test cases, like a tsunami propagation over a real bathymetry, proving its accuracy and efficiency.

\bigskip

\noindent {\large\bf Acknowledgements} The authors thank Quentin
Bl\'etery for his involvement in the analysis of the 2014 Iquique
earthquake (paragraph~\ref{subsec:tsunami}). We also thank  Rapha\"el
Grandin for its contribution in the computations of the bottom
displacements necessary for the tsunami simulations. This work has
been partially funded by the ERC Contract No. ERC-CG-2013-PE10-617472
SLIDEQUAKES. The authors also acknowledge the Inria Project Lab "Algae in Silico" for its financial support.

\bibliographystyle{vancouver}
\bibliography{boussinesq}

\appendix
\label{app}

\noindent {\large\bf Appendix A}

\begin{proof}[Proof of prop.~\ref{prop:NS_mc}]
Using the results of prop.~\ref{thm:model_ml}, it remains to obtain
the expression for the layer-averaged viscous terms.

The layer-averaging of the viscous terms appearing in~\eqref{eq:mvthm}
gives
\begin{eqnarray}
\int_{z_{\alpha+1/2}}^{z_{\alpha+1/2}}\left( \frac{\partial 
    \Sigma_{xx}}{\partial x} + \frac{\partial  \Sigma_{xy}}{\partial y}  + \nu\frac{\partial^2 u}{\partial z^2}  \right) dz
& = & \frac{\partial}{\partial x}\left( h_\alpha \Sigma_{xx,\alpha}\right) + \frac{\partial}{\partial y}\left(
 h_\alpha  \Sigma_{xy,_\alpha}\right) \nonumber\\
& & +\nu\left.\frac{\partial u}{\partial z}\right|_{\alpha+1/2} - \nu\left.\frac{\partial u}{\partial z}\right|_{\alpha-1/2}\nonumber\\
& & - \frac{\partial z_{\alpha+1/2}}{\partial x}\Sigma_{xx,\alpha+1/2}
- \frac{\partial z_{\alpha+1/2}}{\partial y} \Sigma_{xy,\alpha+1/2} \nonumber\\
& & +
\frac{\partial z_{\alpha-1/2}}{\partial
  x}  \Sigma_{xx,\alpha-1/2} +
\frac{\partial z_{\alpha-1/2}}{\partial y} \Sigma_{xy,\alpha-1/2},
\label{eq:sigma_xx1}
\end{eqnarray}
with
\begin{eqnarray*}
h_\alpha \Sigma_{xx,\alpha} & = &
\int_{z_{\alpha-1/2}}^{z_{\alpha-1/2}} \Sigma_{xx} dz.
\end{eqnarray*}
The boundary conditions at the bottom~\eqref{eq:uboundihmf} and at the
free surface~\eqref{eq:uboundihm} imply that in~\eqref{eq:sigma_xx1}
$$ \nu\left.\frac{\partial u}{\partial z}\right|_{N+1/2}- \frac{\partial \eta}{\partial x}\Sigma_{xx,N+1/2}
- \frac{\partial \eta}{\partial y} \Sigma_{xy,N+1/2} = 0,$$
and
$$\nu\left.\frac{\partial u}{\partial z}\right|_{1/2}- \frac{\partial z_b}{\partial x}\Sigma_{xx,1/2}
- \frac{\partial z_b}{\partial y} \Sigma_{xy,1/2} = \kappa  u_1.$$
The above expression for $\Sigma_{xx,\alpha}$ and relation~\eqref{eq:sigma_xx1} have been obtained using
the following formal computation
\begin{eqnarray*}
\int_{z_{\alpha+1/2}}^{z_{\alpha+1/2}} \frac{\partial 
  \Sigma_{xx}}{\partial x} dz & = & \frac{\partial}{\partial x}
\int_{z_{\alpha+1/2}}^{z_{\alpha+1/2}} \Sigma_{xx} dz - \frac{\partial  z_{\alpha+1/2}}{\partial
    x}  \Sigma_{xx,\alpha+1/2} +
  \frac{\partial  z_{\alpha-1/2}}{\partial x} \Sigma_{xx,\alpha-1/2}\\
& = & \frac{\partial}{\partial x} \left( h_\alpha \Sigma_{xx,\alpha}\right) - \frac{\partial  z_{\alpha+1/2}}{\partial
    x}\Sigma_{xx,\alpha+1/2} +
  \frac{\partial  z_{\alpha-1/2}}{\partial x} \Sigma_{xx,\alpha-1/2}.
\end{eqnarray*}
The definitions~\eqref{eq:sigma_mxx},\eqref{eq:sigma_mxy} are
motivated by the following computation
\begin{eqnarray*}
\frac{h_{\alpha+1} + h_\alpha}{2} \Sigma_{xx,\alpha+1/2} & = & \nu_{\alpha+1/2}\int_{z_{\alpha}}^{z_{\alpha+1}}\frac{\partial u}{\partial x} dz \\
& = & \nu_{\alpha+1/2}\frac{\partial }{\partial x}\int_{z_{\alpha}}^{z_{\alpha+1}}
u dz - \nu_{\alpha+1/2}\frac{\partial  z_{\alpha+1}}{\partial x}
u_{\alpha+1} +\nu_{\alpha+1/2}\frac{\partial
  z_{\alpha}}{\partial x} u_{\alpha} \\
& = & \nu_{\alpha+1/2}\frac{\partial}{\partial x} \left(
  \frac{h_\alpha}{2} u_\alpha + \frac{h_{\alpha+1}}{2} u_{\alpha+1}\right)
- \nu_{\alpha+1/2}\frac{\partial  z_{\alpha+1}}{\partial x}
u_{\alpha+1} + \nu_{\alpha+1/2}\frac{\partial
  z_{\alpha}}{\partial x} u_{\alpha}\nonumber\\
& = & \nu_{\alpha+1/2}\left(
  \frac{h_\alpha}{2} \frac{\partial u_\alpha}{\partial x} + \frac{h_{\alpha+1}}{2} \frac{\partial u_{\alpha+1}}{\partial x} \right) -\nu_{\alpha+1/2}\frac{\partial  z_{\alpha+1/2}}{\partial x}
(u_{\alpha+1} - u_\alpha).
\end{eqnarray*}

In order to prove the energy balance~\eqref{eq:energy_glol_ns} we use
the results of prop.~\ref{prop:energy_bal} obtained for the layer-averaged Euler
system and it remains to consider the viscous and frictions terms
multiplied by ${\bf u}_\alpha$.

Let us define ${\bf R}_\alpha$
\begin{multline*}
{\bf R}_\alpha = \begin{pmatrix} R_{x,\alpha}\\
  R_{y,\alpha} \end{pmatrix} =
\nabla_{x,y} .
\bigl( h_\alpha  {\bf \Sigma}_\alpha \bigr)
 - {\bf \Sigma}_{\alpha+1/2} \nabla_{x,y} z_{\alpha+1/2}+ {\bf
   \Sigma}_{\alpha-1/2} \nabla_{x,y} z_{\alpha-1/2}\nonumber\\
+2\nu_{\alpha+1/2} \frac{{\bf u}_{{\alpha}+1}-{\bf
    u}_{\alpha}}{h_{{\alpha}+1}+h_{\alpha}} - 2\nu_{\alpha-1/2}
\frac{{\bf u}_{\alpha}-{\bf
    u}_{{\alpha}-1}}{h_{\alpha}+h_{{\alpha}-1}} -\kappa_{\alpha}{\bf
  u}_{\alpha}.
\end{multline*}
We write
\begin{eqnarray*}
R_{x,\alpha} u_\alpha & = & \frac{\partial
  \bigl( u_\alpha h_\alpha \Sigma_{xx,\alpha}\bigr)}{\partial x} + \frac{\partial
  \bigl( u_\alpha h_\alpha \Sigma_{xy,\alpha}\bigr)}{\partial y}\\
& & +\nu_{\alpha+1/2}
\frac{u_{{\alpha}+1}+u_{\alpha}}{2}\frac{u_{{\alpha}+1}-u_{\alpha}}{h_{{\alpha}+1}+h_{\alpha}}
-\nu_{\alpha-1/2} \frac{u_{\alpha}+u_{\alpha-1}}{2}\frac{u_{\alpha}-u_{\alpha-1}}{h_{\alpha}+h_{\alpha-1}}\\
& &
 -\frac{\partial z_{\alpha+1/2}}{\partial x}\Sigma_{xx,\alpha+1/2}
 u_\alpha -\frac{\partial z_{\alpha+1/2}}{\partial y}\Sigma_{xy,\alpha+1/2}
 u_\alpha \nonumber\\
& &
 +\frac{\partial z_{\alpha-1/2}}{\partial x}\Sigma_{xx,\alpha-1/2}
 u_\alpha +\frac{\partial z_{\alpha-1/2}}{\partial y}\Sigma_{xy,\alpha-1/2}
 u_\alpha \nonumber\\
& & -h_\alpha \Sigma_{xx,\alpha}\frac{\partial u_\alpha}{\partial x} - h_\alpha
\Sigma_{xy,\alpha}\frac{\partial u_\alpha}{\partial y}\\
& & -\nu_{\alpha+1/2} \frac{(u_{\alpha+1}- u_{\alpha})^2}{h_{{\alpha}+1}+h_{\alpha}} - \nu_{\alpha-1/2}
\frac{(u_{\alpha}- u_{\alpha-1})^2}{h_{\alpha}+h_{{\alpha}-1}} -\kappa_{\alpha}
  u^2_{\alpha},
\end{eqnarray*}
and using the closure relation~\eqref{eq:sigma_m}, it comes for $\alpha=2,\ldots,N-1$
\begin{eqnarray*}
R_{x,\alpha} u_\alpha & = & \frac{\partial
  \bigl( u_\alpha h_\alpha \Sigma_{xx,\alpha}\bigr)}{\partial x} + \frac{\partial
  \bigl( u_\alpha h_\alpha \Sigma_{xy,\alpha}\bigr)}{\partial y}\\
& & +\nu_{\alpha+1/2}
\frac{u_{{\alpha}+1}+u_{\alpha}}{2}\frac{u_{{\alpha}+1}-u_{\alpha}}{h_{{\alpha}+1}+h_{\alpha}}
-\nu_{\alpha+1/2} \frac{u_{\alpha}+u_{\alpha-1}}{2}\frac{u_{\alpha}-u_{\alpha-1}}{h_{\alpha}+h_{\alpha-1}}\\
& & -\Sigma_{xx,\alpha+1/2}\left(  \frac{h_\alpha}{2} \frac{\partial
    u_\alpha}{\partial x} + \frac{\partial z_{\alpha+1/2}}{\partial x}
 u_\alpha \right) -\Sigma_{xy,\alpha+1/2}\left(  \frac{h_\alpha}{2}
 \frac{\partial u_\alpha}{\partial y}  + \frac{\partial z_{\alpha+1/2}}{\partial y}
 u_\alpha \right)\\
& & +\Sigma_{xx,\alpha-1/2}\left(  -\frac{h_\alpha}{2} \frac{\partial
    u_\alpha}{\partial x} + \frac{\partial z_{\alpha-1/2}}{\partial x}
 \frac{u_\alpha}{2} \right) + \Sigma_{xy,\alpha-1/2}\left(  -\frac{h_\alpha}{2} \frac{\partial
    u_\alpha}{\partial y} + \frac{\partial z_{\alpha-1/2}}{\partial y}
 \frac{u_\alpha}{2} \right)\\
& & -\nu_{\alpha+1/2} \frac{(u_{\alpha+1}- u_{\alpha})^2}{h_{{\alpha}+1}+h_{\alpha}} - \nu_{\alpha-1/2}
\frac{(u_{\alpha}- u_{\alpha-1})^2}{h_{\alpha}+h_{{\alpha}-1}} -\kappa_{\alpha}
  u^2_{\alpha}.
\end{eqnarray*}
An analogous relation can be easily obtained for $R_{y,\alpha}
v_\alpha$ and computing the sum over the layers of the obtained
quantity, we get
\begin{eqnarray*}
\sum_{\alpha=1}^N {\bf R}_{\alpha} . {\bf u}_\alpha  & = &
\sum_{\alpha=1}^N \left( \frac{\partial
  \bigl( u_\alpha h_\alpha \Sigma_{xx,\alpha}\bigr)}{\partial x} + \frac{\partial
  \bigl( u_\alpha h_\alpha \Sigma_{xy,\alpha}\bigr)}{\partial y} + \frac{\partial
  \bigl( v_\alpha h_\alpha \Sigma_{yx,\alpha}\bigr)}{\partial x} + \frac{\partial
  \bigl( v_\alpha h_\alpha \Sigma_{yy,\alpha}\bigr)}{\partial y}
\right)\\
& & -\sum_{\alpha=1}^{N-1} \frac{h_{\alpha+1} + h_\alpha}{2\nu}\left(\Sigma_{xx,\alpha+1/2}^2+\Sigma_{xy,\alpha+1/2}^2+\Sigma_{yx,\alpha+1/2}^2+\Sigma_{yy,\alpha+1/2}^2\right)\\
& &  -\sum_{\alpha=1}^{N-1} 2\nu \frac{|{\bf u}_{{\alpha}+1}-{\bf
    u}_{\alpha}|^2}{h_{{\alpha}+1}+h_{\alpha}}  -\kappa |{\bf u}_1|^2,
\end{eqnarray*}
proving the result.
\end{proof}

\noindent {\large\bf Appendix B}

Simple computations give
\begin{eqnarray}
F_h & =  & h \int_{\{y_1 \geq - \frac{\tilde{u}}{c}\} \times \R} (\tilde{u}+cy_1)\chi_0(y_1,y_2)
dy_1 dy_2 \nonumber\\
& = & h \int_{y_1= - \frac{\tilde{u}}{c}}^{+\infty}  (\tilde{u}+cy_1)\left(
\int_{-\infty}^{+\infty} \chi_0(y_1,y_2)
dy_2 \right) dy_1
=  \frac{h}{\pi} \int_{y_1= - \frac{\tilde{u}}{c}}^{+\infty}  (\tilde{u}+cy_1)
\sqrt{1-\frac{y_1^2}{4}} dy_1,\label{eq:flux_h_expr}\\
F_{hu} & =  & h \int_{\{y_1 \geq - \frac{\tilde{u}}{c}\} \times \R} (\tilde{u}+cy_1) (u+cn_xy_1)\chi_0(y_1,y_2)
dy_1 dy_2 \nonumber\\
& = & \frac{h}{\pi} \int_{y_1= - \frac{\tilde{u}}{c}}^{+\infty}  (\tilde{u}+cy_1) (u+cn_xy_1)
\sqrt{1-\frac{y_1^2}{4}} dy_1,\nonumber
\end{eqnarray}
and likewise for $F_{hv}$ we have
\begin{equation*}
\begin{split}
F_{hv}  & = \frac{h}{\pi} \int_{y_1= - \frac{\tilde{u}}{c}}^{+\infty}  (\tilde{u}+cy_1) (v+cn_yy_1)
\sqrt{1-\frac{y_1^2}{4}} dy_1
.
\end{split}
\end{equation*}
It is possible to obtain explicit formula for the expressions of $F_h$,
$F_{hu}$ and $F_{hv}$ since defining
\begin{equation*}
\begin{split}
& I_1(z) = \int^z  (u+cz) \sqrt{1-\frac{z^2}{4}} dz,\qquad I_2(z) = \int^z  (u+cz)(v+cnz) \sqrt{1-\frac{z^2}{4}} dz ,
\end{split}
\end{equation*}
we have
$$ I_1(z) = -\frac{4c}{3} \left( 1-\frac{z^2}{4} \right)^{3/2}+u \left( \frac{z}{2}\sqrt {1-\frac{
z^2}{4}}+\arcsin \left( \frac{z}{2} \right)  \right)
,$$
and
\begin{multline*}
I_2(z) =-c^{2}n z \left( 1 - \frac{z^2}{4} \right)
^{3/2}+\frac{c^{2}n + uv}{2}z\sqrt { 1 - \frac{z^2}{4} }\\
+( c^{2}n + uv
)\arcsin
\left( \frac{z}{2} \right) 
-\frac{4c(nu+v)}{3} \left( 1 - \frac{z^2}{4}
\right) ^{3/2} .
\end{multline*}
Therefore, it comes
\begin{equation*}
\begin{split}
F_h & = \left\{ \begin{array}{ll}
0 & \mbox{if } - \frac{\tilde{u}}{c} \geq 2\\
\frac{h}{\pi} \tilde u
                  \arcsin(\frac{\tilde u}{2c})+\frac{h \tilde
                  u}{2}+\frac{h}{c}\left(\frac{\tilde u^2}{6}+\frac{4}{3}c^2\right)\chi_0\left(\frac{\tilde
                  u}{c}\right) &\mbox{if }-2\leq -\frac{\tilde{u}}{c} \leq 2\\
h & \mbox{if } - \frac{\tilde{u}}{c} \leq -2
\end{array}\right.\\
F_{hu} & = \left\{ \begin{array}{ll}
0 & \mbox{if } - \frac{\tilde{u}}{c} \geq 2\\
\frac{h}{\pi} (c^2n_x+u \tilde u)
                     \arcsin\left(\frac{\tilde u}{2c}\right)+\frac{h}{2} (c^2
                     n_x+u \tilde u) & \\
\quad +\frac{h}{12c} \left(2 u \tilde u^2 - n_x \tilde u^3 +16c^2 u+10c^2 \tilde u n_x\right)
                     \chi_0\left(\frac{\tilde u}{c}\right) & \mbox{if }-2\leq -\frac{\tilde{u}}{c} \leq 2\\
h u\tilde u + hc^2 n_x & \mbox{if } - \frac{\tilde{u}}{c} \leq -2
\end{array}\right.\\
F_{hv} & = \left\{ \begin{array}{ll}
0 & \mbox{if } - \frac{\tilde{u}}{c} \geq 2\\
\frac{h}{\pi} (c^2n_y+v \tilde u)
                     \arcsin\left(\frac{\tilde u}{2c}\right)+\frac{h}{2} (c^2
                     n_y+v \tilde u) & \\
\quad +\frac{h}{12c} \left(2 v \tilde u^2 - n_y \tilde u^3 +16c^2 v+10c^2 \tilde u n_y\right)
                     \chi_0\left(\frac{\tilde u}{c}\right) & \mbox{if }-2\leq -\frac{\tilde{u}}{c} \leq 2\\
h v\tilde u + hc^2 n_y & \mbox{if } - \frac{\tilde{u}}{c} \leq -2
\end{array}\right.\\
\end{split}
\end{equation*}

\end{document}